\documentclass[preprint,12pt,english]{elsarticle}
\usepackage[margin=2.5cm]{geometry}

\usepackage{amsmath}    
\usepackage{amssymb}
\usepackage{graphicx}   
\usepackage{verbatim}   
\usepackage{color}      
\usepackage{subfigure}  
\usepackage{hyperref}   
\usepackage{enumerate}
\usepackage{listings}
\usepackage[dvipsnames]{xcolor}
\usepackage{slashed}
\usepackage{amsthm}
\usepackage{adjustbox}
\usepackage{tikz}
\usepackage{amsmath}    
\usepackage{amssymb}
\usepackage{mathrsfs}
\usepackage{bbold}
\usepackage[british]{babel}
\usepackage{pgfplots}
\usepackage{pgfplotstable}

\usepackage{scalerel}

\usepackage[final]{showlabels}

\usepackage{grffile}
\journal{Elsevier}

\begin{document}

\newcommand{\del}{\partial}
\renewcommand{\theta}{\vartheta}
\renewcommand{\phi}{\varphi}
\newcommand{\vecc}[2]{\left ( \begin{array}{c}#1\\#2\\ \end{array}\right )}
\newcommand{\veccc}[3]{\left ( \begin{array}{c}#1\\#2\\#3\\ \end{array}\right )}
\newcommand{\dd}{\mathrm{d}}
\newcommand{\ee}{\mathrm{e}}
\newcommand{\ii}{\mathrm{i}}
\newcommand{\id}{\mathbb{1}}
\newcommand{\atanh}{\,\text{artanh}\,}
\newcommand{\atan}{\arctan}
\newcommand{\back}{\!\!\!}
\newcommand{\nicefrac}[2]{#1 / #2}
\newcommand{\bboxed}[1]{\text{\textsc{Conjecture}: }\boxed{\boxed{#1}}}
\renewcommand{\and}{\wedge}
\newcommand{\primitive}{\text{\textsc{primitive}}}
\newcommand{\lint}{\int\limits}
\renewcommand{\div}{\mathrm{div\,}}
\renewcommand{\vec}{\mathbf}
\newcommand{\todo}[1]{{\color{red}TODO: #1}}

\newcommand{\new}[1]{{\color{black}#1}}
\newcommand{\RI}[1]{{\color{black}#1}}
\newcommand{\RII}[1]{{\color{black}#1}}

\definecolor{mygreen}{rgb}{0,0.6,0}
\definecolor{mygray}{rgb}{0.5,0.5,0.5}
\definecolor{mymauve}{rgb}{0.58,0,0.82}

\newtheorem{theorem}{Theorem}
\newtheorem{definition}[theorem]{Definition}
\newtheorem{proposition}[theorem]{Proposition}
\newtheorem{lemma}[theorem]{Lemma}
\newtheorem{notation}[theorem]{Notation}
\newtheorem{question}[theorem]{Question}
\newtheorem{remark}[theorem]{Remark}
\newtheorem{corollary}[theorem]{Corollary}
\newtheorem{example}[theorem]{Example}

\newcommand{\splitsymbol}{\id}

\begin{frontmatter}

  \title{Genuinely multi-dimensional stationarity preserving  	\\ Finite Volume formulation for nonlinear hyperbolic PDEs}

\author[lab1]{Wasilij Barsukow}
\author[lab2]{Mirco Ciallella}
\author[lab3]{Mario Ricchiuto}
\author[lab4]{Davide Torlo}

\address[lab1]{Institut de Mathématiques de Bordeaux, Université de Bordeaux, CNRS UMR 5251, Talence, France}
\address[lab2]{Laboratoire Jacques-Louis Lions, Université Paris Cité, CNRS UMR 7598, Paris, France}
\address[lab3]{Centre Inria de l'Université de Bordeaux, CNRS UMR 5251, Talence, France}
\address[lab4]{Dipartimento di Matematica, Università di Roma La Sapienza, Rome, Italy}


\begin{abstract}
Classical Finite Volume methods for multi-dimensional problems include stabilization (e.g.\ via a Riemann solver), that is derived by considering several one-dimensional problems in different directions. Such methods therefore ignore a possibly existing balance of contributions coming from different directions, such as the one characterizing multi-dimensional stationary states. Instead of being preserved, they are usually diffused away by such methods. Stationarity preserving methods use a better suited stabilization term that vanishes at the stationary state, allowing the method to preserve it. 
This work presents a general approach to stationarity preserving Finite Volume methods for nonlinear conservation/balance laws. 
\new{It is based on a multi-dimensional stationarity preserving quadrature strategy that allows to naturally introduce
genuinely multi-dimensional  numerical fluxes.}
The new methods are shown to significantly outperform existing ones even if the latter are of higher order of accuracy and even on non-stationary solutions.
\end{abstract}

\end{frontmatter}

\section{Introduction}

This paper focuses on the numerical solution of nonlinear hyperbolic systems of conservation laws in two dimensions:
\begin{align}
  \del_t q + \del_x f + \del_y g &= 0, \label{eq:conservationlaw}
\end{align}
where $q$, $f$ and $g$ are the vectors of conservative variables and fluxes.
Numerical methods for hyperbolic partial differential equations (PDEs) need numerical diffusion
to achieve entropy stability and in order to deal with solutions characterized by strong gradients.
The majority of numerical methods for multi-dimensional problems, though, are developed following
a dimension-by-dimension approach, meaning that the numerical diffusion is usually derived in a
one-dimensional framework and that the diffusion term associated to an edge (or a face, in 3D) usually involves only two states.
Standard numerical methods with one-dimensional Riemann solvers typically introduce 
a diffusion term of the type 
\begin{align}
 \del_t q + \del_x f + \del_y g &= \Delta x \del_x(\nu_x \del_x q) + \Delta y \del_y(\nu_y \del_y q),
\end{align}
where $\Delta x$ and $\Delta y$ provide the size of the discretization, 
and $\nu_x$ and $\nu_y$ represent the diffusion coefficients,
which are often chosen proportional to the spectral radius of the flux Jacobian.
This one-dimensional approach does not take into account possible multi-dimensional
features of the numerical solution, such as the stationary states characterized by
a balance of contributions coming from different directions \cite{barsukow2019stationarity}.
For equation \eqref{eq:conservationlaw}, stationary states are governed by
\begin{align}
 \del_x f + \del_y g &= 0. \label{eq:stationary}
\end{align}
For classical methods, with the two-dimensional diffusion term designed 
following one-dimensional approaches, the solution will be completely
diffused instead of being kept stationary \cite{barsukow2021truly,barsukow2025structure}.
In contrast to Equation \eqref{eq:stationary}, the discrete stationary states are characterized by the much more restrictive conditions $\del_x f = 0$ and $\del_y g = 0$. States where $\del_x f \neq 0$ is balanced by $-\del_y g$ is not a stationary state of the numerical method.
This can be prevented by choosing more sophisticated diffusion operators 
\cite{morton2001vorticity,sidilkover2002factorizable,jeltsch2006curl,mishra2011constraint}. Such methods are called stationarity preserving \cite{barsukow2019stationarity}. 

\RII{Typically, solutions to \eqref{eq:stationary} form a very large set; the equations might even be underdetermined, as in the case, for example, in linear acoustics. 
In that case, they reduce merely to the condition that the velocity field must be divergence-free. If the divergence is understood weakly, this operator is not invertible even if all boundary conditions are prescribed. 
Consequently, no numerical method can exactly preserve all stationary states whenever their set is so rich. 
If we consider again  the example of linear acoustics, given a finite set of point values or averages, it is fundamentally impossible to establish whether they belong to a divergence-free vector field or not. }

\new{Stationarity preserving methods guarantee that the stationary states of a numerical method are a discretization of \emph{all} the stationary states of the PDE. 
The adjective \emph{all} is very important here, and consistency of a numerical method is not enough to guarantee this. 
As shown in \cite{Barsukow_nodeconservative2025}, an important necessary condition for a method to be stationarity preserving is that there exists a finite  set of local approximations 
of \eqref{eq:stationary} that characterize part of the kernel of the numerical method. In other words, when these local operators vanish, then the consistent part
and the numerical dissipation of the numerical method vanish simultaneously. For this condition to be also sufficient,  and for the stationary state uniquely defined,
the number of such approximations should be equal to the total number of unknowns. In the context of linear problems, an alternative definition of a stationarity preserving method (using the discrete Fourier transform) is given in \cite{barsukow2019stationarity}.

These properties can hardly be obtained in the context of classical methods.
As described above, all too often the stationary states of a (consistent) numerical method are discretizations of $\del_x f = 0$ and $\del_y g = 0$, instead of \eqref{eq:stationary}.
Consider, as another example, the trivial equation $\del_t q = 0$, which keeps everything stationary. A numerical method that adds diffusion like $\del_t q = \Delta x \del_x^2 q$ will be consistent, but then only those $q$ which are affine in space will be stationary states of the numerical method. In this paper, we aim to develop nonlinear methods whose stationary states are discretizations of all the stationary states of the PDE, while being stable under explicit time integration.}

Recent examples of stationarity preserving diffusion operators were developed for geostrophic equilibria in the linear and nonlinear case \cite{audusse2018analysis,audusse2025energy}.
A connection to these equilibria in the context of low Mach number limit of the Euler equations, which is related to the long-time limit of linear acoustics, was provided in \cite{jung2022steady,jung2024behavior} through the preservation of discrete divergence with ad-hoc functional spaces. Early examples of stationarity preserving methods for nonlinear conservation laws can be found in \cite{barsukow2018low}.
So far, however, no general theory for the agnostic detection of stationary states of nonlinear multi-dimensional hyperbolic partial differential equations is available.

The method presented in this work \new{exploits an idea formulated in  \cite{barsukow2025structure}, which allows
to modify the quadrature strategy in a way that systematically leads to a stationarity preserving method. This modification
is related to a  technique  often referred to as the global flux method  \cite{gascon2001construction,caselles2009flux,cheng2019new}}  initially introduced for 
hyperbolic  balance laws in one dimension,
\begin{align}
  \del_t q + \del_x f = s, \label{eq:balancelaw}
\end{align}
with the original goal of developing well-balanced methods 
\cite{audusse2004fast,berthon2016fully,castro-pares-jsc20},  
and the treatment of source terms present in the mathematical model. The global flux has already been successfully applied to different contexts and numerical methods \cite{chertock2018well,chertock2022well,ciallella2023arbitrary,mantri2024fully,micalizzi2024novel,kazolea2025approximate,alr2025} to preserve one dimensional equilibria. 
\new{The term ``global''  is  unfortunately misleading. Indeed, quite interestingly this approach leads to fully local discretizations (see e.g. \cite{mantri2024fully,micalizzi2024novel,barsukow2025structure,alr2025,brt25b}),
and indeed it is related to similar techniques used in the analysis of well balanced and asymptotic preserving methods, and referred to as the spatial localization of source terms e.g. in \cite{Gosse-toscani03}). These methods also have very natural connections with the so-called residual distributions schemes 
\cite{Abgrall2022,amr25}.}

\new{The underlying idea of these techniques is to recast}  the source term as a flux $R$:
\begin{align}
\partial_xR = s\Rightarrow R := \int^x s \dd x.
\end{align}
In this framework, equation \eqref{eq:balancelaw} can be recast as
\begin{align}
  \del_t q + \del_x (f - R) = 0,\label{eq:GFbalancelaw}
\end{align}
where discrete steady states satisfy the relation
\begin{align}
 \del_x (f - R) = 0 \Leftrightarrow f - R \equiv \mathscr F_0, 
\end{align}
with $\mathscr F=f-R$ \new{often referred to as   global flux, as it embeds the entire evolution operator,}
 and $\mathscr F_0 = \mathscr F(x_0)$ for a given $x_0$ in the domain.
In the same spirit, a similar approach that integrates the Coriolis term into an apparent bathymetry term was also developed in \cite{bouchut2004frontal}.

The concept of well-balancing is a particular case of the preservation of general stationary solutions. 
The overarching idea is to design numerical schemes in which the artificial diffusion vanishes at relevant equilibria.
The development of well-balanced schemes in one-dimensional problems has reached high levels of maturity in the last decades, but the multi-dimensional extensions are often tackled with trivial dimension-by-dimension approaches \cite{chertock2018well,ciallella2025high,michel2021two,mantri2024fully}, which only allows the preservation of 1D-like equilibria.
\new{In \cite{barsukow2025structure,brt25b} some of the present authors propose, in the context of stabilized finite elements, 
a modified quadrature approach that allows to guarantee the  multi-dimensional stationarity preserving properties of the resulting discretization.
The underlying idea has some connections with the above mentioned global flux methods.
In two space dimensions, it is based on  the idea of recasting the problem as}
\begin{align}
  \begin{split}
  \del_t q + \del_x f + \del_y g &= \del_t q + \del_y \left(\int^y \del_x f \, \dd y  + g\right)  \\ 
  &= \del_t q + \del_x  \left(f +    \int^x \del_y    g \, \dd x \right) = 0. \label{eq:conservationlawglobalflux0}
  \end{split}
\end{align}
\new{By symmetry, one combines the two above formulations, replacing the two-dimensional   the conservation law \eqref{eq:conservationlaw} with} 
\begin{align}
  \del_t q + \del_x f + \del_y g = \del_t q + \del_{x}\del_{y} (F + G) = 0, \label{eq:conservationlawglobalflux}
\end{align}
by defining
\begin{align}
 F &:= \int^y f \, \dd y, & G &:= \int^x g \, \dd x. \label{eq:globalflux}
\end{align}
The new divergence operator $\del_{x}\del_{y} (F + G)$ now is easy to preserve at the discrete level.
Thanks to this formulation, it becomes also straightforward to consider multi-dimensional balance laws with source terms,
\begin{align}
  \del_t q + \del_x f + \del_y g = s. \label{eq:2Dbalancelaw}
\end{align}
In this case, the source flux $R$ can be defined as
\begin{align}\label{eq:globalsource} 
R := \int^x \int^y s \dd x \dd y,
\end{align}
and directly included in the \new{modified divergence operator}. Setting 
\begin{equation}\label{eq:funnyF_def} 
\mathscr F= F + G - R,
\end{equation}
the multi-dimensional \new{stationarity preserving formulation of the original}    problem   becomes again
\begin{equation}\label{eq:mD-GF-form}
  \del_t q + \del_{xy}\mathscr F=  0.
\end{equation}

In this work, we present how this idea can be used to design first-order finite volume methods 
preserving multi-dimensional steady states, not known a priori, \textbf{for general nonlinear hyperbolic PDEs}.
A thorough analysis of the method is presented for linear problems showing a link with
other stationarity preserving methods, as well as a discrete energy estimate.
The approach proposed here naturally leads to the introduction of nonlinear \textbf{genuinely multi-dimensional fluxes at cell corners},
which have been shown to provide fundamental enhancements in the numerical solutions, and enjoy many
theoretical properties (\cite{gaburro2025multidRS,Barsukow_nodeconservative2025,amr25}).
However, differently from previous works, the formulation proposed here naturally
leads to corner fluxes, without any hypotheses on the type of quadrature, \new{or on the Riemann fluxes}. 
This approach is   a starting point for the development of new families
of stationarity preserving high-order methods based on high-degree polynomial reconstruction \cite{ciallella2021arbitrary,ciallella2023arbitrary}, or
 discontinuous Galerkin methods \cite{mantri2024fully,zhang2025well}. \new{Throughout the paper, we will refer to the new method 
 as stationarity preserving formulation or global flux quadrature as in \cite{barsukow2025structure,brt25b}.  
 However, note that the discretization is fully local, as we will show in detail.}

The paper is organized as follows.
In section \ref{sec:models}, we present the examples of PDEs that will be considered when assessing the performance of the method experimentally.
In section \ref{sec:GF1d}, we recall the global flux method in a one-dimensional framework 
for hyperbolic balance laws. In section \ref{sec:GF2d}, we present the  \new{two-dimensional stationarity preserving/global flux quadrature approach for}
hyperbolic PDEs. Here, we discuss the finite volume formulation, 
the stabilization technique, boundary conditions, the treatment of source terms, as well as stability and consistency of the method for a linear model.
In section \ref{sec:FVcomparison}, we present the standard finite volume method with piecewise constant and 
piecewise linear reconstructions used for comparison with the global flux method.
Several numerical experiments are presented in section \ref{sec:numerical} to show the performance of the method.
Finally, we draw some conclusions in section \ref{sec:conclusions}.

\section{Mathematical models}\label{sec:models}
The numerical method presented in this work is rather general and, in order to show its potential, 
we exemplify it on several mathematical models described by both linear and nonlinear hyperbolic systems. 
In particular, herein we will focus on three systems: linear acoustics, Euler equations for gas dynamics and the shallow water equations. For all of them, we focus on two-dimensional problems.

\subsection{Linear acoustic system}\label{sec:acoustic}

The system of linear acoustic is a simple model that directly embeds non-trivial divergence-free steady states.
It can be written in the following 2D and vectorial forms as: 
\begin{equation}\label{eq:LinAc}
\begin{cases}
\del_t u + \del_x p  &= 0,\\
\del_t v + \del_y p  &= 0,\\
\del_t p + \del_x u + \del_y v &= 0,
\end{cases}\qquad \qquad \qquad 
\begin{cases}
  \del_t \mathbf{v} + \nabla p  &= 0,\\
  \del_t p + \nabla\cdot\mathbf{v} &= 0, \\
\end{cases}\qquad  
\end{equation}
where $p$ is the pressure and $\mathbf{v}=(u,v)$ is the velocity.
The system can also be written in the compact form \eqref{eq:conservationlaw} with
\begin{equation}
q = \begin{bmatrix} u \\ v \\ p\end{bmatrix}, \qquad f = \begin{bmatrix} p  \\ 0  \\ u  \end{bmatrix}, 
\qquad g = \begin{bmatrix}  0  \\ p  \\ v  \end{bmatrix}.
\end{equation}
The steady states of this system are given by
\begin{equation}
\del_t q \equiv 0 \qquad \Leftrightarrow \qquad \nabla\cdot\mathbf{v}\equiv 0 \quad \text{and} \quad p\equiv p_0 = \text{const}.
\end{equation}

\subsection{Euler equations}\label{sec:euler}

The Euler equations are a simplification of the full Navier-Stokes system that do not include viscosity effects.
Their use is widespread for the simulation of compressible gas dynamics.
The system can be written in vectorial form as:
\begin{align}\label{eq:Euler}
	\begin{cases}
\del_t \rho + \nabla\cdot(\rho \mathbf{v}) &= 0,\\
\del_t (\rho \mathbf{v}) + \nabla\cdot\left(\rho\mathbf{v}\otimes\mathbf{v} + p\mathrm{I}\right) &= 0,\\
\del_t (\rho E) + \nabla\cdot(\rho H \mathbf{v}) &= 0,
	\end{cases}
\end{align}
having denoted by $\rho$ the density, by $\mathbf{v}$ the velocity field, by $E=e+\|\mathbf{v}\|^2/2$ the specific total energy, being $e$ the specific internal energy and $I$ is the identity matrix.
Finally, the total specific enthalpy is $H = h +\|\mathbf{v}\|^2/2$, with $h= e+p/\rho$ the specific enthalpy. 
To close the system, we use the classical perfect gas equation of state $p = (\gamma-1)\rho e$ with $\gamma$ 
the constant ratio of specific heats ($\gamma=1.4$ for air).

The nonlinear system of Euler equations can also be recast in the compact form  \eqref{eq:conservationlaw} with
\begin{equation}
  q = \begin{bmatrix} \rho \\ \rho u \\ \rho v \\ \rho E \end{bmatrix}, \qquad f = \begin{bmatrix} \rho u  \\  \rho u^2+ p \\  \rho uv  \\ \rho H u \end{bmatrix}, 
  \qquad g = \begin{bmatrix} \rho v  \\  \rho uv \\  \rho v^2+ p \\ \rho H v \end{bmatrix}. 
\end{equation}
Steady states of the Euler equations are more complex but, after some manipulations, 
the smooth steady states can be characterized by the following relations:
\begin{equation}
 \nabla\cdot(\rho \mathbf{v}) = 0, \qquad (\rho\mathbf{v}\cdot\nabla) \mathbf{v} +  \nabla p = 0, \qquad \mathbf{v}\cdot\nabla  H = 0. 
\end{equation}

\subsection{Shallow water system}\label{sec:SW}

The Saint-Venant or shallow water equations describe the dynamics of hydrostatic free surface waves influenced by gravity. 
This model is valid under the hypothesis of very large wavelengths, or very shallow depths,
and is applied in various engineering fields, including river and estuarine hydrodynamics, 
urban flood management, and tsunami risk assessment. In particular, when working with large scale problems, this simplified model becomes 
crucial to speed-up the computational time.

The system can be written in vectorial form as:
\begin{equation}\label{eq:SWE}
\begin{cases}
  \del_t h + \nabla\cdot(h\mathbf{v}) &= 0,\\
  \del_t (h\mathbf{v}) + \nabla\cdot\left(h\mathbf{v}\otimes\mathbf{v} + \frac{1}{2}gh^2\mathrm{I}\right) &=-gh\nabla b,
\end{cases}\qquad
\end{equation}
where $h$ is the water height, $\mathbf{v}$ the velocity field, $b$ is the bathymetry and $g$ is the gravity constant.
The system can also be written in the classical compact notation \eqref{eq:2Dbalancelaw} with
\begin{equation}
q = \begin{bmatrix} h \\ hu \\ hv\end{bmatrix}, \qquad f = \begin{bmatrix} hu  \\  hu^2+ \frac{1}{2}gh^2 \\  huv \end{bmatrix}, 
\qquad g = \begin{bmatrix} hv  \\  huv \\  hv^2+ \frac{1}{2}gh^2  \end{bmatrix}, \qquad s = \begin{bmatrix} 0 \\ -gh\del_x b \\ -gh\del_y b\end{bmatrix}.
\end{equation}
This system admits a large variety of equilibria depending on the interaction between the flux and the source. 
The most studied equilibria in the context of well-balanced methods are the so-called ``lake at rest'' 
states, which are characterized by a constant free surface level $\eta:=h+b\equiv \eta_0$ 
and a zero velocity $\mathbf{v}\equiv 0$. 
However, in the presence of a non-flat bathymetry, the system can also admit non-trivial equilibria,
which are characterized by a non-zero velocity and a non-flat free surface level:
\begin{align}
 \nabla \cdot (h \vec v) &= 0 & (\vec v \cdot \nabla) \vec v + g \nabla (h+b) &= 0
\end{align}

Several works have been devoted to the study of these equilibria in one dimension or in 
a quasi-1D framework \cite{michel2017well,mantri2024fully,ciallella2025high}.
In this work, we are interested in truly multi-dimensional well-balanced schemes that are capable 
of preserving all these equilibria at the discrete level.

\section{Global flux for 1D balance laws}\label{sec:GF1d}

In this section, we recall the main principle of the global flux method and its initial usage in a 1D framework.
Consider a general nonlinear balance law \eqref{eq:balancelaw} and 
define a global flux $\mathscr F$ as
\begin{align}
 \mathscr F &= f - \int^x s \dd x,
\end{align}
such that \eqref{eq:balancelaw} can now be written in a \new{pseudo}-conservative form as 
\begin{align}
 \del_t q + \del_x \mathscr F &= 0.
\end{align}
Steady states given by $\del_t q = 0$
are equivalently characterized by the condition $\del_x\mathscr F = 0 \Leftrightarrow \mathscr F \equiv \mathscr F_0\in \mathbb R$.

In a finite volume framework, the computational domain $\Omega$ is split into $N$ cells 
and the equation \eqref{eq:balancelaw} is integrated over each cell $C_i=\left[x_{i-\frac12},x_{i+\frac12}\right]$:
\begin{align}\label{eq:gf1dsemidiscr}
  \frac{\dd}{\dd t} \bar q_i + \frac{\widehat{\mathscr F}_{i+\frac12} - \widehat{\mathscr F}_{i-\frac12}}{\Delta x} &= 0,
\end{align}
where the cell average is defined as
\begin{align}
   \bar q_i := \frac{1}{\Delta x} \int_{x_{i-\frac12}}^{x_{i+\frac12}} q \, \dd x.
\end{align}
The numerical global flux $\widehat{\mathscr F}_{i+\frac12}$ is considered to be a function of the two values of the global flux $\mathscr F^L_{i+ \frac12}$ and $\mathscr F^R_{i+ \frac12}$ 
reconstructed at both sides of interface $x_{i+ \frac12}$. For piecewise constant reconstructions of the global flux one simply has $\mathscr F^L_{i+ \frac12}=\mathscr F_i$ and $\mathscr F^R_{i+ \frac12}=\mathscr F_{i+1}$.
\begin{remark}[Numerical global flux]
It is important to underline that structure preservation can only be achieved if the interface global flux $\widehat{\mathscr F}_{i+\frac12}$ 
\new{is constant $\forall\;i$ at steady state. That implies that numerical dissipation should also vanish in correspondence of this state.
Unless this condition is imposed with some ad hoc modification (see e.g. \cite{alr2025}), 
this is only possible if the numerical dissipation}
depends only on global fluxes $\{ \mathscr F_j \}_{j\in \mathbb Z}$ in the cells, and not on the values $\{ q_j \}_{j\in \mathbb Z}$ of the conservative variables.
This is due to the fact that, at equilibria, only global fluxes are constant while conservative variables may vary.
\end{remark}
In our previous work \cite{ciallella2023arbitrary}, we employed the following upwind flux:
\begin{align}\label{eq:upwindflux1d}
\widehat{\mathscr F}_{i+\frac12} = \splitsymbol^+ \mathscr F^L_{i+\frac12} + \splitsymbol^- \mathscr F^R_{i+\frac12} \quad \text{where}\quad \splitsymbol^\pm := L^{-1} \frac{\id \pm \mathrm{sign\,}\Lambda}{2} L,
\end{align}
where $L$ is the matrix of left eigenvectors of the flux Jacobian $\del_q f$, and 
$\mathrm{sign\,} \Lambda$ is the diagonal matrix of the sign of the eigenvalues of the flux Jacobian, evaluated using any (average) state at the interface\footnote{In principle, therefore, one should write $\splitsymbol^\pm_{i+\frac12}$ to make clear that they differ from interface to interface; we do not make this depends explicit to ensure readability}.

For the development of the global flux method, a consistent approximation of $R:=\int^x s \, \dd x$ is necessary to define $\mathscr F = f-R$.
This integral can be computed in a recursive manner, starting from the beginning of the domain, by integrating the source in each element.
To simplify the description of the method, we will assume that $q$ and $s$ are constant in each cell,
therefore the source integral can be computed as
 \begin{align}
 R_i 
 &:=  \underbrace{\int^{x_{i-1}} s \,\dd x}_{R_{i-1}}  + \int^{x_{i-\frac12}}_{x_{i-1}} s \,\dd x + \int^{x_i}_{x_{i-\frac12}} s \,\dd x  =  R_{i-1} + \frac{\Delta x }{2}\bar s_{i-1} + \frac{\Delta x}{2}\bar s_i .
 \end{align}
 Hence, the global flux will now depend on both the conservative flux and the source term
 \begin{align}\label{eq:gf1drecursive}
  \mathscr F_i&= f(\bar q_i) - R_i =  f(\bar q_i) - R_{i-1} - \frac{\Delta x }{2}(\bar s_{i-1} +\bar s_i ).
\end{align}
Similarly, the recursive procedure gives us the following values for $\mathscr F_{i-1}$ and $\mathscr F_{i+1}$:
\begin{align}
 \mathscr F_{i-1} &= f(\bar q_{i-1}) - R_{i-1} , \\
 \mathscr F_{i+1} &= f(\bar q_{i+1}) - R_{i-1} - \Delta x \left( \frac{ 1}{2}\bar s_{i-1} + \bar s_{i} + \frac{1}{2}\bar s_{i+1} \right).
\end{align}
It can be noticed that, when considering a simple 
numerical flux 
\begin{align}
\widehat{\mathscr F}_{i+\frac12}(\mathscr F_i, \mathscr F_{i+1}) &= \mathscr F_i
\end{align}
equation \eqref{eq:gf1dsemidiscr} can be recast as
\begin{align}\label{eq:gf1dresidual}
\frac{\dd}{\dd t} \bar q_i = -\frac{\mathscr F_i - \mathscr F_{i-1}}{\Delta x} = -\frac{f(\bar q_i)-f(\bar q_{i-1})}{\Delta x} + \frac{\bar s_i + \bar s_{i-1}}{2},
\end{align}
which shows already a difference with respect to the classical finite volume method, where the source term would be treated in a centered way.

When the upwind numerical flux \eqref{eq:upwindflux1d} is used, one has (having temporarily made the dependence of $\splitsymbol^\pm$ on the interface explicit)
\begin{align}
\frac{\dd}{\dd t} \bar q_i &= -\frac{\splitsymbol^-_{i+\frac12} f (\bar q_{i+1}) + (\splitsymbol^+_{i+\frac12} + \splitsymbol^-_{i-\frac12})f(\bar q_i)- \splitsymbol^+_{i-\frac12}f(\bar q_{i-1})}{\Delta x} \\
& \nonumber + (\splitsymbol^+_{i+\frac12} + \splitsymbol^-_{i+\frac12} - \splitsymbol^-_{i-\frac12})\frac{\bar s_i + \bar s_{i-1}}{2} +  \splitsymbol^-_{i+\frac12}\frac{\bar s_i + \bar s_{i-1}}{2} + (\underbrace{\splitsymbol^+_{i+\frac12} + \splitsymbol^-_{i+\frac12}}_{= \splitsymbol} - \underbrace{(\splitsymbol^+_{i-\frac12} + \splitsymbol^-_{i-\frac12})}_{= \splitsymbol}) R_{i-1},
\end{align}
where the contribution from $R_{i-1}$ cancels out even if $\splitsymbol^\pm$ depend on the interface, since they nevertheless add up to $\splitsymbol$ on each of them.

\begin{remark}[Compactness of global fluxes]
It can be noticed that, although the global flux in \eqref{eq:gf1drecursive} is defined globally with $R_{i-1}$
that depends on previous values, the time residual \eqref{eq:gf1dresidual}  shows that the stencil is actually compact due to the cancellation of these terms. 
\new{In particular, simple manipulations show that the above upwind method can be also written as
$$
\begin{aligned}
\frac{\dd}{\dd t} \bar q_i = - & \dfrac{\widehat{\mathscr F}_{i+\frac12} - \mathscr F_i}{\Delta x} - \dfrac{\mathscr F_i  - \widehat{\mathscr F}_{i-\frac12}}{\Delta x} \\
= -  & \dfrac{ \splitsymbol^-_{i+1/2}(  \mathscr F_{i+1} - \mathscr F_i)}{\Delta x} 
-  \dfrac{ \splitsymbol^+_{i-1/2}(  \mathscr F_{i} - \mathscr F_{i-1})}{\Delta x}  \\
= -  &  \splitsymbol^-_{i+1/2} \left( \dfrac{ f(\bar q_{i+1}) -f(\bar q_{i})  }{\Delta x}  - \frac{\bar s_{i+1} + \bar s_{i}}{2}\right) 
- \splitsymbol^+_{i-1/2} \left( \dfrac{ f(\bar q_{i}) -f(\bar q_{i-1})  }{\Delta x}  - \frac{\bar s_{i} + \bar s_{i-1}}{2}\right)
\end{aligned}
$$
which is the well known upwind splitting method  dating back to the early works by P.L.\ Roe \cite{Roe87}, and later on Bermudez and Vazquez \cite{bv94}. 
In the multidimensional case, a} full analogy with compact residual distribution methods on a dual cell is presented later in section~\ref{sec:consistency}.
\end{remark}

\section{\new{Stationarity preserving formulation} for multi-dimensional hyperbolic PDEs}\label{sec:GF2d}

\subsection{Numerical method}

When dealing with multi-dimensional conservation laws, non-trivial equilibria arise also in absence of a source term in the equation. For steady states $\del_t q = 0$, 
it is no longer just $\del_x f = 0$ that follows, but instead the divergence $\del_x f + \del_y g = 0$, 
which in general might have many solutions.

\new{To  design stationarity preserving finite volumes schemes, we }
 start by rewriting \eqref{eq:conservationlaw} as \eqref{eq:conservationlawglobalflux} using the definitions in \eqref{eq:globalflux} to obtain
\begin{equation}
  \del_t q + \del_{xy}\mathscr F=  0.
\end{equation}
\new{where, by analogy with the one dimensional case, we  refer to $\mathscr F := F + G$  as a}
global flux\new{, accounting for both contributions of the fluxes in the $x$ and $y$ direction (and later of the source term)}.

Integration of \eqref{eq:mD-GF-form} over the cell $C_{i,j}=\left[x_{i-\frac12},x_{i+\frac12}\right]\times\left[y_{j-\frac12},y_{j+\frac12}\right]$ yields
\begin{align}\label{eq:evolutionGF} 
	\begin{split}
  \Delta x\Delta y\frac{\dd}{\dd t}\bar  q_{i,j} + \mathscr F(t, x_{i+\frac12},y_{j+\frac12}) - \mathscr F(t, x_{i-\frac12},y_{j+\frac12})-\mathscr F(t, x_{i+\frac12},y_{j-\frac12}) +\mathscr F(t, x_{i-\frac12},y_{j-\frac12})= 0,
	\end{split}
\end{align}
where the cell average is defined as
$$ \bar q_{i,j}:=\frac{1}{\Delta x\Delta y}\int_{x_{i-\frac12}}^{x_{i+\frac12}}\int_{y_{j-\frac12}}^{y_{j+\frac12}} q \, \dd x \dd y . $$
\new{From \eqref{eq:evolutionGF}, we see that with this new formulation leads to the introduction of} the numerical \textbf{corner fluxes }$\widehat{\mathscr F}_{i\pm\frac12,j\pm\frac12}$ that then allow to write the evolution equation as
\begin{align}\label{eq:GF2d}
  \frac{\dd}{\dd t} \bar q_{i,j} + \frac{\widehat{\mathscr F}_{i+\frac12,j+\frac12} - \widehat{\mathscr F}_{i-\frac12,j+\frac12} - \widehat{\mathscr F}_{i+\frac12,j-\frac12} + \widehat{\mathscr F}_{i-\frac12,j-\frac12}}{\Delta x \Delta y} = 0.
\end{align}
Recall that the global flux $\mathscr F$ \new{accounts for contributions of the} physical fluxes \new{in all spatial directions, as well as of}  the source, see  equations
\eqref{eq:globalflux}, \eqref{eq:globalsource}, and \eqref{eq:funnyF_def}.
In practice, the \new{transversally integrated fluxes} $F$ and $G$ are computed  \new{via quadrature} along, respectively,
the $y$ and $x$ directions in a 1D fashion. In particular, the value of $F_i$ in the barycenter 
of a given cell $i$ can be computed recursively starting from the beginning of the domain,
similarly to section \ref{sec:GF1d}, as
\begin{align} \label{eq:Frecursion}
	\begin{split}
  F_{i,j} &= \int^{y_{j-1}} f \, \dd y + \int_{y_{j-1}}^{y_{j}} f \, \dd y = F_{i,j-1} + \frac{\Delta y}{2} (f_{i,j-1} + f_{i,j}), \qquad \forall i,
	\end{split}
\end{align}
where, for a first order method, trapezoidal rule is accurate enough.
Similarly, $G_i$ can be computed as
\begin{align}\label{eq:Grecursion}
  G_{i,j} &= \int^{x_{i-1}} g \, \dd x + \int_{x_{i-1}}^{x_{i}} g \, \dd x =G_{i-1,j} + \frac{\Delta x}{2} (g_{i-1,j} + g_{i,j}), \qquad \forall j.
\end{align}
When dealing with hyperbolic PDEs with source terms, as for the shallow water equations in section \ref{sec:SW}, 
the integral of the source term $R:=\int^x\int^y s \, \dd x\dd y$ can also be embedded into the global flux. 
Similarly, $R$ can be recursively defined as 
\begin{equation}\label{eq:integral source}
\new{	R_{i,j} = R_{i-1,j} + R_{i,j-1} - R_{i-1,j-1} + \frac{\Delta x \Delta y}{4} \left(s_{i-1,j-1} + s_{i-1,j}+ s_{i,j-1}+s_{i,j}\right),}
\end{equation}
for every $i,j$.
The treatment of source terms can be implemented through the compact version 
of this formula, which involves only local source values, that does not use recursion, and can be found in section \ref{sec:Compactness}. 
More details about the treatment of source terms for the shallow water equations \ref{sec:SW}, 
are given in section \ref{sec:sourceSW}.

\subsection{Numerical corner fluxes}\label{sec:supg}

\begin{figure}
\centering
\subfigure[Main grid]{
  \begin{adjustbox}{width=0.3\textwidth} 
    \begin{tikzpicture}
        \draw[dashed,gray] (0,0) grid [step=2] (6,6);
        \draw[ultra thick] (2,2) grid [step=2] (6,6);
        \draw[red,fill=red] (4,4) circle[radius=4pt];

        \node at (1,5) {\( C_{i-1,j+1} \)};
        \node at (3,5) {\( C_{i,j+1} \)};
        \node at (5,5) {\( C_{i+1,j+1} \)};
        
        \node at (1,3) {\( C_{i-1,j} \)};
        \node at (3,3) {\( C_{i,j} \)};
        \node at (5,3) {\( C_{i+1,j} \)};
        
        \node at (1,1) {\( C_{i-1,j-1} \)};
        \node at (3,1) {\( C_{i,j-1} \)};
        \node at (5,1) {\( C_{i+1,j-1} \)};      
    \end{tikzpicture}
  \end{adjustbox}}\qquad\qquad
  \subfigure[Dual grid]{
    \begin{adjustbox}{width=0.3\textwidth}
      \begin{tikzpicture}
          \draw[dashed,gray] (0,0) grid [step=2] (6,6);
          \draw[ultra thick] (2,2) grid [step=2] (6,6);
          \draw[red,fill=red] (4,4) circle[radius=4pt];
          \draw[ultra thick,green] (3,3) -- (3,5) -- (5,5) -- (5,3) -- (3,3);
          \draw[green,fill=green] (3,3) circle[radius=4pt];
          \draw[green,fill=green] (3,5) circle[radius=4pt];
          \draw[green,fill=green] (5,3) circle[radius=4pt];
          \draw[green,fill=green] (5,5) circle[radius=4pt];
          
          \node at (3,5.5) {\( C_{i,j+1} \)};
          \node at (5,5.5) {\( C_{i+1,j+1} \)};
          
          \node at (3,2.5) {\( C_{i,j} \)};
          \node at (5,2.5) {\( C_{i+1,j} \)};
          
      \end{tikzpicture}
    \end{adjustbox}}
\caption{Cell labeling for the 2D grid.}\label{fig:celllabeling}
\end{figure}
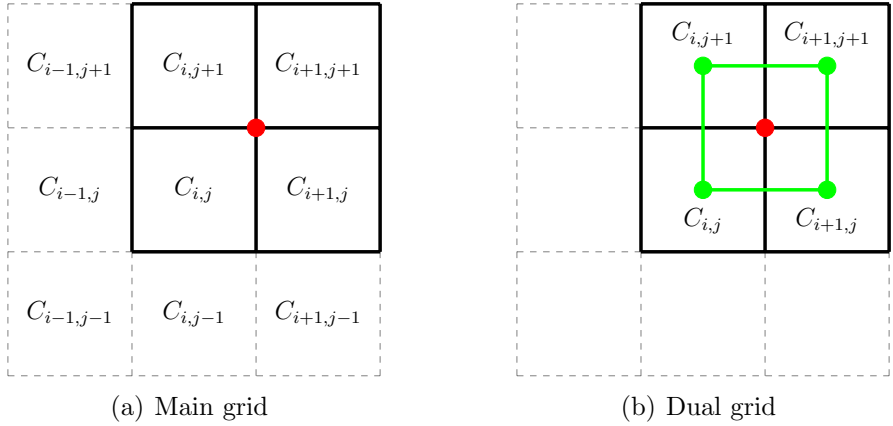

The conservative formulation obtained using global fluxes \eqref{eq:GF2d} requires the definition of numerical corner fluxes
to update cell averages.  This definition is better achieved by considering  the evolution of all the cells neighboring a given node.
To this end it is preferable to recast \eqref{eq:GF2d} as 
\begin{align}\label{eq:GF2d1}
    \frac{\dd}{\dd t} \bar q_{i,j} + \frac{\widehat{\mathscr F}^{(i,j)}_{i+\frac12,j+\frac12} + \widehat{\mathscr F}^{(i,j)}_{i-\frac12,j+\frac12}+ \widehat{\mathscr F}^{(i,j)}_{i+\frac12,j-\frac12} + \widehat{\mathscr F}^{(i,j)}_{i-\frac12,j-\frac12}}{\Delta x \Delta y} = 0.
\end{align}
where we have added the superscript $^{(i,j)}$ to account for the fact that the signs used in the formula involve the  flux balance for $\bar q_{i,j}$. They thus are interpreted as depending on the orientation of the corner normal for the four cells $(i+\ell,j+r)$ for $\ell,r\in\lbrace 0,1 \rbrace$ with respect to the corner $(i+\frac12,j+\frac12)$, defined by
\begin{align}\label{eq:cell_corner_normal}
    &\mathbf{n}^{(i+\ell,j+r)}_{i+\frac12,j+\frac12}:= \begin{pmatrix}
		(-1)^{\ell+1}\\
		(-1)^{r+1}
	\end{pmatrix},\,\text{i.e.,}\\
		\mathbf{n}^{(i,j)}_{i+\frac12,j+\frac12}:= \begin{pmatrix}
		-1\\
		-1
	\end{pmatrix},\,&
	\mathbf{n}^{(i+1,j)}_{i+\frac12,j+\frac12}:= \begin{pmatrix}
	1\\
	-1
	\end{pmatrix},\,
		\mathbf{n}^{(i,j+1)}_{i+\frac12,j+\frac12}:= \begin{pmatrix}
		-1\\
		1
	\end{pmatrix},\,
		\mathbf{n}^{(i+1,j+1)}_{i+\frac12,j+\frac12}:= \begin{pmatrix}
		1\\
		1
	\end{pmatrix}.
\end{align}
Thus $\mathbf{n}^c_p$ is a normal at corner $p$ pointing into cell $c$. Corner normals and corner fluxes alongside a modified concept of conservation associated to corners rather than edges are used e.g. in \cite{Barsukow_nodeconservative2025} for general polygonal grids, where they are shown to be crucial for structure preservation. There, corner normals are defined as the average of the two edge-normals adjacent to the node, scaled with the respective edge lengths. We believe that the definition above is sufficient in the context of Cartesian meshes.

We further define the scalar $n^{(i+\ell,j+r)}_{i+\frac12,j+\frac12 }$ as
\begin{equation}\label{eq:nij}
n^{(i+\ell,j+r)}_{i+ \frac12,j+ \frac12 } = n(\mathbf{n}^{(i+\ell,j+r)}_{i+\frac12,j+\frac12})=  (-1)^{\ell+1}(-1)^{r+1} = (-1)^{\ell+r} ,
\end{equation}
where $n(\mathbf{n}):= \mathbf{n}_x \mathbf{n}_y$ is the product of the two components. 

The reinterpreted definition of the corner fluxes above requires  a different setting compared to the classical one, and appears in the context of genuinely multi-dimensional Riemann solvers using more than two states as input (see e.g.\ \cite{balsarahlle2d,BalsaraMultiDRS,gaburro2025multidRS,amr25} and references therein). 

Next, we aim at defining the notion of the numerical global flux in the multi-dimensional context as generally as possible. Then, we elucidate the conditions imposed on the functional form of the numerical flux by consistency, conservation, and preservation of steady states.
We define the numerical corner fluxes at the corner $(i+\frac12,j+\frac12)$ with respect to the four cells $(i+\ell, j+r)$ with $\ell,r\in\lbrace 0,1\rbrace$ as
$$
  \widehat{\mathscr F}_{i+\frac12,j+\frac12}^{(i+\ell,j+r)}  =  \widehat{\mathscr F}(\mathscr F_{i,j},\mathscr F_{i,j+1},\mathscr F_{i+1,j},\mathscr F_{i+1,j+1};\bar q_{i,j},\bar q_{i,j+1},\bar q_{i+1,j},\bar q_{i+1,j+1}|\mathbf{n}^{(i+\ell,j+r)}_{i+\frac12,j+\frac12}). 
$$
We can now  formulate  local consistency as
\begin{equation}\label{eq:consistency-cornerF}
 \widehat{\mathscr F} (\mathscr F,\mathscr F,\mathscr F,\mathscr F; q, q,q,q| \mathbf{n})   = \mathscr F\, n(\mathbf{n}) .
\end{equation}
A stronger property is the steady state preservation requirement which can be expressed as 
\begin{equation}\label{eq:ssp-cornerF}
 \widehat{\mathscr F}
 (\mathscr F,\mathscr F,\mathscr F,\mathscr F;\bar q_{i,j},\bar q_{i,j+1},\bar q_{i+1,j},\bar q_{i+1,j+1}| \mathbf n^{(i,j)}_{i+\frac12,j+\frac12})   = \mathscr F\, n^{(i,j)}_{i+\frac12,j+\frac12}.
\end{equation}

\begin{remark}[Steady state subspace]\label{rem:steady_state_subspace}
Following \cite{barsukow2025structure}, it will be shown below that  
steady state preservation may be actually already proven if
\begin{equation}\label{eq:steady_state_global_flux}
	\mathscr F^*_{i,j} = \mathscr F + \alpha_i + \beta_j
\end{equation}
for any two data distributions $\alpha$ and $\beta$, such that $\alpha$ only
depends on $i$ and $\beta$ only on $j$. 
In this case, a conservation property more general of \eqref{eq:ssp-cornerF} reads
\begin{equation}\label{eq:ssp-cornerF_up_to_ij}
	\widehat{\mathscr F}_{i+\frac12,j+\frac12}^{(i,j)}(\mathscr F^*_{i,j},\mathscr F^*_{i,j+1},\mathscr F^*_{i+1,j},\mathscr F^*_{i+1,j+1};\bar q_{i,j},\bar q_{i,j+1},\bar q_{i+1,j},\bar q_{i+1,j+1}| \mathbf n^{(i,j)}_{i+\frac12,j+\frac12})   = \mathscr F\, n^{(i,j)}_{i+\frac12,j+\frac12}.
\end{equation}
This condition will be used also in the consistency analysis of Section~\ref{sec:consistency}.
\end{remark}

Conservation cannot be expressed by face in this framework, as in standard finite volume methods.
It is instead formulated at corners as follows: 
\begin{align}\label{eq:conservation}
	\widehat{\mathscr F}^{(i,j)}_{i+\frac12,j+\frac12} + \widehat{\mathscr F}^{(i+1,j)}_{i+\frac12,j+\frac12} + \widehat{\mathscr F}^{(i,j+1)}_{i+\frac12,j+\frac12}  + \widehat{\mathscr F}^{(i+1,j+1)}_{i+\frac12,j+\frac12} &= 0.
\end{align}
 
Having established the necessary conditions that the numerical global fluxes have to satisfy, we next propose a particular choice. As in \cite{gaburro2025multidRS}, we define numerical global fluxes as the sum of a consistent 
central flux plus a diffusion term $\mathcal D$.
In the present paper, corner fluxes are obtained   extending to quadrilaterals the multi-dimensional Osher-Solomon  flux proposed in  \cite{gaburro2025multidRS}, 
and combining it  with the recent work \cite{barsukow2025structure} on  \new{stationarity preserving, global flux quadrature} SUPG stabilization. 
We define the numerical corner flux around the corner $(x_{i+\frac12},y_{j+\frac12})$, which 
 involves the cells $C_{i,j}$, $C_{i+1,j}$, $C_{i,j+1}$ and $C_{i+1,j+1}$ (see figure \ref{fig:celllabeling}). 
The same principle is applied to the other corners. We set
\begin{align}\label{eq:diff}
\begin{split}
  \widehat{\mathscr F}^{(i,j+1)  }_{i+\frac12,j+\frac12} &= \overline{\mathscr F}_{i+\frac12,j+\frac12}n^{(i,j+1)}_{i+\frac12,j+\frac12}  +  \mathcal D^{(i,j+1)  }_{i+\frac12,j+\frac12},
  \quad\quad 
  \widehat{\mathscr F}^{(i+1,j+1)}_{i+\frac12,j+\frac12} = \overline{\mathscr F}_{i+\frac12,j+\frac12}n^{(i+1,j+1)}_{i+\frac12,j+\frac12}  +  \mathcal D^{(i+1,j+1)}_{i+\frac12,j+\frac12}, \\
  \widehat{\mathscr F}^{(i,j)    }_{i+\frac12,j+\frac12} &= \overline{\mathscr F}_{i+\frac12,j+\frac12}n^{(i,j)}_{i+\frac12,j+\frac12}  +  \mathcal D^{(i,j)    }_{i+\frac12,j+\frac12},
  \quad\quad 
  \widehat{\mathscr F}^{(i+1,j)  }_{i+\frac12,j+\frac12} = \overline{\mathscr F}_{i+\frac12,j+\frac12}  n^{(i+1,j)}_{i+\frac12,j+\frac12}+  \mathcal D^{(i+1,j)  }_{i+\frac12,j+\frac12},
\end{split}
\end{align}
where 
$$\overline{\mathscr F}_{i+\frac12,j+\frac12} = \frac{1}{4}( \mathscr F_{i,j} + \mathscr F_{i+1,j} + \mathscr F_{i+1,j+1} + \mathscr F_{i,j+1} )$$ 
is the average of the global fluxes at the corner.

To define the numerical dissipation we consider corner dual cells, as  depicted on
the right on figure \ref{fig:celllabeling}. The conservation condition \eqref{eq:conservation} requires that
$$
 \mathcal D^{(i,j)}_{i+\frac12,j+\frac12} + \mathcal D^{(i+1,j)}_{i+\frac12,j+\frac12} + \mathcal D^{(i,j+1)}_{i+\frac12,j+\frac12}  +  \mathcal D^{(i+1,j+1)}_{i+\frac12,j+\frac12} = 0.
$$
To define the corner dissipation, we cannot proceed as in \cite{gaburro2025multidRS}, since this would
break the stationarity preserving property. Instead,  we take inspiration from the streamline upwind stabilization (SUPG), studied in the global flux context in \cite{barsukow2025structure}. To this end, on the dual cell $\widetilde C_{i+\frac12,j+\frac12}$
we compute
  SUPG stabilizing terms 
\begin{align}
	\begin{split}
   \mathcal{D}^{(i+\ell,j+r)}_{i+\frac12,j+\frac12}:=&\mathcal{D} (\widetilde{\mathscr F}_{i+\frac12,j+\frac12},\bar{q}_{i+\frac12,j+\frac12}|\mathbf{n}^{(i+\ell,j+r)}_{i+\frac12,j+\frac12})\\
   =& \alpha \Delta \int_{\widetilde C_{i+\frac12,j+\frac12}} \left(\frac{1}{\Delta x} J^x \del_\xi \phi_{\ell,r} + \frac{1}{\Delta y} J^y \del_\eta \phi_{\ell,r}\right) \del_{\xi\eta} \widetilde{\mathscr F}_{i+\frac12,j+\frac12} \, \dd \xi \dd \eta,
	\end{split}
\end{align}
where $\widetilde{\mathscr  F}_{i+\frac12,j+\frac12}$ is a bi-linear $\mathbb Q^1$ reconstruction of the global flux on the dual cell from the four adjacent values 
$\mathscr F_{i,j}$, $\mathscr F_{i+1,j}$, $\mathscr F_{i,j+1}$, $\mathscr F_{i+1,j+1}$.
Here, $J^x$ and $J^y$ are the Jacobians of the fluxes $f$ and $g$ computed in the average value $\bar{q}_{i+\frac12,j+\frac12}= \frac{q_{i,j}+q_{i+1,j}+q_{i,j+1}+q_{i+1,j+1}}{4}$. $\Delta=\frac{\sqrt{\Delta x^2+\Delta y^2}}{\sqrt{2}}$ is the characteristic mesh size,
and $\alpha=1/ \lambda_m$ with $\lambda_m$ the maximal spectral radius of the flux Jacobians computed with the average 
state of the four reconstructed values at the corner. The $\varphi_{\ell,r}$ for $\ell,r \in \lbrace 0,1 \rbrace$ in the above definition are the standard bi-linear finite element
basis functions on the quadrilateral $\widetilde C$ defined by 
\begin{equation}
	\varphi_{\ell,r}(\xi,\eta)=\frac14 (1+(-1)^{\ell+1}\xi) (1+(-1)^{r+1}\eta),
\end{equation} i.e.,
\begin{align}
	\begin{split}
	\phi_{0,0}(\xi,\eta)   = \frac14(1-\xi)(1-\eta), \quad &\phi_{1,0}(\xi,\eta) = \frac14(1+\xi)(1-\eta)  \\
	\phi_{0,1}(\xi,\eta) = \frac14(1-\xi)(1+\eta), \quad &\phi_{1,1}(\xi,\eta) = \frac14(1+\xi)(1+\eta)  
	\end{split}
\end{align}
on the reference element $\xi,\eta \in [-1,1]$. With this,
we can explicitly evaluate the streamline upwind dissipation terms as 
\begin{equation}\label{eq:SUdiff_general}
	\mathcal{D}(\widetilde{\mathscr{F}},\bar{q}|\mathbf{n})= \frac{\alpha \Delta}{4} \left( \frac{\mathbf{n}_x}{\Delta x}J^x + \frac{\mathbf{n}_y}{\Delta y} J^y \right) \Phi (\widetilde{\mathscr{F}}),
\end{equation}
i.e.,
\begin{align}\label{eq:SUdiff}
	\begin{split}
\mathcal D^{(i,j+1)  }_{i+\frac12,j+\frac12} &= \frac{\alpha\Delta}{4} \left(- \frac{ J^x}{\Delta x} +  \frac{ J^y}{\Delta y}\right) \Phi_{i+\frac12,j+\frac12}, \qquad\quad \mathcal D^{(i+1,j+1)}_{i+\frac12,j+\frac12} = \frac{\alpha\Delta}{4} \left(+ \frac{ J^x}{\Delta x} +  \frac{ J^y}{\Delta y}\right) \Phi_{i+\frac12,j+\frac12}, \\
\mathcal D^{(i,j)    }_{i+\frac12,j+\frac12} &= \frac{\alpha\Delta}{4} \left(- \frac{ J^x}{\Delta x} -  \frac{ J^y}{\Delta y}\right) \Phi_{i+\frac12,j+\frac12}, \qquad\quad \mathcal D^{(i+1,j)  }_{i+\frac12,j+\frac12} = \frac{\alpha\Delta}{4} \left(+ \frac{ J^x}{\Delta x} -  \frac{ J^y}{\Delta y}\right) \Phi_{i+\frac12,j+\frac12},
\end{split}
\end{align}
\new{and where  $\Phi_{i+\frac12,j+\frac12}$ is the dual cell residual}
\begin{equation}\label{eq:rd0}
\Phi_{i+\frac12,j+\frac12} :=\Phi(\widetilde{\mathscr{F}}_{i+\frac12,j+\frac12}):= \mathscr F_{i+1,j+1}-\mathscr F_{i,j+1}-\mathscr F_{i+1,j}+\mathscr F_{i,j}= \int_{\widetilde C_{i+\frac12,j+\frac12}}\partial_{xy} \widetilde{\mathscr F}_{i+\frac12,j+\frac12}\dd x \dd y\;.
\end{equation}

The next sections are devoted to the analysis of some properties of the scheme obtained with the above definitions, as well as some enhancements.

\subsection{Compactness of the method}\label{sec:Compactness}

 Taking into account only the central flux, without the diffusion, one obtains
 \begin{subequations}
 \begin{align}
 \frac{\dd}{\dd t} \bar q_{i,j} &= - \frac{\widehat{\mathscr F}^{(i,j)}_{i+\frac12,j+\frac12} + \widehat{\mathscr F}^{(i,j)}_{i-\frac12,j+\frac12}+ \widehat{\mathscr F}^{(i,j)}_{i+\frac12,j-\frac12} + \widehat{\mathscr F}^{(i,j)}_{i-\frac12,j-\frac12}}{\Delta x \Delta y} \\
 &= - \frac{\overline{\mathscr F}_{i+\frac12,j+\frac12}n^{(i,j)}_{i+\frac12,j+\frac12} + 
 \overline{\mathscr F}_{i-\frac12,j+\frac12}  n^{(i,j)}_{i-\frac12,j+\frac12}  + 
 \overline{\mathscr F}_{i+\frac12,j-\frac12}n^{(i,j)}_{i+\frac12,j-\frac12}  + 
 \overline{\mathscr F}_{i-\frac12,j-\frac12}n^{(i,j)}_{i-\frac12,j-\frac12}
 }{\Delta x \Delta y}\\
 &= - \frac{\overline{\mathscr F}_{i+\frac12,j+\frac12} -
 \overline{\mathscr F}_{i-\frac12,j+\frac12}   - 
 \overline{\mathscr F}_{i+\frac12,j-\frac12} + 
 \overline{\mathscr F}_{i-\frac12,j-\frac12}
 }{\Delta x \Delta y}.
 \end{align}
 \end{subequations}
Define 
\RII{\begin{align}
 \overline{F}_{i+\frac12,j+\frac12} &= \frac14 (F_{i,j} + F_{i+1,j} + F_{i,j+1} + F_{i+1,j+1}) \\
 \overline{G}_{i+\frac12,j+\frac12} &= \frac14 (G_{i,j} + G_{i+1,j} + G_{i,j+1} + G_{i+1,j+1}) \\
\overline{R}_{i+\frac12,j+\frac12} &= \frac14 (R_{i,j} + R_{i+1,j} + R_{i,j+1} + R_{i+1,j+1})
\end{align}}%
such that \RII{$\overline{\mathscr F}_{i+\frac12,j+\frac12} = \overline{F}_{i+\frac12,j+\frac12} + \overline{G}_{i+\frac12,j+\frac12} + \overline{R}_{i+\frac12,j+\frac12}$}, and use the recursions \eqref{eq:Frecursion}--\eqref{eq:Grecursion} to obtain
\begin{align}
 \overline{F}_{i+\frac12,j+\frac12} - \overline{F}_{i+\frac12,j-\frac12} = \frac{\Delta y}{8} (f_{i+1,j+1} + 2 f_{i+1,j} + f_{i+1,j-1}) +  \frac{\Delta y}{8} (f_{i,j+1} + 2 f_{i,j} + f_{i,j-1})  .
\end{align}
We find an analogous formula for $\overline{F}_{i-\frac12,j+\frac12} - \overline{F}_{i-\frac12,j-\frac12}$. 
\RII{
The element source contribution is given from recursion \eqref{eq:integral source} and it reads:
\begin{equation}\label{eq:Rrecursive}
  \begin{split}
&\overline{R}_{i+\frac12,j+\frac12} - \overline{R}_{i-\frac12,j+\frac12} - \overline{R}_{i+\frac12,j-\frac12} + \overline{R}_{i-\frac12,j-\frac12} = \\
&\frac{\Delta x \Delta y}{16} \left(s_{i+1,j+1} +2 s_{i+1,j} +  s_{i+1,j-1}  + 2 s_{i,j+1}+ 4 s_{i,j} + 2 s_{i,j-1} + s_{i-1,j+1}+ 2 s_{i-1,j}  + s_{i-1,j-1}\right)
,
  \end{split}
\end{equation}%
such that in the end
\begin{align}
\frac{\dd}{\dd t} \bar q_{i,j} &=- \frac{1}{\Delta x}\langle  [\![ f_{\cdot,\cdot} ]\!]_i  \rangle_j - \frac{1}{\Delta y} [\![ \langle   g_{\cdot,\cdot}   \rangle_i ]\!]_j + \left \langle \left \langle  s_{\cdot,\cdot}\right \rangle_i \right \rangle_j 
 \end{align}}%
 with the average and jump operators on the cells defined by
 \begin{align}
 \langle a_{i,\cdot} \rangle_j &:= \frac14 (a_{i,j+1} + 2 a_{i,j} + a_{i,j-1}), \label{eq:averagebracket}\\
 [\![ a_{i,\cdot} ]\!]_j &:= \frac12 (a_{i,j+1} - a_{i,j-1}). \label{eq:jumpbracket}
 \end{align}
 One observes that all the global fluxes drop out and the central part of the method is local. 

Next, we turn to the numerical stabilization.
 Defining
 \RII{\begin{align}
  \Phi^F_{i+\frac12,j+\frac12} &= F_{i+1,j+1}- F_{i,j+1}- F_{i+1,j}+ F_{i,j}, \\
  \Phi^G_{i+\frac12,j+\frac12} &= G_{i+1,j+1}- G_{i,j+1}- G_{i+1,j}+ G_{i,j},\\
  \Phi^R_{i+\frac12,j+\frac12} &= R_{i+1,j+1}- R_{i,j+1}- R_{i+1,j}+ R_{i,j}
 \end{align}
 such that $\Phi_{i+\frac12,j+\frac12} = \Phi^F_{i+\frac12,j+\frac12} + \Phi^G_{i+\frac12,j+\frac12}+\Phi^R_{i+\frac12,j+\frac12}$}, and using again the recursions \eqref{eq:Frecursion}--\eqref{eq:Grecursion}, one obtains for every $i,j$
 \begin{align}
  \Phi^F_{i+\frac12,j+\frac12} &= \frac{\Delta y}{2} (f_{i+1,j+1} + f_{i+1,j} - f_{i,j+1} - f_{i,j}).
 \end{align}
Analogously,
 \begin{align}
  \Phi^G_{i+\frac12,j+\frac12} &= \frac{\Delta x}{2} (g_{i+1,j+1} + g_{i,j+1} - g_{i+1,j} - g_{i,j}).
 \end{align}
 and
\RII{\begin{align}
  \Phi^R_{i+\frac12,j+\frac12} &= \frac{\Delta x\Delta y}{4}  \left(s_{i+1,j} + s_{i+1,j+1} + s_{i,j} + s_{i,j+1}\right).
 \end{align}}
 
 One observes again that all the global fluxes drop out \RII{and this shows that the method is completely compact/local with a stencil $3\times 3$.}
 
 The Jacobians involved in the update of $q_{i,j}$ are evaluated at the four corners, such that the method with only the numerical stabilization reads
  \begin{align}
 \begin{split}
 \frac{\dd}{\dd t} \bar q_{i,j} &= - \frac{\mathcal D^{(i,j)}_{i+\frac12,j+\frac12} + \mathcal D^{(i,j)}_{i-\frac12,j+\frac12}+ \mathcal D^{(i,j)}_{i+\frac12,j-\frac12} + \mathcal D^{(i,j)}_{i-\frac12,j-\frac12}}{\Delta x \Delta y} \\
 &= - \frac{1}{\Delta x \Delta y} \frac{\alpha\Delta}{4} \left[  \left(- \frac{ J^x_{i+\frac12,j+\frac12}}{\Delta x} -  \frac{ J^y_{i+\frac12,j+\frac12}}{\Delta y}\right) \Phi_{i+\frac12,j+\frac12} +    \left( \frac{ J^x_{i-\frac12,j+\frac12}}{\Delta x} -  \frac{ J^y_{i-\frac12,j+\frac12}}{\Delta y}\right) \Phi_{i-\frac12,j+\frac12} \right. \\ &  \left.
 + \left(- \frac{ J^x_{i+\frac12,j-\frac12}}{\Delta x} +  \frac{ J^y_{i+\frac12,j-\frac12}}{\Delta y}\right) \Phi_{i+\frac12,j-\frac12}
 +\left( \frac{ J^x_{i-\frac12,j-\frac12}}{\Delta x} +  \frac{ J^y_{i-\frac12,j-\frac12}}{\Delta y}\right) \Phi_{i-\frac12,j-\frac12}
 \right ].
 \end{split}
 \end{align}
 
\RII{We can conclude that the scheme has a local character, as the central part is local after the combination of the four corner fluxes, while the diffusion part is local at each corner residual. This leads to a scheme with a compact $3\times 3$ stencil.}
 
 \subsubsection{Explicit characterization of the method in the linear case}
 To give the spirit of the method, assume for the moment that the Jacobians are evaluated on same state (or that we deal with a linear problem). Then, the numerical stabilization becomes
\RII{\begin{align}
  \begin{split}
\frac{\dd}{\dd t} \bar q_{i,j} &=  - \frac{1}{\Delta x \Delta y} \frac{\alpha\Delta}{4} \left[   
\frac{ J^x}{\Delta x}\left( - \Phi_{i+\frac12,j+\frac12}+ \Phi_{i-\frac12,j+\frac12}-\Phi_{i+\frac12,j-\frac12}+\Phi_{i-\frac12,j-\frac12}\right )
\right . \\ &\qquad \qquad + \left.\frac{ J^y}{\Delta y}\left( -\Phi_{i+\frac12,j+\frac12}   -\Phi_{i-\frac12,j+\frac12} + \Phi_{i+\frac12,j-\frac12}+\Phi_{i-\frac12,j-\frac12} \right ) \right] 
  \end{split}\\
\begin{split}
  &=  \frac{\alpha\Delta}{4} \left[   
\frac{ J^x}{\Delta x^2} \langle f_{i+1,\cdot} - 2 f_{i,\cdot} + f_{i-1, \cdot} \rangle_j
+\frac{ J^x}{\Delta x \Delta y} [\![ \, [\![ g_{\cdot,\cdot} ]\!]_i\,]\!]_j + \frac{J^x}{\Delta x} \langle [\![ s_{\cdot,\cdot} ]\!]_i \rangle_j  \label{eq:diffusionmultidnonlinear}
\right . \\ &\qquad  +  \left.
\frac{ J^y}{\Delta x \Delta y} [\![ \, [\![ f_{\cdot,\cdot} ]\!]_i\,]\!]_j
+\frac{ J^y}{\Delta y^2} \langle g_{\cdot,j+1} - 2 g_{\cdot,j} + g_{\cdot,j-1} \rangle_i+ \frac{J^y}{\Delta y}  [\![ \langle s_{\cdot,\cdot}  \rangle_i ]\!]_j
\right ].
\end{split}
\end{align}
 Here, the average $\langle \cdot \rangle$ and jump $ [\![ \cdot ]\!]$ operators introduced in \eqref{eq:averagebracket} and \eqref{eq:jumpbracket} have been used again.}

Finally, the method can be expressed in classical flux form
\begin{align}
 \frac{\dd}{\dd t} q_{i,j} + \frac{\hat f_{i+\frac12,j} - \hat f_{i-\frac12,j}}{\Delta x} + \frac{\hat g_{i,j+\frac12} - \hat g_{i,j-\frac12}}{\Delta y} = \RII{ \left \langle \left \langle  s_{\cdot,\cdot}\right \rangle_i \right \rangle_j }.
\end{align}
This demonstrates that additionally to the notion \eqref{eq:conservation}, the method is also conservative in the classical sense. The numerical flux through the edge $(i+\frac12,j)$ reads
\begin{align}
 \hat f_{i+\frac12,j} = \frac12 \langle f_{i+1,\cdot} + f_{i,\cdot} \rangle_j - \frac{\alpha \Delta }{2}  \frac{J^x_{i+\frac12,j+\frac12} \Phi_{i+\frac12,j+\frac12} + J^x_{i+\frac12,j-\frac12} \Phi_{i+\frac12,j-\frac12}}{2\Delta x \Delta y}.
\end{align}
In a quasi-1D situation, i.e. when nothing depends on $j$ and when $g=0$, the flux is
\RII{\begin{align}
 \hat f_{i+\frac12} = \frac12( f_{i+1} + f_{i} ) - \frac{\alpha }{2} J^x_{i+\frac12} \left(f_{i+1}  - f_{i} -\frac{\Delta x}{2}(s_{i+1}+s_i)\right).
\end{align}}%
With this, next we discuss the interplay between the numerical stabilization and stationarity preservation.

\subsection{Analysis of the method for the linear acoustic system}\label{sec:linearanalysis}

In this section, we focus on the analysis of the new numerical method by analyzing the numerical diffusion that allows to achieve stationarity preservation, and obtaining an energy estimate. 

\subsubsection{Numerical diffusion and stationarity preservation}

We start by considering the linear acoustic system, but similar results can be shown for nonlinear problems.
A classical dimensionally split finite volume scheme with a local Lax-Friedrichs numerical flux provides the following
discretization of the linear acoustic system: 
\begin{align}
	\begin{split}
  \frac{\dd}{\dd t} p + D_x u + D_y v  &= \frac{\lambda_m \Delta x}{2} D_{xx} p + \frac{\lambda_m \Delta y}{2} D_{yy} p,\\
  \frac{\dd}{\dd t} u + D_x p                &= \frac{\lambda_m \Delta x}{2} D_{xx} u,\\
  \frac{\dd}{\dd t} v + D_y p                &= \frac{\lambda_m \Delta y}{2} D_{yy} v,
	\end{split}
\end{align}
where $D$ represent the discrete derivative operators given by
\begin{align}
  (D_x q)_{i,j} = \frac{q_{i+1,j}-q_{i-1,j}}{2\Delta x},\qquad
  (D_{xx} q)_{i,j} = \frac{q_{i+1,j}-2q_{i,j}+q_{i-1,j}}{\Delta x^2},
\end{align}
and similarly for $D_y$ and $D_{yy}$. 
As can be noticed, the stencil used in this discretization is a simple 5-points stencil.

Contrary to this, the stencils involved in the new first order global flux method, equipped 
with SUPG corner fluxes as described above, includes the cell itself and its eight neighbors (9-points stencil):
\begin{equation}\label{eq:GFDiffusionAcoustic}
	\begin{split}
  \frac{\dd}{\dd t} p + \bar D_x u + \bar D_y v  &= \frac{\alpha \Delta}{2} ( \bar D_{xx} p + \bar D_{yy} p),\\
  \frac{\dd}{\dd t} u + \bar D_x p          &= \frac{\alpha \Delta}{2} \left(\bar D_{xx} u +  D_{xy} v \right),\\
  \frac{\dd}{\dd t} v + \bar D_y p          &= \frac{\alpha \Delta}{2} \left(D_{xy} u + \bar D_{yy} v \right),
	\end{split}
\end{equation}%
\begin{figure}
	\begin{center}
		\subfigure[$\Delta x \bar D_x$]{
			\begin{tikzpicture}
				\draw[thick] (0,0) grid [step=1] (3,3);
				
				\node at (-0.5,2.5) {\( j+1 \)};
				\node at (-0.5,1.5) {\( j \)};
				\node at (-0.5,0.5) {\( j-1 \)};
				\node at (2.5,-0.5) {\( i+1 \)};
				\node at (1.5,-0.5) {\( i \)};
				\node at (0.5,-0.5) {\( i-1 \)};
				\node at (0.5,2.5) {\( -\frac18 \)};
				\node at (1.5,2.5) {\(  0       \)};
				\node at (2.5,2.5) {\(  \frac18 \)};
				
				\node at (0.5,1.5) {\( -\frac14 \)};
				\node at (1.5,1.5) {\(  0       \)};
				\node at (2.5,1.5) {\(  \frac14 \)};
				
				\node at (0.5,0.5) {\( -\frac18 \)};
				\node at (1.5,0.5) {\(  0       \)};
				\node at (2.5,0.5) {\(  \frac18 \)};
				
			\end{tikzpicture}
		}\hspace{5mm}
		\subfigure[$\Delta x^2\bar D_{xx}$]{
			\begin{tikzpicture}
				\draw[thick] (0,0) grid [step=1] (3,3);
				
				\node at (-0.5,2.5) {\( j+1 \)};
				\node at (-0.5,1.5) {\( j \)};
				\node at (-0.5,0.5) {\( j-1 \)};
				\node at (2.5,-0.5) {\( i+1 \)};
				\node at (1.5,-0.5) {\( i \)};
				\node at (0.5,-0.5) {\( i-1 \)};
				\node at (0.5,2.5) {\(  \frac14 \)};
				\node at (1.5,2.5) {\( -\frac12 \)};
				\node at (2.5,2.5) {\(  \frac14 \)};
				
				\node at (0.5,1.5) {\(  \frac12 \)};
				\node at (1.5,1.5) {\(    -1    \)};
				\node at (2.5,1.5) {\(  \frac12 \)};
				
				\node at (0.5,0.5) {\(  \frac14 \)};
				\node at (1.5,0.5) {\( -\frac12 \)};
				\node at (2.5,0.5) {\(  \frac14 \)};      
			\end{tikzpicture}
		}\hspace{5mm}
		\subfigure[$\Delta x\Delta y D_{xy}$]{
			\begin{tikzpicture}
				\draw[thick] (0,0) grid [step=1] (3,3);
				
				\node at (-0.5,2.5) {\( j+1 \)};
				\node at (-0.5,1.5) {\( j \)};
				\node at (-0.5,0.5) {\( j-1 \)};
				\node at (2.5,-0.5) {\( i+1 \)};
				\node at (1.5,-0.5) {\( i \)};
				\node at (0.5,-0.5) {\( i-1 \)};
				\node at (0.5,2.5) {\( -\frac14 \)};
				\node at (1.5,2.5) {\(  0       \)};
				\node at (2.5,2.5) {\(  \frac14 \)};
				
				\node at (0.5,1.5) {\(  0 \)};
				\node at (1.5,1.5) {\(  0 \)};
				\node at (2.5,1.5) {\(  0 \)};
				
				\node at (0.5,0.5) {\(  \frac14 \)};
				\node at (1.5,0.5) {\(  0       \)};
				\node at (2.5,0.5) {\( -\frac14 \)};
				
			\end{tikzpicture}
		}
	\end{center}
	\caption{Finite difference-like stencils for global flux differential operators.}\label{fig:FDstencils}
\end{figure}
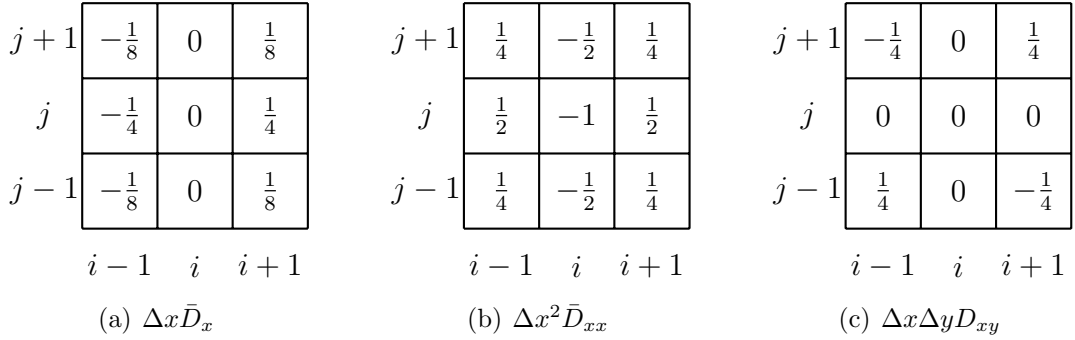%
where the new finite difference operators are given by the stencils in figure~\ref{fig:FDstencils}. $\bar D_x$ and $\bar D_{xx}$ are standard discrete first and second order derivatives in $x$, but including a particular averaging in $y$ direction, introduced in \eqref{eq:averagebracket} as $\langle a_{i,\cdot} \rangle_j := \frac14 (a_{i,j+1} + 2 a_{i,j} + a_{i,j-1})$. These operators have first appeared in \cite{morton2001vorticity} and then in virtually all subsequent works on stationarity and vorticity preservation for linear acoustics on Cartesian grids, e.g. in \cite{morton2001vorticity,sidilkover2002factorizable,jeltsch2006curl,mishra2011constraint,lung14,barsukow2019stationarity}.

In \cite{barsukow2025structure} these finite difference operators appeared naturally as Kronecker products of uni-direction operators: $D_{xy} = D_x \otimes D_yI_y $, $\bar D_{xx} = D_{xx} \otimes D_yI_y$ etc., with matrices $I_y$ responsible for the integration and $D_yI_y$ being the particular averaging matrix corresponding to $\langle \cdot \rangle$. 

One observes that the diffusion operators for the velocity no longer depend on second derivatives of individual components,
but instead on the gradient of the divergence operator.
Although this characteristic of the scheme is more readily visible for a simplified model like the linear acoustic
system \eqref{eq:GFDiffusionAcoustic}, similar considerations can be drawn also for more complex nonlinear systems, as is obvious from \eqref{eq:diffusionmultidnonlinear}.

\begin{remark}[Numerical diffusion for nonlinear problems]
For the shallow water equations, the first order global flux method with SUPG corner fluxes leads to the following discrete evolution equations:
\begin{align}\label{eq:GFdiffusionSW}
  \begin{split}
  \frac{\dd}{\dd t} h + \bar{D}_x f_h + \bar D_y g_h  &= \frac{\alpha \Delta}{2} ( \bar{D}_{xx} f_{hu} + D_{xy} g_{hu} ) + \frac{\alpha \Delta}{2} (D_{xy} f_{hv} + \bar{D}_{yy} g_{hv}) ,\\
  \frac{\dd}{\dd t} hu+ \bar{D}_x f_{hu} + \bar{D}_y g_{hu} &= \frac{\alpha \Delta}{2} (g\bar h -\bar u^2)  (\bar{D}_{xx} f_{h} + D_{xy} g_{h}) - \frac{\alpha \Delta}{2} \bar u \bar v (D_{xy} f_{h} + \bar{D}_{yy} g_h) +\\
  &\hspace{-3cm} +\alpha \Delta \bar u (\bar{D}_{xx} f_{hu} + D_{xy} g_{hu}) + \frac{\alpha \Delta}{2}\bar v (D_{xy} f_{hu} + \bar{D}_{yy} g_{hu}) + \frac{\alpha \Delta}{2} \bar u (D_{xy} f_{hv} + \bar{D}_{yy} g_{hv}) ,\\
  \frac{\dd}{\dd t} hv+ \bar{D}_x f_{hv} + \bar{D}_y g_{hv} &= \frac{\alpha \Delta}{2} (g\bar h -\bar v^2)  (D_{xy} f_{h} + \bar{D}_{yy} g_{h}) - \frac{\alpha \Delta}{2} \bar u \bar v (\bar{D}_{xx} f_{h} + D_{xy} g_h) +\\
  &\hspace{-3cm} +\alpha \Delta \bar v (D_{xy} f_{hv} + \bar{D}_{yy} g_{hv}) + \frac{\alpha \Delta}{2}\bar v (\bar{D}_{xx} f_{hu} + D_{xy} g_{hu}) + \frac{\alpha \Delta}{2} \bar u (\bar{D}_{xx} f_{hv} + D_{xy} g_{hv}),
  \end{split}
\end{align}
where we considered constant Jacobians defined in an average state $\bar{q}=(\bar h, \bar h \bar{u}, \bar h\bar v)$ to regroup the terms in a compact form.
For notational convenience, we have introduced the terms $f_{hu}$ and $g_{hu}$ to denote the fluxes for the momentum equation in $hu$, and similarly for the other equations.  
Again, we obtain diffusion terms that depends on the gradient of the divergence operator, which is essential for stationarity preservation.

\end{remark}

\subsubsection{Semi-discrete energy stability}\label{sec:energy}
In this section, we focus on the semi-discrete energy estimates for the linear acoustic system.
In particular, in the continuous setting, it can be easily proven that the conserved energy of the system is
$\mathcal{E} = \frac{u^2+v^2}{2}+\frac{p^2}{2}$, by multiplying \eqref{eq:LinAc} by $q^T$, 
summing the three equations and integrating over the whole domain $\Omega$:
\begin{align}
  \int_\Omega [q^T\del_t q + q^T \mathbf{J}\nabla q] \dd \vec x &=\int_\Omega [u\del_t u + u\del_x p + v\del_t v + v\del_y p + p\del_t p + p\del_x u + p\del_y v] \dd \vec x \nonumber\\
&= \frac{\dd }{\dd t}\int_\Omega \left[\frac{u^2+v^2}{2}+\frac{p^2}{2}\right]\dd \vec x + \int_{\del\Omega} p\mathbf{v}\cdot\mathbf{n}\dd S
\end{align}
where the second term is zero for periodic boundary conditions.
In our discrete framework, we would like to prove that 
$$ \frac{\dd }{\dd t}\int_\Omega \left[\frac{u^2+v^2}{2}+\frac{p^2}{2}\right]\dd x \leq  0 .$$

To do that, we will write the new differential operators introduced above in a
tensor product form to split the contributions from the two dimensions, for more details see \cite{barsukow2025structure}, 
\begin{align*}
  \bar D_x &= \left(D_{+}M_{-}\right)\otimes\left(M_{+}M_{-}\right), &
  \bar{D}_{xx} &= \left(D_{+}D_{-}\right)\otimes\left(M_{+}M_{-}\right) ,\\
  \bar D_{yy}& = \left(M_{+}M_{-}\right) \otimes\left(D_{+}D_{-}\right),&
  D_{xy} &=  \left(D_{+}M_{-}\right)\otimes\left(D_{+}M_{-}\right)
\end{align*}
where the derivative, $D$, and average, $M$, operators are defined as
\begin{equation}
  D_{+} = \begin{bmatrix} -1 & 1 &0 & \ldots& \ldots \\ 0 &-1 & 1 & \ldots& \ldots \\ & \ddots & \ddots & \ddots \\ \ldots& \ldots & 0 & -1 &1 \\ 1 & \ldots&\ldots & 0 & -1\end{bmatrix},\quad
  M_{+} = \begin{bmatrix} \frac12 & \frac12 &0 & \ldots& \ldots \\ 0 &\frac12 & \frac12 & \ldots& \ldots \\ & \ddots & \ddots  & \ddots \\  \ldots &\ldots & 0 & \frac12 &\frac12 \\ \frac12 & \ldots &\ldots & 0 &\frac12 \end{bmatrix}, 
\end{equation}
with periodic boundary conditions, and by 
$$D_{-} = -D_{+}^T =: D \qquad\text{and}\qquad M_{-} = M_{+}^T =: M.$$

%
\begin{proposition}[Semi-discrete energy inequality]
The following semi-discrete energy inequality holds, 
\begin{equation}  
  {\normalfont\frac{\dd }{\dd t}} \sum_{i,j} \mathcal{E}_{i,j} \leq  0.
\end{equation}
\end{proposition}
\begin{proof}
The central part of the method preserves the energy since, e.g. $u^T \bar D_x p + p^T \bar D_x u = u^T \bar D_x p + u^T \bar D_x^T p = 0$ up to boundary terms. The evolution of the energy of the system is thus entirely given by the numerical stabilization as follows:
\begin{align}
  \frac{2}{\alpha \Delta}\frac{\dd }{\dd t} \sum_{i,j} \mathcal{E}_{i,j} &=  p^T \bar D_{xx} p + p^T \bar D_{yy} p +  u^T \bar D_{xx} u + u^T D_{xy} v + v^T D_{xy} u + v^T \bar D_{yy} v, 
\end{align}
where the terms on the right-hand side can be recast as
\begin{align*}
p^T \bar D_{xx} p &=  p^T \left(D_{+}D_{-}\right)\otimes\left(M_{+}M_{-}\right) p = -\|(D\otimes M)p\|^2\leq 0, \\
p^T \bar D_{yy} p &=  p^T \left(M_{+}M_{-}\right)\otimes\left(D_{+}D_{-}\right) p = -\|(M\otimes D)p\|^2\leq 0, \\   
u^T \bar D_{xx} u &=  u^T \left(D_{+}D_{-}\right)\otimes\left(M_{+}M_{-}\right) u = -\left[(D\otimes M)u\right]^T\left[(D\otimes M)u\right],  \\
u^T D_{xy} v          &=  u^T \left(D_{+}M_{+}\right)\otimes\left(M_{-}D_{-}\right) v = -\left[(D\otimes M)u\right]^T\left[(M\otimes D)v\right],  \\
v^T D_{xy} u          &=  v^T \left(M_{+}D_{+}\right)\otimes\left(D_{-}M_{-}\right) u = -\left[(M\otimes D)v\right]^T\left[(D\otimes M)u\right],  \\
v^T \bar D_{yy} v &=  v^T \left(M_{+}M_{-}\right)\otimes\left(D_{+}D_{-}\right) v = -\left[(M\otimes D)v\right]^T\left[(M\otimes D)v\right],   
\end{align*}
where the mixed operator was manipulated thanks to $MD=DM$.
Hence, the semi-discrete energy is found to decrease:
\begin{align*}
  \frac{2}{\alpha \Delta} \frac{\dd }{\dd t} \sum_{i,j} \mathcal{E}_{i,j}  \leq & - \left[(D\otimes M)u\right]^T\left[(D\otimes M)u+(M\otimes D)v\right]\\&  -\left[(M\otimes D)v\right]^T\left[(D\otimes M)u+(M\otimes D)v\right]\\
  =& -\|(D\otimes M)u+(M\otimes D)v\|^2\leq 0.
\end{align*} 

This is a discrete version of $\displaystyle \int_{\Omega} \vec v \cdot \nabla (\nabla \cdot \vec v) \, \dd \vec x = -  \int_{\Omega} (\nabla \cdot \vec v)^2 \dd \vec x + \text{boundary terms}$.

\end{proof}

\subsection{Analogy with Residual Distribution and discrete steady states}\label{sec:consistency}
The recent work of \cite{gaburro2025multidRS} has provided a general analysis of the relations between multi-dimensional
finite volume methods  with point fluxes and residual distribution schemes.  Earlier, it has been shown in \cite{Abgrall2022}
that residual distribution schemes can be reformulated in terms of a global flux finite volume method.
This section elaborates on these aspects for   the global flux finite volume approach proposed here. This allows us to
give more details on the discrete steady states of the method in a more general setting. 

Following the last reference we start from the conservation condition \eqref{eq:conservation} at each corner. 
Consider, instead of \eqref{eq:diff}, the following ansatz for the numerical global flux, given by
the trace of the cell (global) flux plus a fluctuation:
\begin{equation}\label{eq:rd1}
  \widehat{\mathscr F}^{(i+\ell,j+m)  }_{i+\frac12,j+\frac12} =    \mathscr F_{i,j}  n^{(i+\ell,j+m)}_{i+\frac12,j+\frac12} + \Phi^{(i+\ell,j+m)  }_{i+\frac12,j+\frac12} , \qquad \ell,m \in \lbrace 0,1 \rbrace.
\end{equation}
All the  properties of the numerical flux can be translated into requirements
on the fluctuations $\Phi^{(i+\ell,j+m)  }_{i+\frac12,j+\frac12}$. The  most interesting ones are related to conservation and stationarity preservation.
Corner conservation is written by using the above ansatz in \eqref{eq:conservation}, which leads to the requirement
$$
\Phi^{(i,j)  }_{i+\frac12,j+\frac12}  + \Phi^{(i+1,j)  }_{i+\frac12,j+\frac12} + \Phi^{(i+1,j+1)  }_{i+\frac12,j+\frac12}  + \Phi^{(i,j+1)  }_{i+\frac12,j+\frac12}  
=  - \sum\limits_{\ell, m\in\{0,1\}}
 \mathscr F_{i+\ell,j+m}  n^{(i+\ell,j+m)  }_{i+\frac12,j+\frac12}  = -  \Phi_{i+\frac12,j+\frac12},
$$
where $\Phi_{i+\frac12,j+\frac12}$ is the global flux integral on the corner dual cell as defined in \eqref{eq:rd0}. 
By virtue of \eqref{eq:rd1}, defining a corner flux is  thus equivalent to defining a residual distribution scheme satisfying 
\begin{equation}\label{eq:rd2}
\sum\limits_{\ell, m\in\{0,1\}}\Phi^{(i+\ell,j+m)  }_{i+\frac12,j+\frac12}  =-
\Phi_{i+\frac12,j+\frac12} .
\end{equation}
This analogy goes much further, and it is in fact a full equivalence.
In particular, we can  prove the following facts.
\begin{proposition}[Equivalence with RD]\label{prop:RDeq} Consider the multi-dimensional global flux finite volume method \eqref{eq:GF2d1}, with numerical fluxes
written in terms of fluctuations \eqref{eq:rd0}. Then,
\begin{enumerate}
\item the multidimensional finite volume global flux method \eqref{eq:GF2d1} with piecewise constant data is equivalent to the Residual Distribution scheme
\begin{equation}\label{eq:rd}
\Delta x \Delta y \frac{\dd}{\dd t} \bar q_{i,j} + \Phi^{(i,j)  }_{i+\frac12,j+\frac12} + \Phi^{(i,j)  }_{i+\frac12,j-\frac12} + \Phi^{(i,j)  }_{i-\frac12,j+\frac12}+\Phi^{(i,j)  }_{i-\frac12,j-\frac12} =0\,,
\end{equation}
with  fluctuations  $ \Phi^{(i,j)  }_{i\pm\frac12,j\pm\frac12}$ verifying the conservation  condition   \eqref{eq:rd2} at each corner $(i\pm\frac12,j\pm\frac12)$;
\item the finite volume global flux method \eqref{eq:GF2d1} with   average flux $   \widehat{\mathscr F}^{(i+\ell,j+m) }_{i+\frac12,j+\frac12}   =\overline{\mathscr F}_{i+\frac12,j+\frac12} n^{(i+\ell,j+m)  }_{i+\frac12,j+\frac12},$ $\forall \,\ell,m \in \{0,1\}$ is equivalent to the residual distribution scheme with   fluctuations  
$$
\Phi^{(i+\ell,j+m)  }_{i+\frac12,j+\frac12} =  ( \overline{\mathscr F}_{i+\frac12,j+\frac12}-  \mathscr F_{i+\ell,j+m} )n^{(i+\ell,j+m) }_{i+\frac12,j+\frac12}
$$
\item  the finite volume method including the numerical dissipation in \eqref{eq:diff}, defined by the streamline upwind terms \eqref{eq:SUdiff}, is equivalent to the  
 residual distribution scheme  defined by 
$$
\Phi^{(i+\ell,j+m)  }_{i+\frac12,j+\frac12} =  ( \overline{\mathscr F}_{i+\frac12,j+\frac12}-  \mathscr F_{i+\ell,j+m} )n^{(i+\ell,j+m) }_{i+\frac12,j+\frac12}  -\dfrac{1}{4}  \delta^{(i+\ell,j+m)  }_{i+\frac12,j+\frac12} \Phi_{i+\frac12,j+\frac12}\;,\quad \ell,m\in\{0,1\}
$$
with $\delta^{(i+\ell,j+m)  }_{i+\frac12,j+\frac12}=\alpha\Delta\left( (-1)^\ell \frac{J^x}{\Delta x} + (-1)^m \frac{J^y}{\Delta y} \right )$ according to \eqref{eq:SUdiff};
\item both the centered and the stabilized method are steady state  preserving with respect to global fluxes of the type \eqref{eq:steady_state_global_flux}. 
In particular, they both admit  discrete steady solutions verifying
$$
 \Phi_{i+\frac12,j+\frac12} =0 \quad \forall i,j \;;
$$ 
\item both the centered and the stabilized method are formally second order accurate at steady state for smooth enough solutions.
\end{enumerate}
\begin{proof}
The first fact is a consequence of the identity $n^{(i,j)}_{i+\frac12,j+\frac12}+n^{(i,j)}_{i-\frac12,j+\frac12} +n^{(i,j)}_{i+\frac12,j-\frac12}+n^{(i,j)}_{i-\frac12,j-\frac12}=0$,
so when  replacing  \eqref{eq:rd1} in \eqref{eq:GF2d1} we obtain immediately \eqref{eq:rd}: Using \eqref{eq:rd} or \eqref{eq:GF2d1}  
to represent the scheme is absolutely equivalent. The second and third properties are obtained by simply subtracting from  the average flux, and from  \eqref{eq:diff} 
the cell contributions to obtain the corresponding fluctuations.\\

Concerning the discrete kernel property, consider first
 the average flux without extra dissipation,   and compute explicitly  the  sum of the corner contributions:
\begin{align}
\overline{\mathscr F}_{i+\frac12,j+\frac12} n^{(i,j)  }_{i+\frac12,j+\frac12} &+\overline{\mathscr F}_{i+\frac12,j-\frac12} n^{(i,j)  }_{i+\frac12,j-\frac12}
+\overline{\mathscr F}_{i-\frac12,j+\frac12} n^{(i,j)  }_{i-\frac12,j+\frac12}+\overline{\mathscr F}_{i-\frac12,j-\frac12} n^{(i,j)  }_{i-\frac12,j-\frac12} \nonumber \\
  &= \dfrac{ \mathscr F_{i+1,j+1} - \mathscr F_{i+1,j-1} - \mathscr F_{i-1,j+1} +  \mathscr F_{i-1,j-1}}{4} \\
 &=
 \frac14 \left( \Phi_{i+\frac12,j+\frac12} + \Phi_{i-\frac12,j+\frac12} + \Phi_{i+\frac12,j-\frac12} + \Phi_{i-\frac12,j-\frac12} \right) \label{eq:eqq}.
\end{align}
This shows that  $ \Phi_{i+\frac12,j+\frac12} =0$ is   in the kernel of the average flux method. The same is  true for the stabilized scheme for which we can write 
after assembly  around cell $i,j$  
\begin{equation}\label{eq:eqqq}
\begin{aligned}
\sum_{\ell,m\in\{-1,1\}} \left[  \overline{\mathscr F}_{i+ \frac\ell2,j+ \frac{m}2} n^{(i,j)  }_{i+ \frac\ell2,j+ \frac{m}2}  -\dfrac{1}{4}  \delta^{(i,j)  }_{i+ \frac\ell2,j+\frac{m}2} 
 \Phi_{i+ \frac\ell2,j+\frac{m}2} \right] \\=
\sum_{\ell,m\in\{-1,1\}}\dfrac{1}{4} \left[ \left(\mathrm{I} -  \delta^{(i,j)  }_{i+ \frac\ell2,j+\frac{m}2} \right)
 \Phi_{i+ \frac\ell2,j+\frac{m}2} \right] 
\end{aligned}
\end{equation}
%
%
%
%
%
%
So to prove point 4., we just   check that  $ \Phi_{i+\frac12,j+\frac12} =0$ is also a consequence of  the  steady state preservation condition of Remark~\ref{rem:steady_state_subspace} in \eqref{eq:ssp-cornerF_up_to_ij}. For $\mathscr F_{i,j} = \mathscr F+\alpha_i+\beta_j$ as in \eqref{eq:steady_state_global_flux}, in each dual cell
$$\Phi_{i+\frac12,j+\frac12}=\mathscr F+\alpha_{i+1}+\beta_{j+1}-\mathscr F-\alpha_{i}-\beta_{j+1}-\mathscr F-\alpha_{i+1}-\beta_{j}+\mathscr F+\alpha_{i}+\beta_{j}=0$$
  and thus  $\widehat{\mathscr F}^{(i+\ell,j+m) }_{i+\frac12,j+\frac12}=\mathscr F$.   Finally, the second-order consistency at steady state
is a consequence of results for residual distribution schemes with bounded distribution coefficients on linear and bi-linear elements 
(see \cite{RD-ency,HDR-MR,AR:17} and references therein).
Second-order accuracy at steady state thus follows from the boundedness of the coefficients  
$1/4$, and  $( \mathrm{I} -\delta^{(i+\ell,j+m)  }_{i+\frac12,j+\frac12})/4$ which appear in the equivalent forms of the scheme \eqref{eq:eqq} and \eqref{eq:eqqq}.
\end{proof}
\end{proposition}
 The switch from conservative finite volume to  residual distribution methods contains some nuances on which we would like to comment. In particular,
 the proof uses   two different writings of the average flux  scheme. This may be  confusing as to which is the proper form to use.
 The confusion is originated from the fact that for this simple case, due to cancellation, 
 the same global assembly can be obtained from several different  local  contributions, not all fitting into the same conservation framework.
 For example, the proof shows that the scheme
\begin{align} \label{eq:wrongcentralRD}
 \Delta x \Delta y \frac{\dd}{\dd t} \bar q_{i,j} +   \dfrac{\Phi_{i+\frac12,j+\frac12}}{4} + \dfrac{\Phi_{i+\frac12,j-\frac12}}{4} + \dfrac{\Phi_{i-\frac12,j+\frac12}}{4}+\dfrac{\Phi_{i-\frac12,j-\frac12}}{4} =0\,
 \end{align}
 is equivalent to the average flux scheme.  In light of \eqref{eq:rd}, this may lead to the conclusion, that  the definition $  \Phi_{i+\frac12,j+\frac12}^{(i,j)}=  \Phi_{i+\frac12,j+\frac12}/4 $ 
 is a viable one; which is, however, wrong. Such a definition 
    is not acceptable   here, as it does not satisfy the  local conservation constraint  \eqref{eq:rd2}, which has a minus on the right hand side. 
    This change of sign is related to the difference between internal and exterior oriented normals.
    Scheme   \eqref{eq:wrongcentralRD} can also be shown to be conservative in the classical cell-vertex residual distribution framework, with  the appropriate  conservation constraint. 
    On Cartesian meshes, this mismatch cancels out and  one ends up with the same  discretization  after assembly.
 However, on general meshes the change in sign
must be carefully accounted for (see e.g. \cite{gaburro2025multidRS}), and the appropriate local conservation and consistency conditions respected. 
In particular, the   central  RD   \eqref{eq:wrongcentralRD} can indeed  be written in a flux form, but the numerical flux has a much more involved expression
than the simple average flux. 
The interested reader can refer to the last reference,  and to \cite{Abgrall2022} for more details. 

\subsection{Source modification to  preserve  solutions at rest}\label{sec:sourceSW}

This subsection has the goal of providing a direct way to embed bathymetry source terms typical of the shallow water system with bottom topography. 
For simplicity, this part revolves around this PDE system, but the same approach can be also applied for other cases. 
In  particular, the goal is to achieve stationarity preservation for motionless equilibria, 
i.e.\ lake at rest preservation, coming from the balance between the hydrodynamic pressure and bottom topography, 
which is present in both 1D and 2D configurations:
\begin{equation}\label{eq:lakeatrest}
  \text{1D:}
  \begin{cases}
    h(x) + b(x) \equiv \eta_0,\\
    u(x)   \equiv 0 ,
  \end{cases}\qquad
  \text{2D:}
  \begin{cases}
    h(x,y) + b(x,y) \equiv \eta_0,\\
    u(x,y) = v(x,y)  \equiv 0. 
  \end{cases}
\end{equation}
As also shown in section \ref{sec:GF1d}, in a 1D global flux framework (\cite{ciallella2023arbitrary}) it has been proposed to integrate the source terms within the flux derivative, thus obtaining a quasi-conservative formulation of the PDE:
\begin{equation}\label{eq:1deqsource}
  \begin{cases}
    \partial_t h   + \partial_x q_x &= 0,\\
    \partial_t q_x + \partial_x \left(\frac{q^2_x}{h} + g\frac{h^2}{2}\right)  &= -gh\partial_x b, \\
  \end{cases} \quad\Longrightarrow\quad
  \begin{cases}
    \partial_t h   + \partial_x q_x  &= 0,\\
    \partial_t q_x + \partial_x \left(\frac{q^2_x}{h} + g\frac{h^2}{2}+\int^x gh\partial_\xi b \dd \xi\right)  &= 0. \\
  \end{cases} 
\end{equation}
However, contrary to classical source terms, the bathymetry term shall be treated differently
given the presence of its derivative. 
To achieve consistency and well-balancedness for the high order method, in \cite{ciallella2023arbitrary}
the integral of the bathymetry source is considered to jump at each interface. 
In the current low order framework the approach amounts to
\begin{align}\label{eq:sourcequad}
  R^x_i = \int^{x_i} gh\partial_\xi b \;\dd \xi &= \int^{x_{i-1}} gh\partial_\xi b \dd \xi + \int_{x_{i-1}}^{x_{i}} gh\partial_\xi b \dd \xi  = R^x_{i-1} + g \frac{h_{i}+h_{i-1}}{2} (b_{i}-b_{i-1}),   
\end{align}
where the second integral has been computed using a consistent approximation of $\partial_\xi b$, 
while for the $h$ a simple trapezoidal rule has been used.

The same approach can also be developed for 2D systems, including the the source term containing $\partial_x$ in the $x$-flux, 
and the one containing $\partial_y$ in the $y$-flux. Starting from equation \eqref{eq:SWE}, we can integrate the source terms present in the momentum equation as follows:
\begin{align}\label{eq:SWEsource}
  \begin{cases}
    \partial_t h   + \partial_x (hu) + \partial_y (hv) = 0\\
    \partial_t (hu) + \partial_x \left(hu^2 + g\frac{h^2}{2} + \int^x gh\partial_\xi b \;\dd \xi\right) + \partial_y \left(huv\right) = 0, \\
    \partial_t (hv) + \partial_x \left(huv\right) + \partial_y \left(hv^2 + g\frac{h^2}{2}+ \int^y gh\partial_\eta b \;\dd \eta\right) = 0.    
  \end{cases}    
\end{align}

\begin{proposition}[Lake at rest preservation]
  The 2D global flux scheme of the system \eqref{eq:SWEsource} with the source term quadrature provided in 
  equation \eqref{eq:sourcequad} is exactly well-balanced for the lake at rest solution \eqref{eq:lakeatrest}. 
\end{proposition}
\begin{proof}
To prove that the family of equilibria \eqref{eq:lakeatrest} is exactly preserved when $u=v\equiv0$ and $\eta=h+b\equiv\eta_0$,
Since the mixed terms depend only on the velocity and thus vanish, the two momentum equations can be treated separately for 
the $x$ and $y$ contribution. In particular we want to show that given a zero velocity and constant free surface elevation, 
$f_{i,j} + R^x_{i,j} = f_{i-1,j} + R^x_{i-1,j},\;\forall j$. Without loss of generality, we will show this result only for the $x$ direction.

By substitution of the relevant quantities, we obtain
\begin{align}
  f_i + R^x_i - f_{i-1} - R^x_{i-1} &= g\frac{h^2_i-h^2_{i-1}}{2} +  g \frac{h_i+h_{i-1}}{2}\left(b_{i}-b_{i-1}\right) \\
  &= g \frac{h_i+h_{i-1}}{2} \left(\eta_i - \eta_{i-1}\right) = 0,\qquad\forall j,
\end{align}
where the last equality holds when $\eta_i = \eta_{i-1}\equiv\eta_0$, with $\eta_i = h_i + b_i$.
\end{proof}

\RII{
\begin{remark}[Compactness of bathymetry source term]
  Following the notation used in section~\ref{sec:Compactness}, we can derive a compact definition of the method also for a source term defined as above. In particular, the central discretization will include in the update equation of $hu_{i,j}$ the following term
  \begin{align}\frac{1}{\Delta x}
    \left \langle g \frac{h_{i+1,\cdot}+h_{i,\cdot}}{2}(b_{i+1,\cdot}-b_{i,\cdot}) +g \frac{h_{i,\cdot}+h_{i-1,\cdot}}{2}(b_{i,\cdot}-b_{i-1,\cdot})\right\rangle_j,
  \end{align}
  which is clearly compact. And similarly, for the equation of $hv$. On the other hand, in the diffusion part, we will have the following term entering the definition of $\Phi^{R,hu}_{i+\frac12,j+\frac12}$
\begin{align}
  \Delta y\left \langle 
      g \frac{h_{i+1,\cdot}+h_{i,\cdot}}{2}(b_{i+1,\cdot}-b_{i,\cdot})
  \right\rangle_{j+\frac12},
\end{align}
where $\langle z_{i+\frac12,\cdot} \rangle_{j+\frac12} =z_{i+\frac12,j+1}+z_{i+\frac12,j}$
and similarly for $\Phi^{R,hv}_{i+\frac12,j+\frac12}$. 
This shows that the whole method, also in the lake-at-rest well-balanced version, has a compact stencil of size $3\times 3$. 

\end{remark}
}%

\subsection{Compatible boundary conditions}\label{sec:bc}

Boundary conditions play an essential role in the practical application of numerical methods.
Let us for example consider the usual ghost cell approach. In a classical dimension-by-dimension 
finite volume method, homogeneous Neumann boundary conditions on the state variables can be simply enforced by copying the state.
This is somewhat consistent with the internal treatment based on  one dimensional  Riemann  fluxes using two states.
However,  this approach cannot be steady state preserving since it is not based on the global flux. One should consider corner fluxes on the boundaries,
and choose the ghost states consistently with  the equilibrium condition. 
Here, we construct compatible Neumann 
boundaries  based on the  discrete constraint  $\Phi_{\text{corner}} = 0$, which we have shown to be the relevant multi-dimensional  characterization of our discrete steady states.  Compatible transmissive conditions require
this relation to be verified by the ghost cells at each  boundary corner.

Let us consider the right boundary domain (see figure \ref{fig:ghostcell}). We can impose $\mathscr F_{N+1,j-1}:=\mathscr F_{N,j-1}$ for all $j$ on the global fluxes, instead of computing the state variables
as for classical homogeneous Neumann conditions. Then, we can use use directly the global flux at the boundary corners. This will lead to the following relation 
\begin{equation} 
  \Phi_{N+\frac12,j-\frac12} = \mathscr F_{N+1,j} - \mathscr F_{N,j} - \mathscr F_{N+1,j-1} + \mathscr F_{N,j-1}=0,
\end{equation} 
which gives a steady state solution.

On corner ghosts, similarly, one has to impose 
$
\mathscr{F}_{N_x+1,N_y+1}:=\mathscr{F}_{N_x,N_y+1},
$
if also on the top side transmissive conditions must be imposed, this will results in 
$$
\mathscr{F}_{N_x+1,N_y+1}=\mathscr{F}_{N_x,N_y+1}=\mathscr{F}_{N_x+1,N_y}=\mathscr{F}_{N_x,N_y}.
$$

Similar ideas can be used also for   other  boundary types, but in our experience transmissive conditions are 
the most critical  to correctly maintain the internal structure, since they involve no external data, which
 provide some link to the correct solution for other boundary conditions.
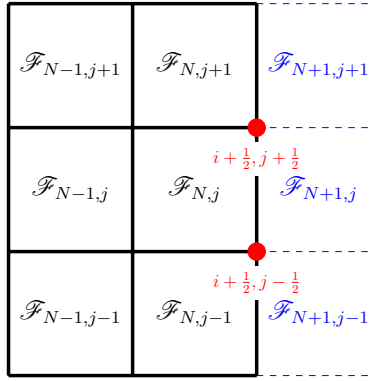
\begin{figure}
  \centering
   \begin{adjustbox}{width=0.3\textwidth}
    \begin{tikzpicture}
        \draw[dashed,blue] (4,0) grid [step=2] (6,6);
        \draw[ultra thick] (0,0) grid [step=2] (4,6);
        \draw[red,fill=red] (4,4) circle[radius=4pt];
        \draw[red,fill=red] (4,2) circle[radius=4pt];
        \node[scale=0.7,red,fill=white] at (4,3.5) {\( i+\frac12,j+\frac12 \)};
        \node[scale=0.7,red,fill=white] at (4,1.5) {\( i+\frac12,j-\frac12 \)};
        
        \node at (1,5) {\( \mathscr F_{N-1,j+1} \)};
        \node at (3,5) {\( \mathscr F_{N,j+1} \)};
        \node[blue] at (5,5) {\( \mathscr F_{N+1,j+1} \)};
        
        \node at (1,3) {\( \mathscr F_{N-1,j} \)};
        \node at (3,3) {\( \mathscr F_{N,j} \)};
        \node[blue] at (5,3) {\( \mathscr F_{N+1,j} \)};
        
        \node at (1,1) {\( \mathscr F_{N-1,j-1} \)};
        \node at (3,1) {\( \mathscr F_{N,j-1} \)};
        \node[blue] at (5,1) {\( \mathscr F_{N+1,j-1} \)};
    \end{tikzpicture}
  \end{adjustbox}
  \caption{Ghost cell labelling: ghost cells are highlighted in blue.}\label{fig:ghostcell}
  \end{figure}

When seeking to preserve steady states, it is crucial to ensure that the number of equations imposed—either by boundary conditions or by the steady state equations—does not exceed the number of unknowns, or that these equations are mutually compatible.

The steady state conditions enforced by $\Phi_{i+\frac12,j+\frac12}$ for $i=0,\dots,N_x$ and $j=0,\dots,N_y$ introduce $N_{eq}(N_x+1)(N_y+1)$ linearly independent constraints on $N_{eq}N_xN_y$ unknowns $\bar q_{i,j}$ (for $i=1,\dots,N_x$, $j=1,\dots,N_y$). To satisfy these extra constraints, $N_{eq}(2N_x + 2N_y + 4)$ ghost cell values are introduced. This leaves $N_{eq}(N_x + N_y + 3)$ equations that can be specified at the boundaries, typically through Dirichlet conditions. Homogeneous Neumann conditions, as previously discussed, are compatible with the internal constraints and do not introduce additional equations.
Therefore, the $N_{eq}(N_x + N_y + 3)$ remaining degrees of freedom correspond to at most two sides where Dirichlet boundary conditions can be imposed, and if not consecutive sides, they should not have all corners included.

In summary, when seeking equilibria,   the boundary conditions must be  compatible with the internal constraints. It is only possible, and necessary, to impose complete Dirichlet conditions on (at most) two boundaries.
In this respect, this work is  changing the perspective on this  issue. For many years, schemes similar to the one obtained here have been considered as flawed due to the existence of the  steady states 
characterized by proposition~\ref{prop:RDeq}. This is due to the fact that spurious oscillating modes  may also satisfy the condition $\Phi_{i+1/2,j+1/2}=0$ $\forall i,\, j$ (see e.g. 
\cite{HDR-MR, AR:17} and references therein).  However, this is only true if one considers the problem locally, which is a wrong way to define multidimensional steady states as they must include boundary conditions. 
If these are imposed in a compatible manner, spurious modes can be controlled.
This work, as well as the work discussed in \cite{barsukow2025structure}, contributes to rectifying this notion.

\section{Standard finite volume scheme used for comparison}\label{sec:FVcomparison}

In this section, we present the standard finite volume (FV) scheme used for comparison with the novel global flux (GF) 
scheme in the numerical experiments presented in section \ref{sec:numerical}.
The classical finite volume formulation for the 2D nonlinear hyperbolic problem \eqref{eq:2Dbalancelaw} can be
written by integrating it in the cell $C_{i,j}$:
\begin{align}
  \frac{\dd}{\dd t} \bar q_{i,j} + \frac{\hat f_{i+\frac12,j}-\hat f_{i-\frac12,j}}{\Delta x} + \frac{\hat g_{i,j+\frac12}-\hat g_{i,j-\frac12}}{\Delta y} = \bar s_{i,j},
\end{align}
where the numerical flux $\hat f_{i+\frac12,j}$ is computed through the local Lax-Friedrichs (or Rusanov) flux:
\begin{align}
  \hat f_{i+\frac12,j}= \frac12\left( f^L_{i+\frac12,j} + f^R_{i+\frac12,j} \right) - \frac{\lambda_m}{2} \left( q^R_{i+\frac12,j} - q^L_{i+\frac12,j} \right),
\end{align}
where 
\begin{align}
  f^L_{i+\frac12,j} &= f(q^L_{i+\frac12,j}),\qquad f^R_{i+\frac12,j} = f(q^R_{i+\frac12,j}), 
\end{align}
and similarly for the others. The source term is computed by integrating the source term over the cell $C_{i,j}$:
\begin{align}
  \bar s_{i,j} = \frac{1}{\Delta x \Delta y} \int_{x_{i-\frac12}}^{x_{i+\frac12}}\int_{y_{j-\frac12}}^{y_{j+\frac12}} s \dd x \dd y.
\end{align}
In this work, we are going to compare the first order global flux scheme against the standard finite volume with both
piece-wise constant (first order accurate) and piece-wise linear (second order accurate) reconstructions.
Since the solution update can be performed in a dimension-by-dimension way, we can focus only on the 
$x$-direction for simplicity.
For piece-wise constant reconstruction, the left and right states at interface $i+\frac12$ simply are
\begin{equation}
  q^L_{i+\frac12,j} = \bar q_{i,j},\qquad q^R_{i+\frac12,j} = \bar q_{i+1,j}, \qquad\forall j.
\end{equation}
While in the second order case, the left and right states are computed through a piece-wise linear 
reconstruction of the solution as
\begin{equation}
  \tilde q (x,y ) =\bar  q_{i,j} + (x-x_i) \left(\del_x q\right)_{i,j} + (y-y_j) \left(\del_y q\right)_{i,j}, \quad x,y\in C_{i,j}. 
\end{equation} 
Hence, the left and right states at interface $i+\frac12$ are
\begin{equation}
  q^L_{i+\frac12,j} = \bar q_{i,j} + \frac{\Delta x}{2} \left(\del_x q\right)_{i,j},\qquad q^R_{i+\frac12,j} =\bar  q_{i+1,j} - \frac{\Delta x}{2} \left(\del_x q\right)_{i+1,j}, \qquad\forall j.
\end{equation}
Here, the slopes $\left(\del_x q\right)_{i,j}$ are evaluated using the generalized minmod limiter \cite{nessyahu1990non}:
\begin{equation}
  \left(\del_x q\right)_{i,j} = \text{minmod}\left(\theta\frac{\bar q_{i+1,j}-\bar q_{i,j}}{\Delta x},\frac{\bar q_{i+1,j}-\bar q_{i-1,j}}{2\Delta x},\theta\frac{\bar q_{i,j}-\bar q_{i-1,j}}{\Delta x}\right), 
\end{equation}
where $\theta$ is used to control the amount of dissipation.
In particular, the larger $\theta$ is, the sharper and more oscillatory the reconstruction will be. 
For the simulations presented in section \ref{sec:numerical}, when not specified, we set the parameter $\theta=1.3$.
The same approach has been used along the $y$ direction.

The classical minmod function is defined as
\begin{equation}
  \text{minmod}(a,b,c) = \begin{cases}
    \min(a,b,c), & \text{if } a,b,c>0,\\
    \max(a,b,c), & \text{if } a,b,c<0,\\
    0,         & \text{otherwise}.
  \end{cases}
\end{equation}

\section{Numerical experiments}\label{sec:numerical}

In this section, the goal is to show the performance of the new global flux scheme (GF)
when compared to the classical finite volume first order (FV-1) and second order (FV-2) approaches presented in section \ref{sec:FVcomparison}. 
Several test cases are presented to study the impact of the method on both linear and nonlinear hyperbolic problems: 
linear acoustics, shallow water and the Euler equations.
The numerical experiments are performed taking the gravity $g=9.812$, for the shallow water system, 
and the ratio of specific heats $\gamma=1.4$ for the Euler system.
All convergence analyses are performed on a set of nested quadrilateral meshes with $N_x=N_y=20,40,80,160,320$.
Time integration is performed through classical explicit Euler and second order Runge-Kutta methods.

\subsection{Linear acoustic system}
\subsubsection{Stationary vortex}

The first test case considered here concerns the simulation of the linear acoustic system.
The initial condition is a compactly supported vortex centered in $(x_0,y_0)=(0.5,0.5)$ defined on 
the square $[0,1]\times[0,1]$ with periodic boundary conditions, which is given by
\begin{align*}
  p(x,y) &=1,\\
  u(x,y) &= (y-y_0) f(\rho(x,y)),\\
  v(x,y) &=-(x-x_0) f(\rho(x,y)),
\end{align*}
with $\rho(x,y)=\frac{\sqrt{(x-x_0)^2+(y-y_0)^2}}{r_0}$, where $r_0=0.45$,
$f(\rho) = \gamma(1+\cos(\pi\rho))^2$ and $\gamma=\frac{12\pi\sqrt{0.981}}{r_0\sqrt{315\pi^2-2048}}$.
This vortex is taken from the work \cite{ricchiuto2021analytical}, where its derivation is described. 
This initial condition is a steady state of the acoustic system.

In table \ref{tab:ACvortexconv}, the errors computed with the $L_2$ norm and convergence rates are shown.
As can be noticed, the GF method outperforms the standard FV-1 and FV-2 methods in terms of discretization errors.
Although the GF is in principle first order accurate due to the piece-wise constant
reconstruction, a superconvergence behavior is experienced for stationary solutions (compare Proposition \ref{prop:RDeq}).
Hence, the GF method is not only able to preserve the vortex structure, but does so at
second order accuracy. 
\RI{In figure~\ref{fig:acoustics_vortex_N_error}, we compare the computational time with the $L^\infty$ error and with the mesh size of the three methods. The computational cost of the GF-FV is slightly larger than the one of the classical FV-1 (on average less than a factor 2), while it is comparable with the compuational costs of FV-2. Still, the error of the GF is many order of magnitude better than classical schemes.}
In particular, this is evident when increasing the final time of the simulation, as shown in figure \ref{fig:acoustic_longtime}.
Classical methods like FV-1 and FV-2 are not able to preserve the vortex structure for long times,
due to their numerical dissipation that spoils the final solution.

\begin{table}
  \caption{Linear acoustic system: vortex ($t_f=1$). $L_2$ error and order of accuracy $\tilde{n}$ for FV-1, FV-2 and GF.}\label{tab:ACvortexconv}
  \footnotesize
  \centering
  \begin{tabular}{ccccccc}
          \hline
          &\multicolumn{2}{c}{$p$} &\multicolumn{2}{c}{$u$}   &\multicolumn{2}{c}{$v$}\\[0.5mm]
          \cline{2-7}
          $N_x,N_y$ & $L_2$        & $\tilde{n}$ & $L_2$        & $\tilde{n}$ & $L_2$      & $\tilde{n}$ \\[0.5mm]\hline
          &\multicolumn{6}{c}{FV-1}\\[0.5mm]
          20  &  3.51E-05  &   --   &  6.51E-02   &  --  & 6.51E-02 &   --   \\
          40  &  4.58E-05  & -0.38  &  5.42E-02   & 0.26 & 5.42E-02 &  0.26  \\
          80  &  2.71E-05  &  0.75  &  3.97E-02   & 0.44 & 3.97E-02 &  0.44  \\
          160 &  1.06E-05  &  1.35  &  2.54E-02   & 0.64 & 2.54E-02 &  0.64  \\
          320 &  3.34E-06  &  1.66  &  1.47E-02   & 0.79 & 1.47E-02 &  0.79  \\
          &\multicolumn{6}{c}{FV-2}\\[0.5mm]
          20  &  4.31E-04  &  --  &  2.58E-02   &  --   & 2.58E-02 &   --   \\
          40  &  2.54E-04  & 0.76 &  6.12E-03   &  2.07 & 6.12E-03 &   2.07 \\
          80  &  6.37E-05  & 1.99 &  1.61E-03   &  1.93 & 1.61E-03 &   1.93 \\
          160 &  1.29E-05  & 2.30 &  4.46E-04   &  1.84 & 4.46E-04 &   1.84 \\
          320 &  2.73E-06  & 2.24 &  1.25E-04   &  1.83 & 1.25E-04 &   1.83 \\
          &\multicolumn{6}{c}{GF}\\[0.5mm]
          20  &  4.72E-05  &  --    &  3.95E-04   &   --  & 3.95E-04  &  --  \\
          40  &  4.53E-05  & 0.05   &  9.17E-05   &  2.10 & 9.17E-05  & 2.10 \\
          80  &  1.95E-05  & 1.21   &  2.26E-05   &  2.01 & 2.26E-05  & 2.01 \\
          160 &  6.42E-06  & 1.60   &  5.58E-06   &  2.01 & 5.58E-06  & 2.01 \\
          320 &  1.85E-06  & 1.79   &  1.38E-06   &  2.01 & 1.38E-06  & 2.01 \\
          \hline\hline\\[1pt]
  \end{tabular}
  \end{table}

\begin{figure}
\centering
\begin{tikzpicture}
\begin{loglogaxis}[
    width=0.48\textwidth,
    xlabel={Computational time [s]},
    ylabel={$L_\infty$ error in $u$},
    legend style={font=\small, legend pos=south west},
    grid=major,
    tick label style={font=\small},
]
\addplot+[thick, blue, mark=*] table [x=comp_time, y=u_Linf, col sep=comma] {acoustics_vortex/error_FV.csv};
\addlegendentry{FV}
\addplot+[thick, gray, mark=square*] table [x=comp_time, y=u_Linf, col sep=comma] {acoustics_vortex/error_FV2.csv};
\addlegendentry{FV2}
\addplot+[thick, magenta, mark=triangle*] table [x=comp_time, y=u_Linf, col sep=comma] {acoustics_vortex/error_GF_efficient.csv};
\addlegendentry{GF}
\end{loglogaxis}
\end{tikzpicture}
\hfill
\begin{tikzpicture}
	\begin{loglogaxis}[
		width=0.48\textwidth,
		ylabel={$N_x=N_y$},
		xlabel={Computational time [s]},
		legend style={font=\small, legend pos=north west},
		grid=major,
		tick label style={font=\small},
		]
		\addplot+[thick, blue, mark=*] table [x=comp_time, y=N, col sep=comma] {acoustics_vortex/error_FV.csv};
		\addlegendentry{FV}
		\addplot+[thick, gray, mark=square*] table [x=comp_time, y=N, col sep=comma] {acoustics_vortex/error_FV2.csv};
		\addlegendentry{FV2}
		\addplot+[thick, magenta, mark=triangle*] table [x=comp_time, y=N, col sep=comma] {acoustics_vortex/error_GF_efficient.csv};
		\addlegendentry{GF}
	\end{loglogaxis}
\end{tikzpicture}
\caption{Comparison of computational time vs $L_\infty$ error in $u$ and vs grid size $N_x$ for FV, FV-2, and GF methods.}
\label{fig:acoustics_vortex_N_error}
\end{figure}

\begin{figure}
  \centering
  \subfigure[FV-1]{\includegraphics[width=0.3\textwidth,height=0.25\textwidth,trim={38cm 0cm 0cm 0cm},clip]{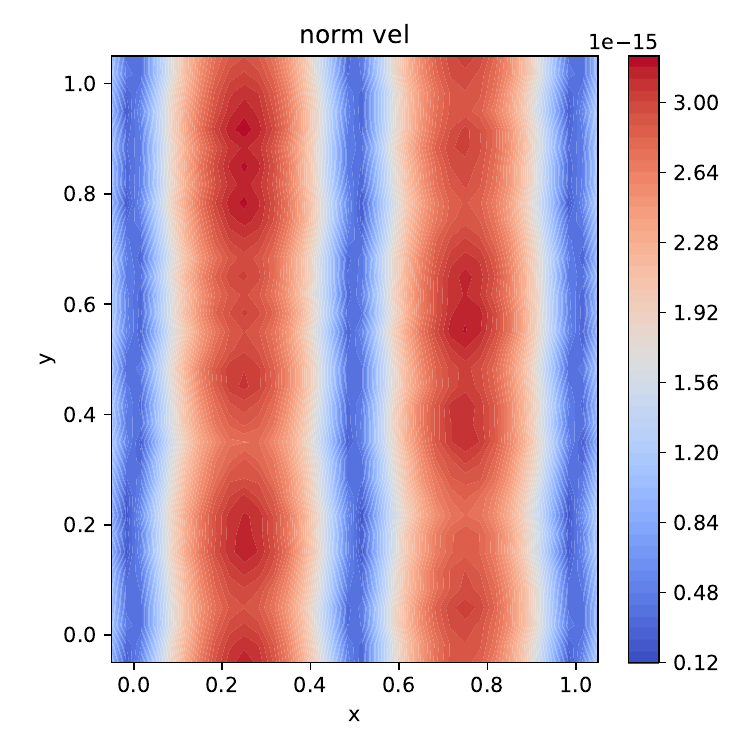}}\qquad
  \subfigure[FV-2]{\includegraphics[width=0.3\textwidth,height=0.25\textwidth,trim={38cm 0cm 0cm 0cm},clip]{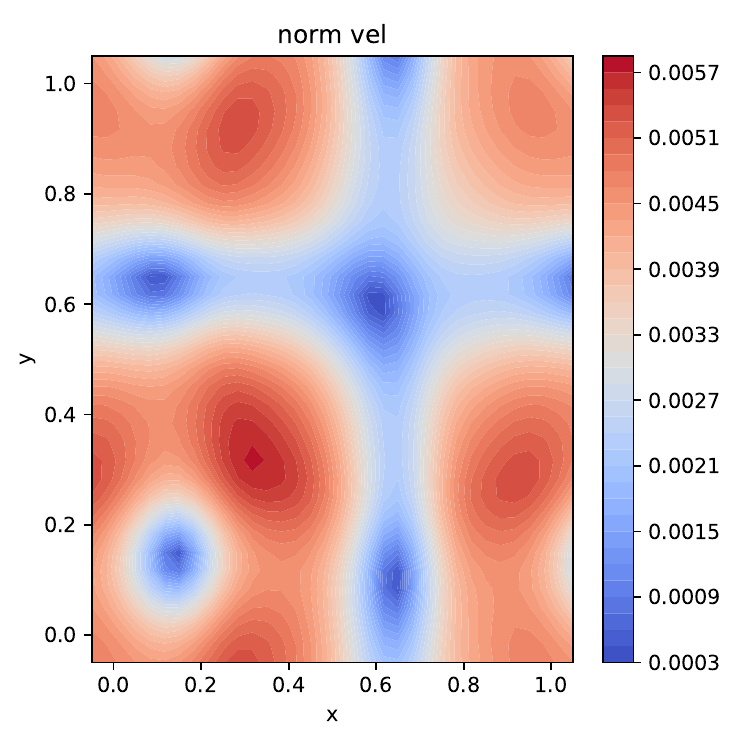}}\qquad
  \subfigure[GF]  {\includegraphics[width=0.3\textwidth,height=0.25\textwidth,trim={38cm 0cm 0cm 0cm},clip]{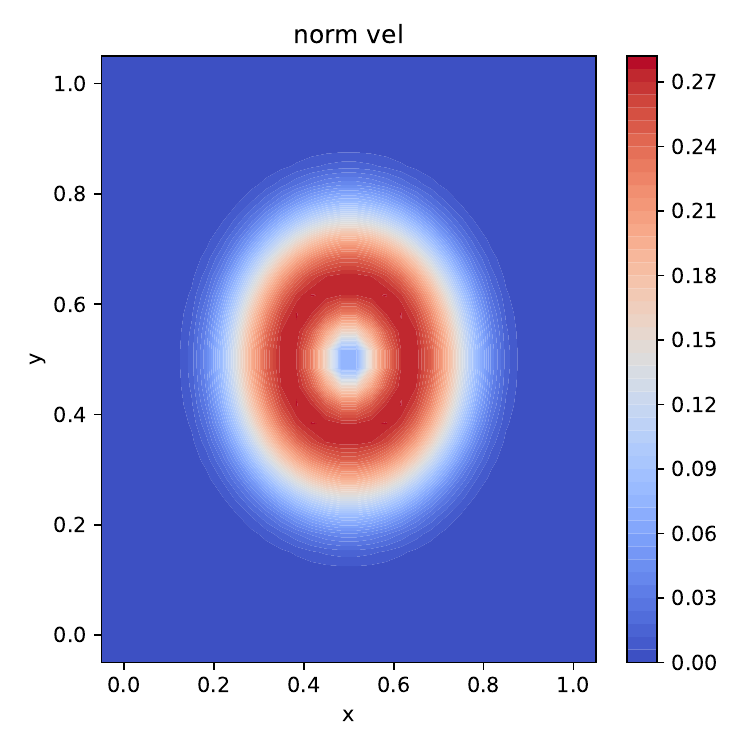}}
  \caption{Linear acoustic system: vortex. Isocontours of the velocity norm obtained with FV-1, FV-2 and GF after a long time integration ($t_f=200$).}
  \label{fig:acoustic_longtime}
\end{figure}

\subsection{Euler equations}
\subsubsection{Isentropic vortex}

In this section, we test the proposed method on a smooth isentropic vortex \cite{jiang1999high}.
The initial condition is given in terms of primitive variables and it consists in superposition of a
homogeneous background flow and a perturbation:
$$ \left(\rho,u,v,p\right) = \left(1+\delta\rho,\,u_0+\delta u,\,v_0+\delta v,\,1+\delta p\right).$$
The test case is set up in a $[0,10]\times[0,10]$ domain with periodic boundary conditions and vortex radius $r=\sqrt{(x-5)^2+(y-5)^2}$.
The vortex strength is $\epsilon=5$, and the entropy perturbation is assumed to be zero.
Given these hypotheses, the perturbations on velocity and temperature can be written as
$$ \begin{bmatrix} \delta u \\ \delta v \end{bmatrix} = \frac{\epsilon}{2\pi}\exp\left(\frac{1-r^2}{2}\right)\begin{bmatrix} -(y-5) \\ (x-5) \end{bmatrix},\qquad
\delta T = -\frac{(\gamma-1)\epsilon^2}{8\gamma\pi^2}\exp(1-r^2).$$
It follows that the perturbations on density and pressure read
$$ \delta\rho = (1+\delta T)^{\frac{1}{\gamma-1}}-1 ,\qquad \delta p = (1+\delta T)^{\frac{\gamma}{\gamma-1}}-1. $$
This test case is a stationary solution of the Euler equations.   \new{Note, that the maximum Mach number for this set up is about Ma= 0.7.}

In table~\ref{tab:SHUconv}, the convergence analysis for the isentropic vortex is presented by comparing the FV-1, FV-2 and GF methods
by running the simulation of a static vortex, i.e.\ $u_0=v_0=0$, until a final time $t_f=1$. 
As observed above, the GF shows superconvergent behavior with order 2. 
In terms of discretization errors, it outperforms not only the classical piecewise constant finite volume method, 
but also the second-order approach equipped with a linear reconstruction.
\RI{In figure~\ref{fig:Shu_steady_vortex_N_error}, we compare the computational costs of the method with respect to the $L^\infty$ error of $\rho u$ and with the mesh size. We confirm also for the nonlinear case that the GF is slightly more expensive of FV-1, but comparable to FV-2, while the error is increadibly smaller.}
Even after very long simulations times (see figure \ref{fig:Shu_longtime}) the new GF method is able to maintain the vortex, 
while the first order FV-1 dissipates everything away, and FV-2 significantly distorts the vortex structure and still diffuses it more than GF. 
Observe that the nonlinearity of the equations makes this test significantly more challenging than the corresponding test for linear acoustics, 
where in particular advection is not present.

In table~\ref{tab:SHUconv_moving}, the convergence analysis for a moving isentropic vortex is presented with $u_0=v_0=1$ and a final time $t_f=10$.
Here, we can directly observe that the GF method is indeed first order accurate, as expected, since the reconstruction is piecewise constant.
No superconvergence is observed in this case, as the solution is not stationary. The results show an relevant improvement of the GF method over the FV-1 in both
discretization errors and convergence rates, while the FV-2 method is, in this case, the best since it is able to achieve second order accuracy.
In figure \ref{fig:Shu_moving}, the solution at the final time is shown for the three methods.

\begin{table}
  \caption{Euler equations: isentropic vortex with $u_0=v_0=0$ ($t_f=1$). $L_2$ error and order of accuracy $\tilde{n}$ for FV-1, FV-2 and GF methods.}
  \label{tab:SHUconv}
  \footnotesize
  \centering
  \begin{tabular}{ccccccccc}
          \hline\hline
          &\multicolumn{2}{c}{$\rho $} &\multicolumn{2}{c}{$\rho u$}   &\multicolumn{2}{c}{$\rho v$} &\multicolumn{2}{c}{$\rho E$}\\[0.5mm]
          \cline{2-9}
          $N_x,N_y$ & $L_2$        & $\tilde{n}$ & $L_2$        & $\tilde{n}$ & $L_2$      & $\tilde{n}$ & $L_2$      & $\tilde{n}$ \\[0.5mm]\hline
          &\multicolumn{8}{c}{FV-1}\\[0.5mm]
          20  & 3.58E-01  &  --   &  6.77E-01  &  --   & 6.77E-01  &  --   & 1.16E+00  &  --    \\
          40  & 2.47E-01  & 0.53  &  4.40E-01  & 0.62  & 4.40E-01  & 0.62  & 8.29E-01  & 0.48   \\
          80  & 1.49E-01  & 0.72  &  2.59E-01  & 0.76  & 2.59E-01  & 0.76  & 5.15E-01  & 0.68   \\
          160 & 8.33E-02  & 0.84  &  1.43E-01  & 0.85  & 1.43E-01  & 0.85  & 2.91E-01  & 0.82   \\
          320 & 4.42E-02  & 0.91  &  7.56E-02  & 0.91  & 7.56E-02  & 0.91  & 1.56E-01  & 0.90   \\
          &\multicolumn{8}{c}{FV-2}\\[0.5mm]
          20  & 1.06E-01  &   --  & 2.05E-01   &  --  & 2.00E-01  &  --   &  4.32E-01 &  --    \\
          40  & 3.62E-02  & 1.55  & 6.74E-02   & 1.60 & 6.71E-02  & 1.57  &  1.20E-01 & 1.85   \\
          80  & 1.07E-02  & 1.76  & 1.93E-02   & 1.80 & 1.95E-02  & 1.78  &  2.91E-02 & 2.04   \\
          160 & 2.39E-03  & 2.16  & 5.58E-03   & 1.78 & 5.61E-03  & 1.79  &  7.04E-03 & 2.04   \\
          320 & 5.12E-04  & 2.22  & 1.39E-03   & 2.00 & 1.39E-03  & 2.01  &  1.56E-03 & 2.17   \\
          &\multicolumn{8}{c}{GF}\\[0.5mm]
          20  & 1.52E-02  &  --   & 3.67E-02  &  --  & 3.67E-02  &   --  & 4.59E-02  &  --   \\
          40  & 5.95E-03  & 1.35  & 1.15E-02  & 1.67 & 1.15E-02  & 1.67  & 1.54E-02  & 1.57   \\
          80  & 1.76E-03  & 1.76  & 3.06E-03  & 1.90 & 3.06E-03  & 1.90  & 4.35E-03  & 1.82   \\
          160 & 4.69E-04  & 1.90  & 7.87E-04  & 1.96 & 7.87E-04  & 1.96  & 1.16E-03  & 1.90   \\
          320 & 1.21E-04  & 1.95  & 2.00E-04  & 1.97 & 2.00E-04  & 1.97  & 3.02E-04  & 1.94   \\
          \hline\hline\\[1pt]
  \end{tabular}
  \end{table}

\begin{figure}
	\centering
	\begin{tikzpicture}
		\begin{loglogaxis}[
			width=0.48\textwidth,
			xlabel={Computational time [s]},
			ylabel={$L_\infty$ error in $\rho u$},
			legend style={font=\small, legend pos=south west},
			grid=major,
			tick label style={font=\small},
			]
			\addplot+[thick, blue, mark=*] table [x=comp_time, y=rhou_Linf, col sep=comma] {Euler_vortex_Shu/error_FV.csv};
			\addlegendentry{FV}
			\addplot+[thick, gray, mark=square*] table [x=comp_time, y=rhou_Linf, col sep=comma] {Euler_vortex_Shu/error_FV2.csv};
			\addlegendentry{FV2}
			\addplot+[thick, magenta, mark=triangle*] table [x=comp_time, y=rhou_Linf, col sep=comma] {Euler_vortex_Shu/error_GF_efficient.csv};
			\addlegendentry{GF}
		\end{loglogaxis}
	\end{tikzpicture}
	\hfill
	\begin{tikzpicture}
		\begin{loglogaxis}[
			width=0.48\textwidth,
			ylabel={$N_x=N_y$},
			xlabel={Computational time [s]},
			legend style={font=\small, legend pos=north west},
			grid=major,
			tick label style={font=\small},
			]
			\addplot+[thick, blue, mark=*] table [x=comp_time, y=N, col sep=comma] {Euler_vortex_Shu/error_FV.csv};
			\addlegendentry{FV}
			\addplot+[thick, gray, mark=square*] table [x=comp_time, y=N, col sep=comma] {Euler_vortex_Shu/error_FV2.csv};
			\addlegendentry{FV2}
			\addplot+[thick, magenta, mark=triangle*] table [x=comp_time, y=N, col sep=comma] {Euler_vortex_Shu/error_GF_efficient.csv};
			\addlegendentry{GF}
		\end{loglogaxis}
	\end{tikzpicture}
	\caption{Comparison of computational time vs $L_\infty$ error in $\rho u$ and vs grid size $N_x$ for FV, FV-2, and GF methods.}
	\label{fig:Shu_steady_vortex_N_error}
\end{figure}

\begin{table}
  \caption{Euler equations: isentropic vortex with $u_0=v_0=1$ ($t_f=10$). $L_2$ error and order of accuracy $\tilde{n}$ for FV-1, FV-2 and GF methods.}
  \label{tab:SHUconv_moving}
  \footnotesize
  \centering
  \begin{tabular}{ccccccccc}
          \hline\hline
          &\multicolumn{2}{c}{$\rho $} &\multicolumn{2}{c}{$\rho u$}   &\multicolumn{2}{c}{$\rho v$} &\multicolumn{2}{c}{$\rho E$}\\[0.5mm]
          \cline{2-9}
          $N_x,N_y$ & $L_2$        & $\tilde{n}$ & $L_2$        & $\tilde{n}$ & $L_2$      & $\tilde{n}$ & $L_2$      & $\tilde{n}$ \\[0.5mm]\hline
          &\multicolumn{8}{c}{FV-1}\\[0.5mm]
          20  &  6.50E-01 &  --   &  1.54E+00  &  --   &  1.54E+00  &  --   &  3.12E+00  &  --   \\
          40  &  6.21E-01 &  0.06 &  1.46E+00  &  0.07 &  1.46E+00  &  0.07 &  3.01E+00  &  0.05 \\
          80  &  5.82E-01 &  0.09 &  1.31E+00  &  0.15 &  1.31E+00  &  0.15 &  2.83E+00  &  0.09 \\
          160 &  5.13E-01 &  0.18 &  1.06E+00  &  0.30 &  1.07E+00  &  0.29 &  2.48E+00  &  0.19 \\
          320 &  4.01E-01 &  0.35 &  7.58E-01  &  0.49 &  7.63E-01  &  0.49 &  1.92E+00  &  0.36 \\
          &\multicolumn{8}{c}{FV-2}\\[0.5mm]
          20  & 5.29E-01  &  --   &  1.07E+00  &  --   & 1.12E+00 &   --  &  2.48E+00  &  --   \\
          40  & 2.45E-01  &  1.10 &  4.42E-01  &  1.28 & 4.84E-01 &  1.21 &  1.17E+00  &  1.08 \\
          80  & 6.55E-02  &  1.90 &  1.26E-01  &  1.80 & 1.37E-01 &  1.81 &  2.82E-01  &  2.04 \\
          160 & 1.85E-02  &  1.82 &  3.22E-02  &  1.97 & 3.42E-02 &  2.00 &  6.03E-02  &  2.22 \\
          320 & 3.38E-03  &  2.45 &  7.49E-03  &  2.10 & 8.12E-03 &  2.07 &  1.28E-02  &  2.23 \\
          &\multicolumn{8}{c}{GF}\\[0.5mm]
          20  & 5.46E-01  &  --   &  1.30E+00  &  --   &  1.13E+00  &  --   &  2.58E+00  &  --   \\
          40  & 4.44E-01  & 0.30  &  1.02E+00  &  0.34 &  7.71E-01  &  0.55 &  2.10E+00  &  0.29   \\
          80  & 3.18E-01  & 0.48  &  6.96E-01  &  0.56 &  4.92E-01  &  0.64 &  1.53E+00  &  0.45   \\
          160 & 1.99E-01  & 0.67  &  4.12E-01  &  0.75 &  2.96E-01  &  0.73 &  9.73E-01  &  0.65   \\
          320 & 1.12E-01  & 0.82  &  2.24E-01  &  0.87 &  1.68E-01  &  0.81 &  5.55E-01  &  0.80   \\
          \hline\hline\\[1pt]
  \end{tabular}
  \end{table}

  \begin{figure}
    \centering
    \subfigure[FV-1]{\includegraphics[width=0.3\textwidth,height=0.25\textwidth,trim={51cm 0cm 0cm 0cm},clip]{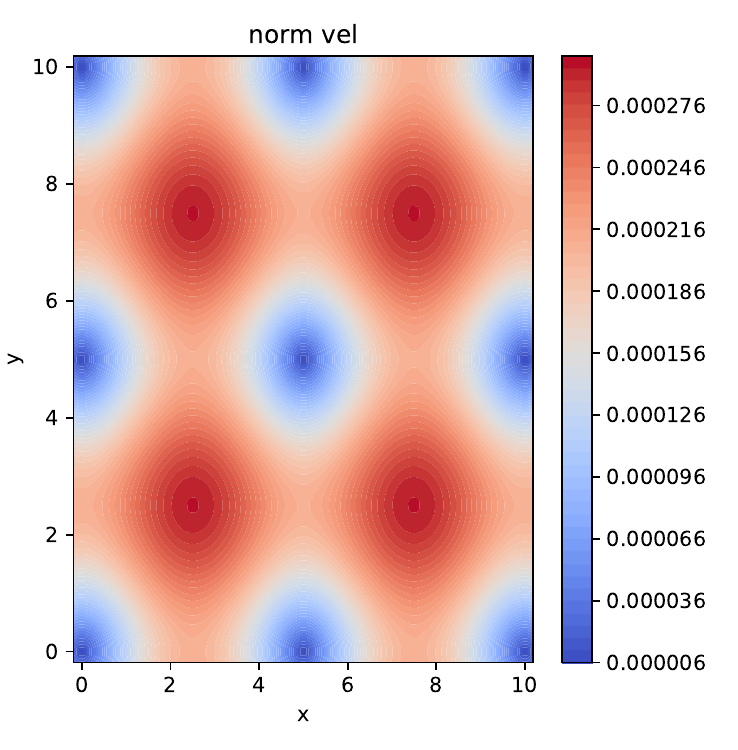}}\qquad
    \subfigure[FV-2]{\includegraphics[width=0.3\textwidth,height=0.25\textwidth,trim={51cm 0cm 0cm 0cm},clip]{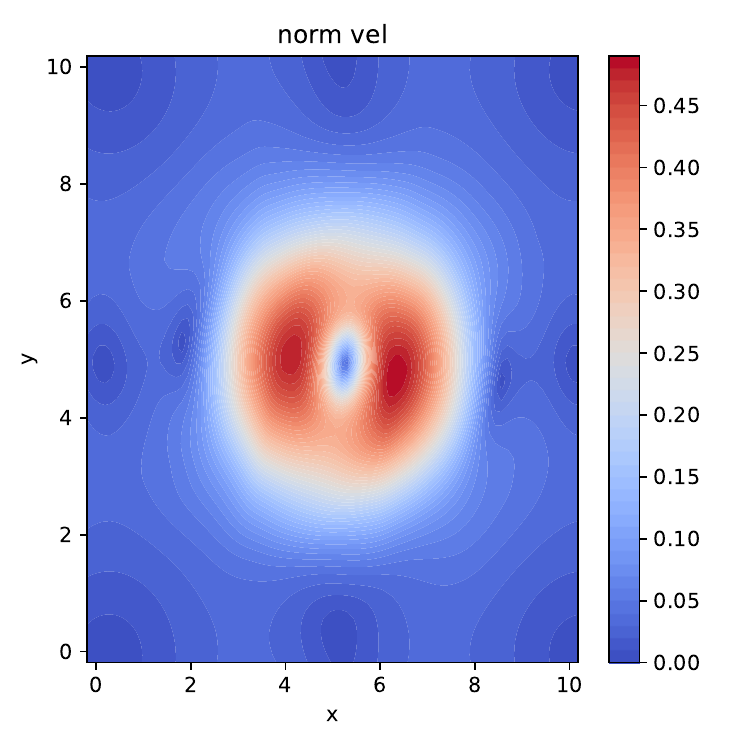}}\qquad
    \subfigure[GF]{\includegraphics[width=0.3\textwidth,height=0.25\textwidth,trim={51cm 0cm 0cm 0cm},clip]{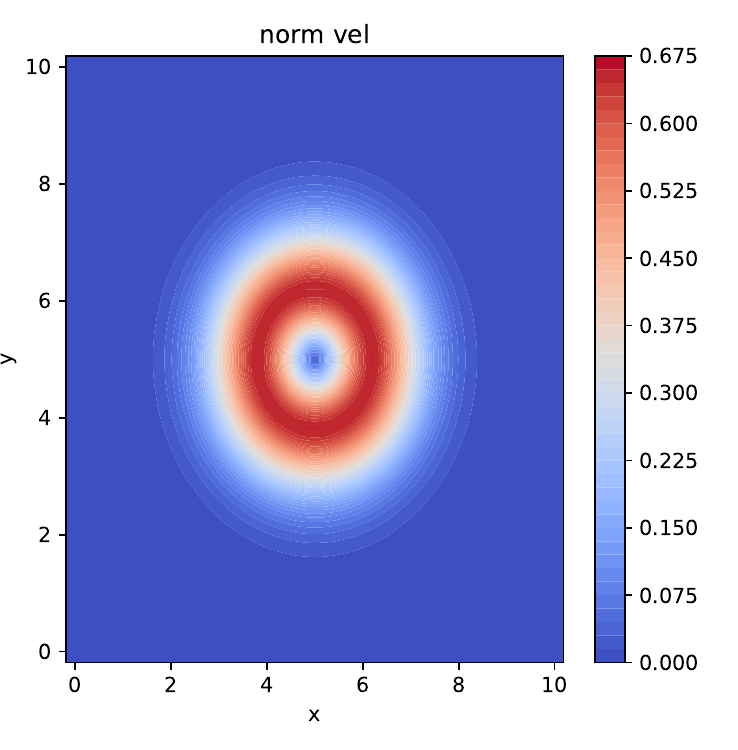}}
    \caption{Euler equations: isentropic vortex with $u_0=v_0=0$. Isocontours of the velocity norm obtained with FV-1, FV-2 and GF after a long time integration ($t_f=200$).}
    \label{fig:Shu_longtime}
  \end{figure}

  \begin{figure}
    \centering
    \subfigure[FV-1]{\includegraphics[width=0.3\textwidth,height=0.25\textwidth,trim={51cm 0cm 0cm 0cm},clip]{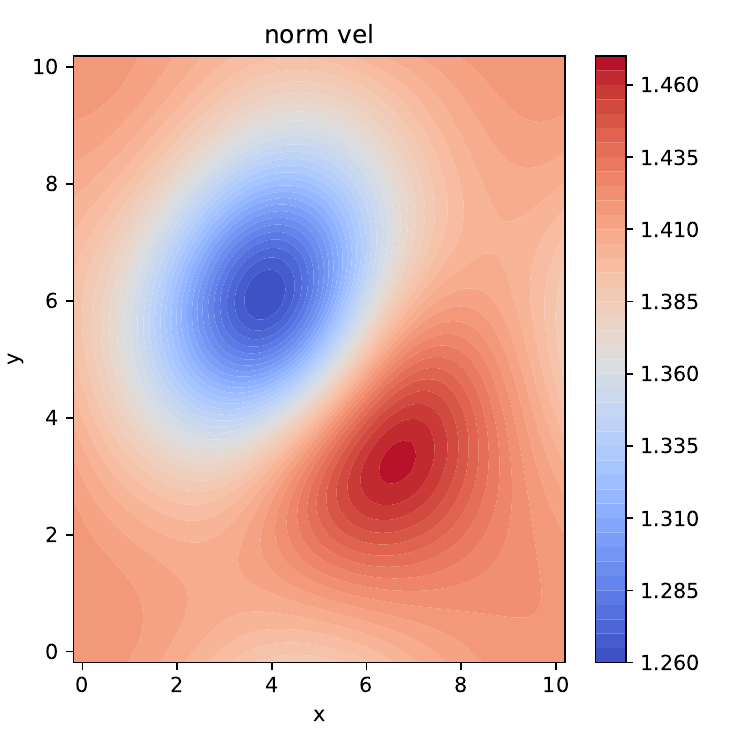}}\qquad
    \subfigure[FV-2]{\includegraphics[width=0.3\textwidth,height=0.25\textwidth,trim={51cm 0cm 0cm 0cm},clip]{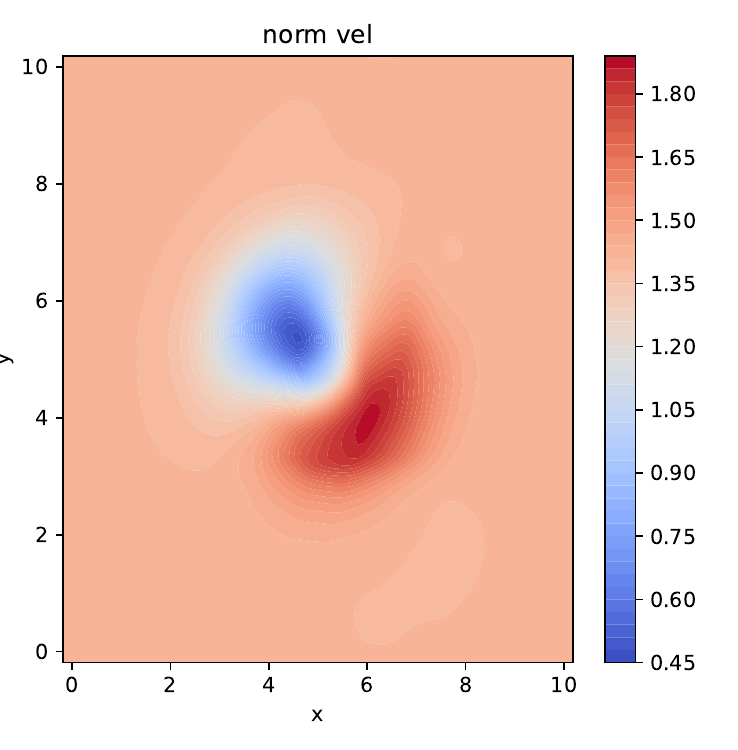}}\qquad
    \subfigure[GF]{\includegraphics[width=0.3\textwidth,height=0.25\textwidth,trim={51cm 0cm 0cm 0cm},clip]{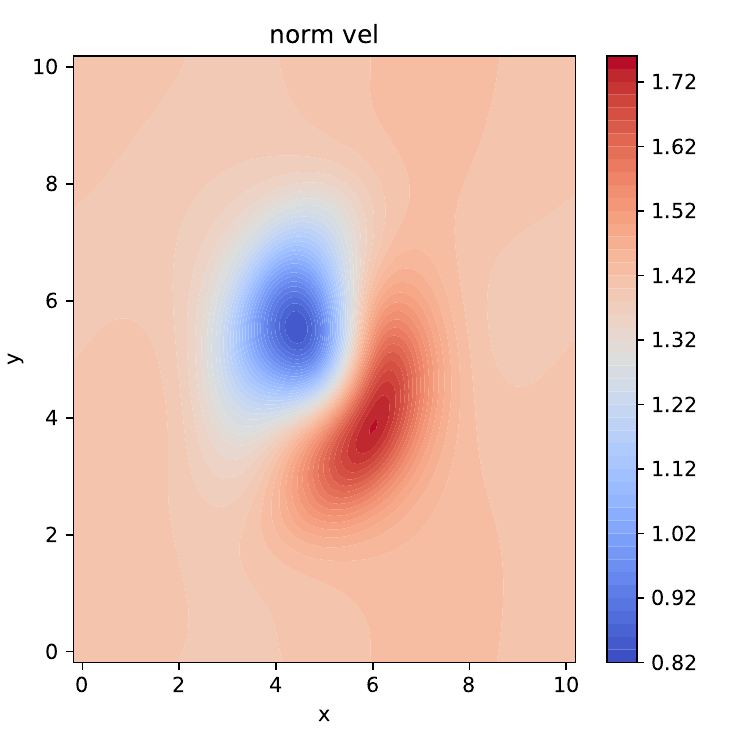}}
    \caption{Euler equations: isentropic vortex with $u_0=v_0=1$. Isocontours of the velocity norm obtained with FV-1, FV-2 and GF at $t_f=10$.}
    \label{fig:Shu_moving}
  \end{figure}

\subsubsection{\new{Isentropic vortex:  low Mach behaviour}}

\new{As recalled in the introduction, the resolution of the  long-time limit of linear acoustics is tightly connected to 
the  low Mach number limit of the Euler equations. This connection is studied  in depth e.g. in  
  \cite{jung2022steady,jung2024behavior}.  A connection between stationarity preservation
  and preservation of asymptotic  has not been rigorously  proved, however 
previous studies have shown that stationarity preserving methods are also well behaved for low Mach
  \cite{barsukow2025stationaritypreservationlowmach,barsukow2018low}. Following the last references, in this section we investigate 
  the low Mach behaviour of the new stationarity preserving formulation. To this end we compute long time 
  simulations of the same vortex for Mach numbers Ma = $10^{-2}$,  Ma = $10^{-4}$, and Ma = $10^{-6}$.
}
\new{
   The contours of the velocity norm at time $t_f=200$
  on a 40$\times$40 mesh are visualized on figure  \ref{fig:Shu_low_ma}. The figure shows clearly that the  first order  finite volume
  method without reconstruction  is unable to provide a reasonable approximation on this mesh level, giving essentially a constant state. 
  The second order method improves this, however considerably losing the circular symmetry of the vortex, with significant dissipation.
  The new method provides a remarkable approximation, with very little dissipation, and 
perfect circular symmetry despite of the coarse Cartesian mesh.}
  
\new{
  \begin{figure}
    \centering
    \subfigure[FV-1, Ma = $10^{-2}$ ]{\includegraphics[width=0.3\textwidth,height=0.25\textwidth,trim={51cm 0cm 0cm 0cm},clip]{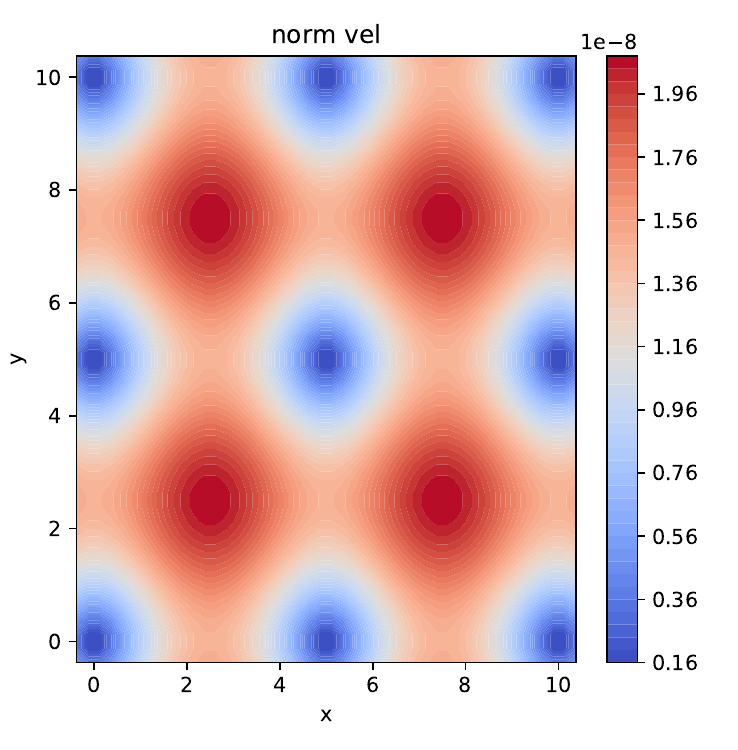}}\qquad
    \subfigure[FV-1, Ma = $10^{-4}$]{\includegraphics[width=0.3\textwidth,height=0.25\textwidth,trim={51cm 0cm 0cm 0cm},clip]{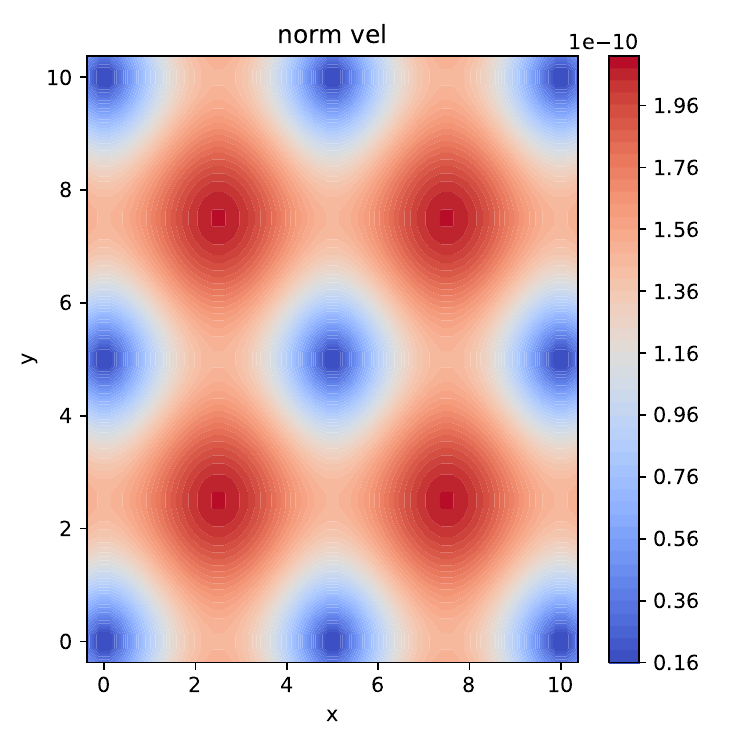}}\qquad
    \subfigure[FV-1, Ma = $10^{-6}$]{\includegraphics[width=0.3\textwidth,height=0.25\textwidth,trim={51cm 0cm 0cm 0cm},clip]{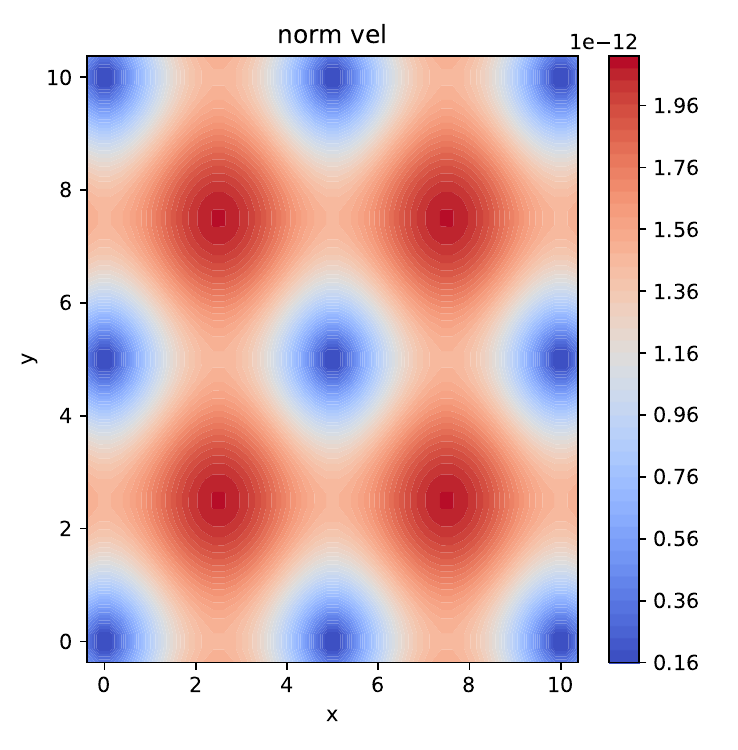}}
    
    \subfigure[FV-2, Ma = $10^{-2}$]{\includegraphics[width=0.3\textwidth,height=0.25\textwidth,trim={51.5cm 0cm 0.4cm 0cm},clip]{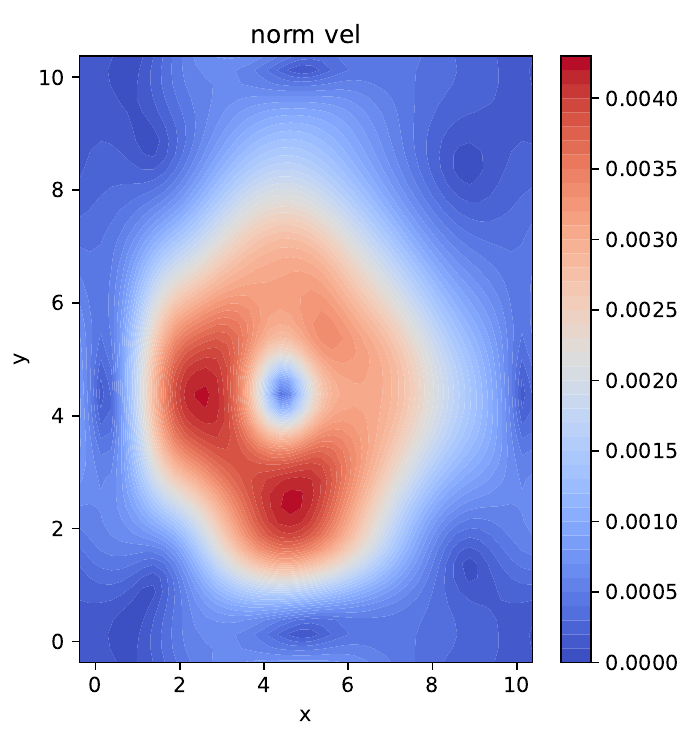}}\qquad
    \subfigure[FV-2, Ma = $10^{-4}$]{\includegraphics[width=0.3\textwidth,height=0.25\textwidth,trim={51.5cm 0cm 0.2cm 0cm},clip]{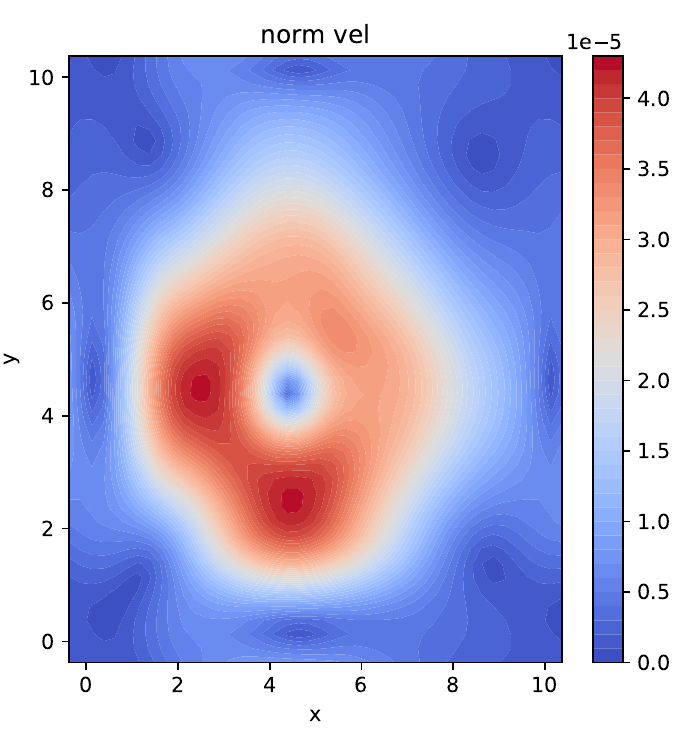}}\qquad
    \subfigure[FV-2, Ma = $10^{-6}$]{\includegraphics[width=0.3\textwidth,height=0.25\textwidth,trim={51.5cm 0cm 0.2cm 0cm},clip]{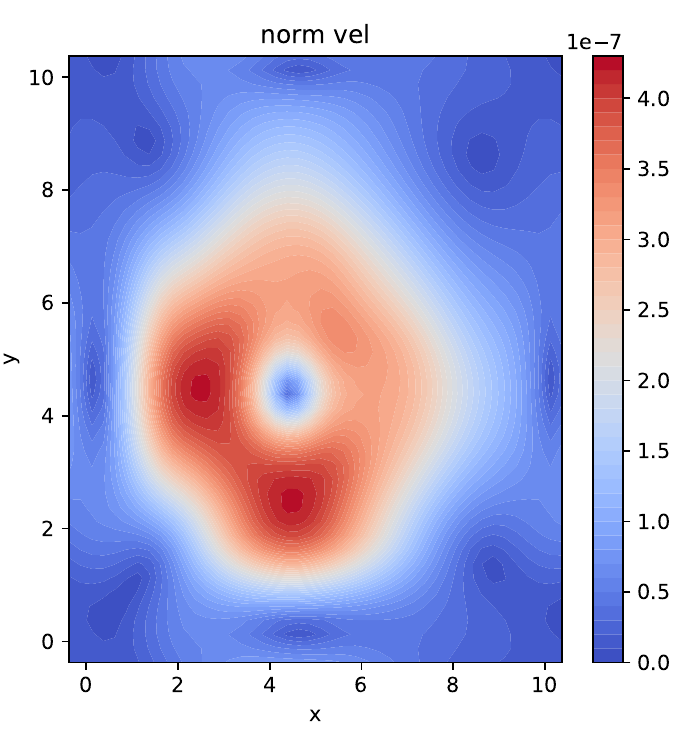}}
    
        \subfigure[GF, Ma = $10^{-2}$]{\includegraphics[width=0.3\textwidth,height=0.25\textwidth,trim={51.5cm 0cm 0.2cm 0cm},clip]{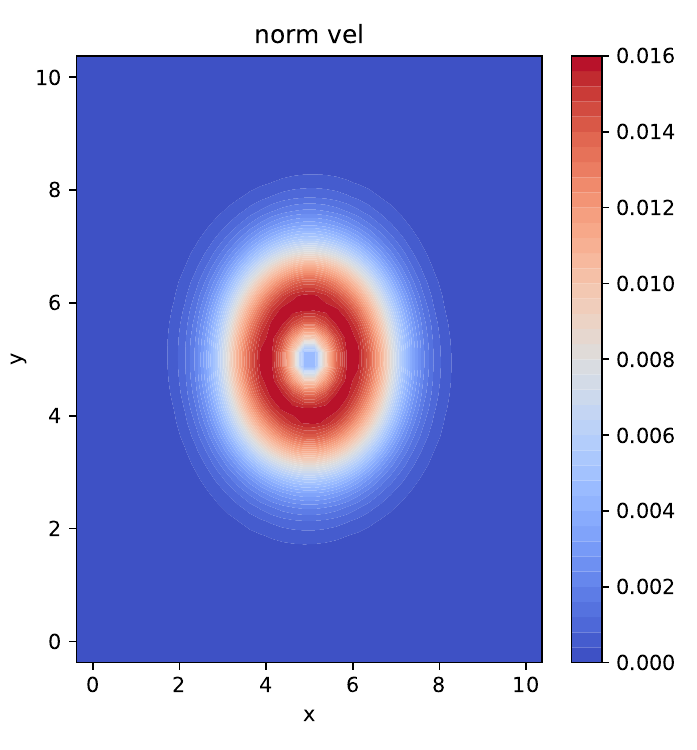}}\qquad
    \subfigure[GF, Ma = $10^{-4}$]{\includegraphics[width=0.3\textwidth,height=0.25\textwidth,trim={51.6cm 0cm 0.55cm 0cm},clip]{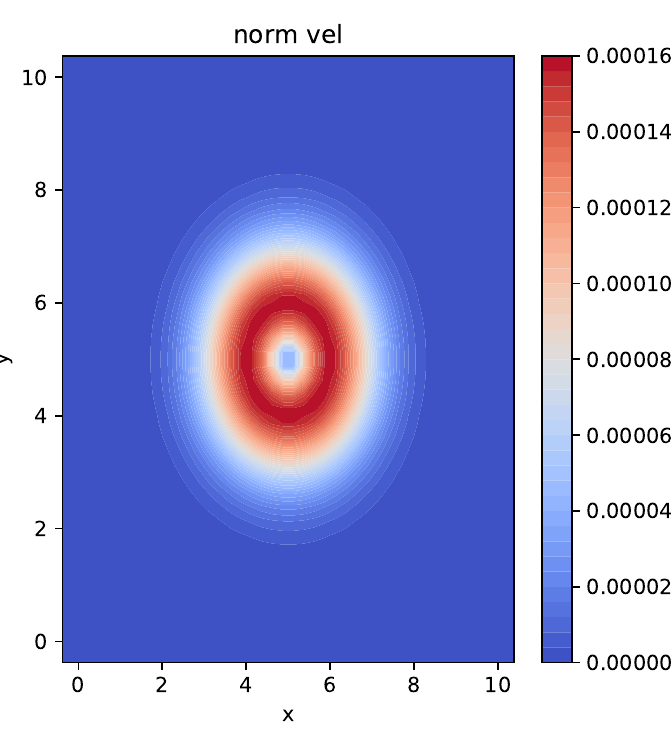}}\qquad
    \subfigure[GF, Ma = $10^{-6}$]{\includegraphics[width=0.3\textwidth,height=0.25\textwidth,trim={51.5cm 0cm 0.2cm 0cm},clip]{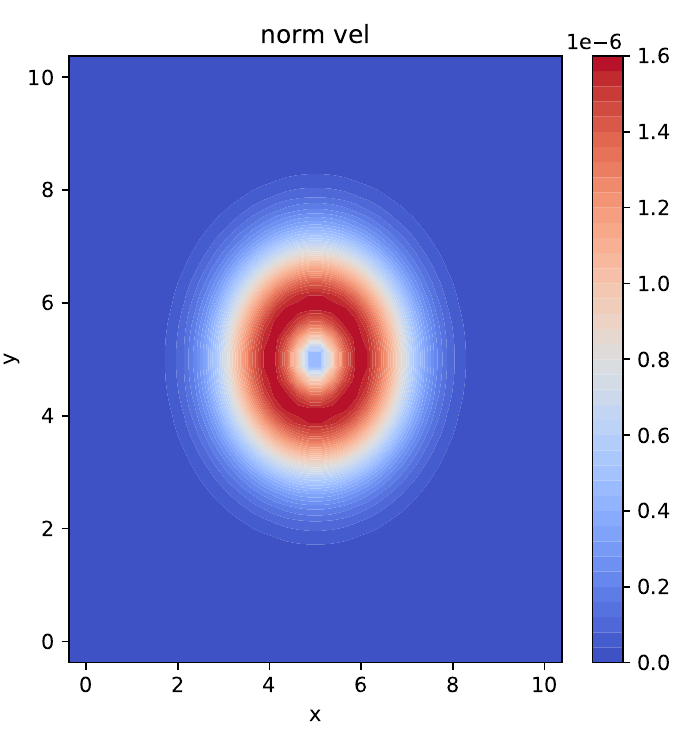}}
    
    \caption{Euler equations:   Low  Mach stationary isentropic vortex. Isocontours of the velocity norm   at $t_f=200$ obtained with  FV-1 (top row), FV-2  (middle row), and GF (bottom row).
Left:  Ma=0.01. Center: Ma = 0.0001. Right: Ma=0.000001.}
    \label{fig:Shu_low_ma}
  \end{figure}

To provide a quantitative assessment of the enhancements brought by our approach, on figure \ref{fig:Shu_low_ma1} we report the convergence of the error
of the $x$ momentum at $t_f=200$. The error is scaled by the exact maximum value (which is of the order of the Mach number itself), 
to allow a comparison of the results for the three values of the Mach number.
The results for the FV method are similar to the unfiltered ones presented in \cite{jung2022steady} for the velocity: 
the method barely converges on the coarser resolutions, and starts converging with a slope of about 1/2 only on the finest resolution (which is considerably finer
that those used in the reference). Despite of the poor  qualitative results obtained on the coarsest meshes,
the second order FV method does provide the expected slope for all Mach numbers. However, our approach provides errors which 
are systematically two orders of magnitude smaller, for all grid sizes and meshes. This supports the idea that stationarity preservation
and asymptotic preservation are tightly linked.

\begin{figure}
    \subfigure[FV-1]{\includegraphics[trim={4cm 5cm 1.5cm 0cm},clip,width=0.32\textwidth]{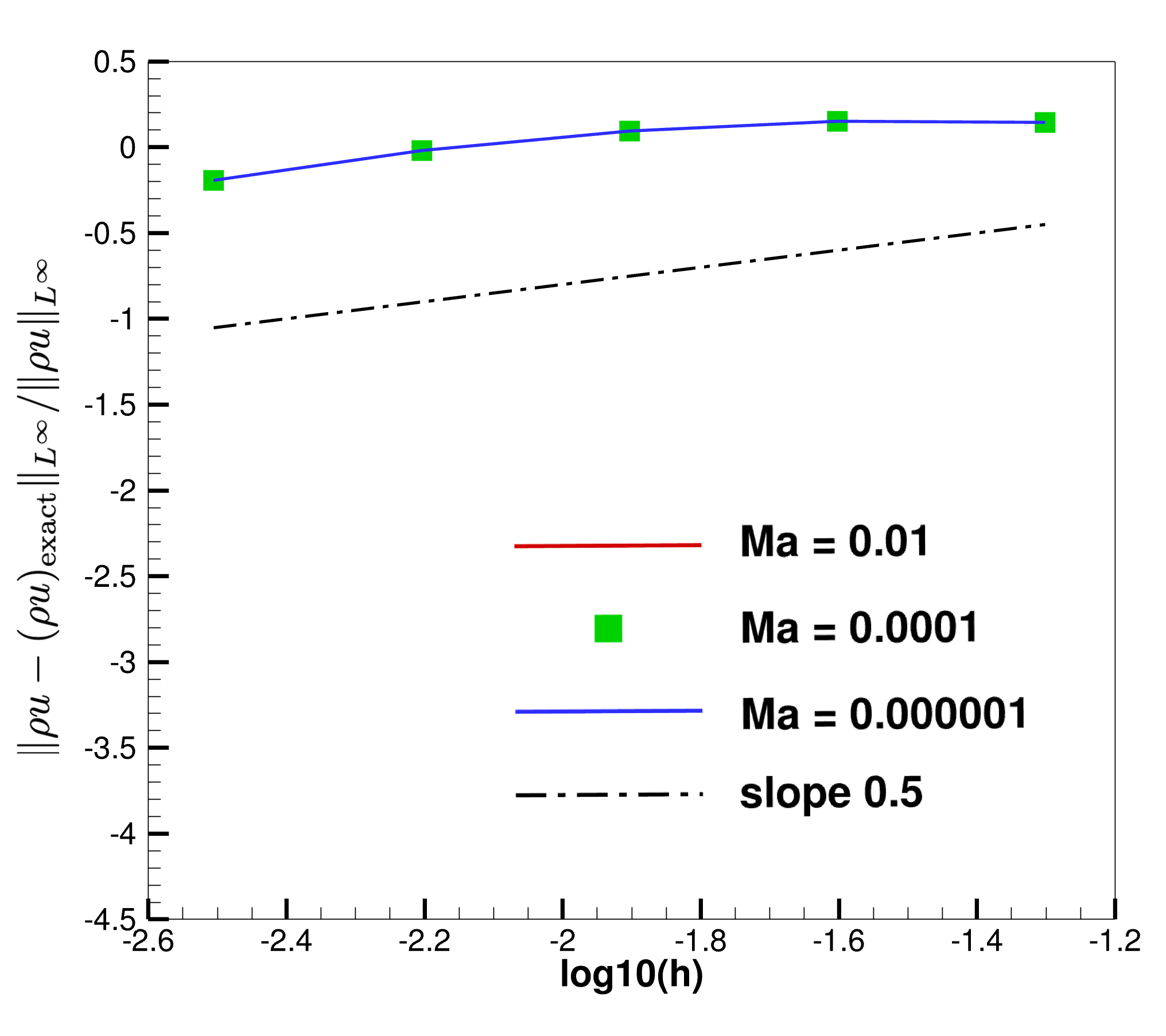}} 
    \subfigure[FV-2]{\includegraphics[trim={4cm 5cm 1.5cm 0cm},clip,width=0.32\textwidth]{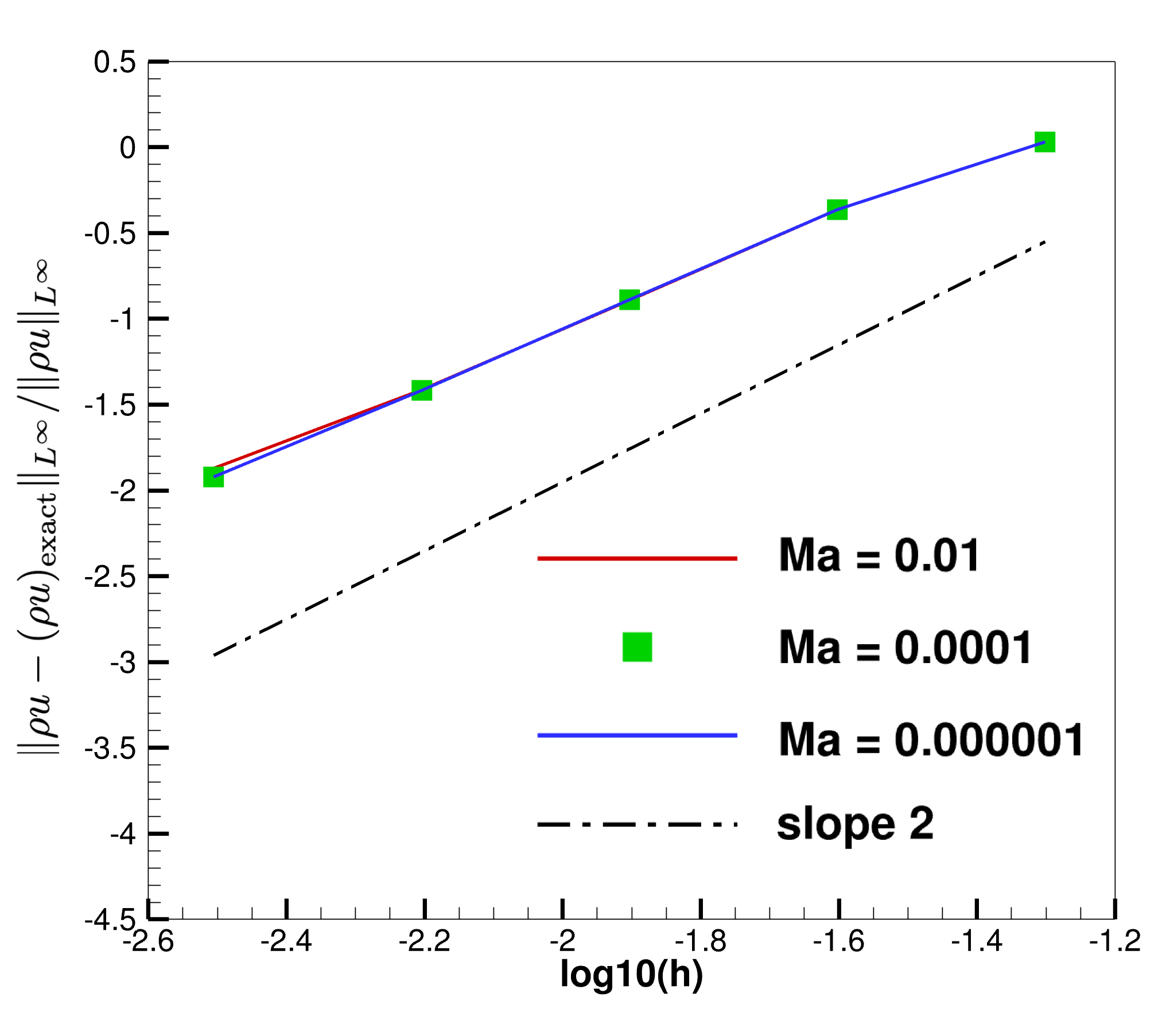}} 
    \subfigure[GF]{\includegraphics[trim={4cm 5cm 1.5cm 0cm},clip,width=0.32\textwidth]{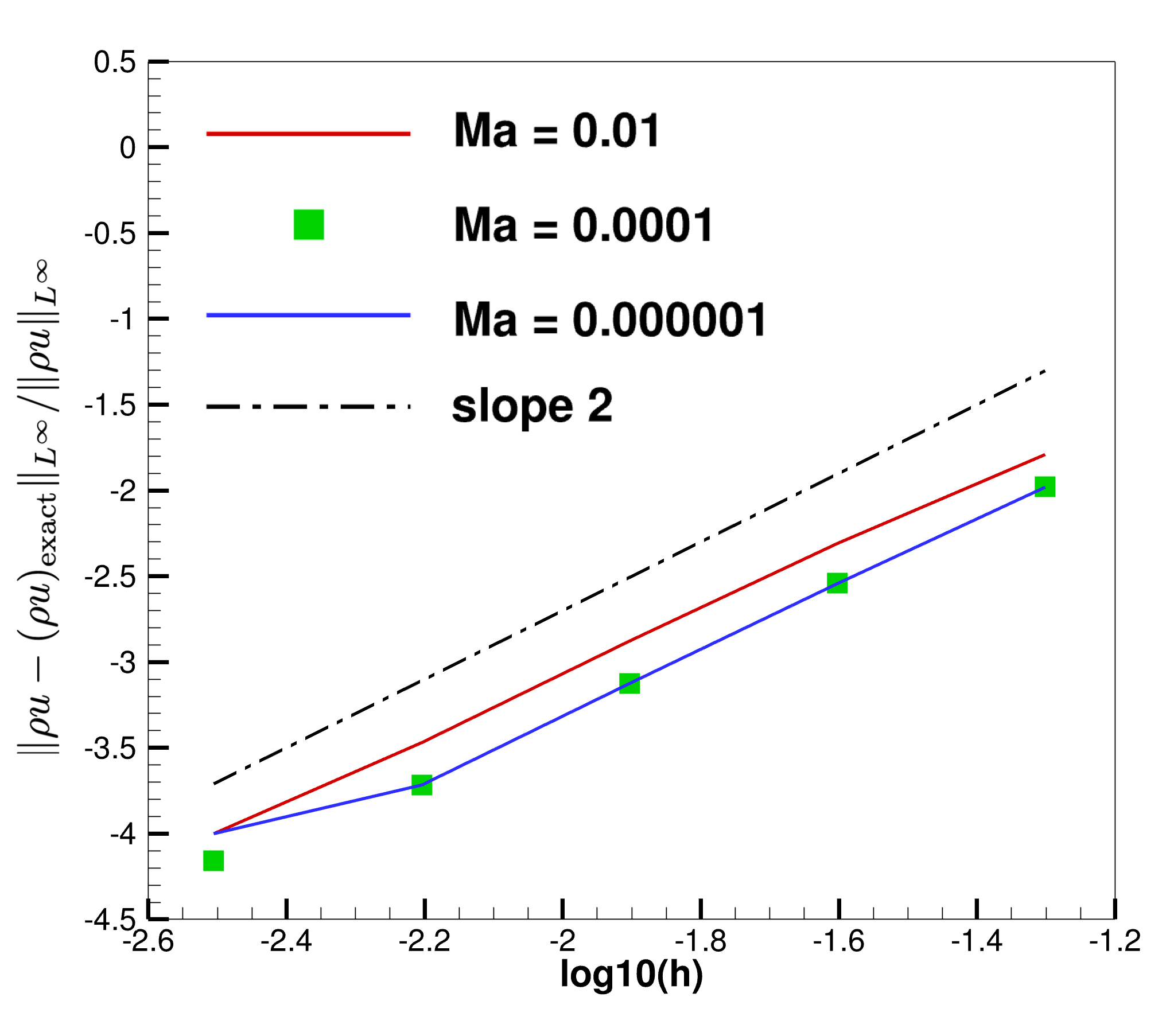}}
    
    \caption{Euler equations: Low  Mach stationary isentropic vortex. Mesh convergence  of the error on the $x$ momentum (scaled by the exact maximum value)
 at $t_f=200$ and different Mach numbers for the    FV-1 (left), FV-2  (center), and GF (right). On y-axis the $\log_{10} \left( \lVert \rho u - (\rho u)_{\text{exact}}\rVert_{L^\infty}/\lVert (\rho u)_{\text{exact}} \rVert_{L^\infty} \right)$ while on the x-axis the $\log_{10}(\Delta x)$.}
    \label{fig:Shu_low_ma1}
  \end{figure}
}

\begin{table}
  \caption{Euler equations: isentropic vortex with $u_0=v_0=1$ ($t_f=10$). $L_2$ error and order of accuracy $\tilde{n}$ for FV-1, FV-2 and GF methods.}
  \label{tab:SHUconv_moving}
  \footnotesize
  \centering
  \begin{tabular}{ccccccccc}
          \hline\hline
          &\multicolumn{2}{c}{$\rho $} &\multicolumn{2}{c}{$\rho u$}   &\multicolumn{2}{c}{$\rho v$} &\multicolumn{2}{c}{$\rho E$}\\[0.5mm]
          \cline{2-9}
          $N_x,N_y$ & $L_2$        & $\tilde{n}$ & $L_2$        & $\tilde{n}$ & $L_2$      & $\tilde{n}$ & $L_2$      & $\tilde{n}$ \\[0.5mm]\hline
          &\multicolumn{8}{c}{FV-1}\\[0.5mm]
          20  &  6.50E-01 &  --   &  1.54E+00  &  --   &  1.54E+00  &  --   &  3.12E+00  &  --   \\
          40  &  6.21E-01 &  0.06 &  1.46E+00  &  0.07 &  1.46E+00  &  0.07 &  3.01E+00  &  0.05 \\
          80  &  5.82E-01 &  0.09 &  1.31E+00  &  0.15 &  1.31E+00  &  0.15 &  2.83E+00  &  0.09 \\
          160 &  5.13E-01 &  0.18 &  1.06E+00  &  0.30 &  1.07E+00  &  0.29 &  2.48E+00  &  0.19 \\
          320 &  4.01E-01 &  0.35 &  7.58E-01  &  0.49 &  7.63E-01  &  0.49 &  1.92E+00  &  0.36 \\
          &\multicolumn{8}{c}{FV-2}\\[0.5mm]
          20  & 5.29E-01  &  --   &  1.07E+00  &  --   & 1.12E+00 &   --  &  2.48E+00  &  --   \\
          40  & 2.45E-01  &  1.10 &  4.42E-01  &  1.28 & 4.84E-01 &  1.21 &  1.17E+00  &  1.08 \\
          80  & 6.55E-02  &  1.90 &  1.26E-01  &  1.80 & 1.37E-01 &  1.81 &  2.82E-01  &  2.04 \\
          160 & 1.85E-02  &  1.82 &  3.22E-02  &  1.97 & 3.42E-02 &  2.00 &  6.03E-02  &  2.22 \\
          320 & 3.38E-03  &  2.45 &  7.49E-03  &  2.10 & 8.12E-03 &  2.07 &  1.28E-02  &  2.23 \\
          &\multicolumn{8}{c}{GF}\\[0.5mm]
          20  & 5.46E-01  &  --   &  1.30E+00  &  --   &  1.13E+00  &  --   &  2.58E+00  &  --   \\
          40  & 4.44E-01  & 0.30  &  1.02E+00  &  0.34 &  7.71E-01  &  0.55 &  2.10E+00  &  0.29   \\
          80  & 3.18E-01  & 0.48  &  6.96E-01  &  0.56 &  4.92E-01  &  0.64 &  1.53E+00  &  0.45   \\
          160 & 1.99E-01  & 0.67  &  4.12E-01  &  0.75 &  2.96E-01  &  0.73 &  9.73E-01  &  0.65   \\
          320 & 1.12E-01  & 0.82  &  2.24E-01  &  0.87 &  1.68E-01  &  0.81 &  5.55E-01  &  0.80   \\
          \hline\hline\\[1pt]
  \end{tabular}
  \end{table}

\subsubsection{Perturbation of the isentropic vortex}

In this section, we present a test case for the Euler equations that consists in a perturbation
of the isentropic vortex presented in the previous section.
The initial conditions for the three schemes FV-1, FV-2 and GF are taken as the final results
$q_{\text{eq}}$ of the respective simulations run until final time $t_f=50$ with a $80\times 80$ mesh.

Then, we add to the initial conditions a density perturbation $\delta\rho$ centered in $(4,4)$ 
of the form:
$$ \delta\rho = A e^{-\frac{(x-4)^2 + (y-4)^2}{\sigma^2}}$$
where $A = 5 \cdot 10^{-3}$ and $\sigma = 0.8$.
The simulation is run until a final time $t = 2$ to compare the effect of the numerical viscosity
on the evolution of the perturbation.

In figure \ref{fig:perturbationSHU}, we show the density contour plot at the final time for the three methods.
The GF method is able to capture the perturbation sharply, 
while the FV-1 and FV-2 methods have discretization errors too large to capture it properly.
By looking at the isocontours scales, it is clear that the perturbation is completely dissipated for 
the FV-1 method, while for the FV-2 method the perturbation is still visible but with a much larger 
error compared to the expected solution.

\begin{figure}
  \centering
  \subfigure[FV-1]{\includegraphics[width=0.3\textwidth,height=0.25\textwidth,height=0.25\textwidth,trim={0cm 0cm 51cm 0cm},clip]{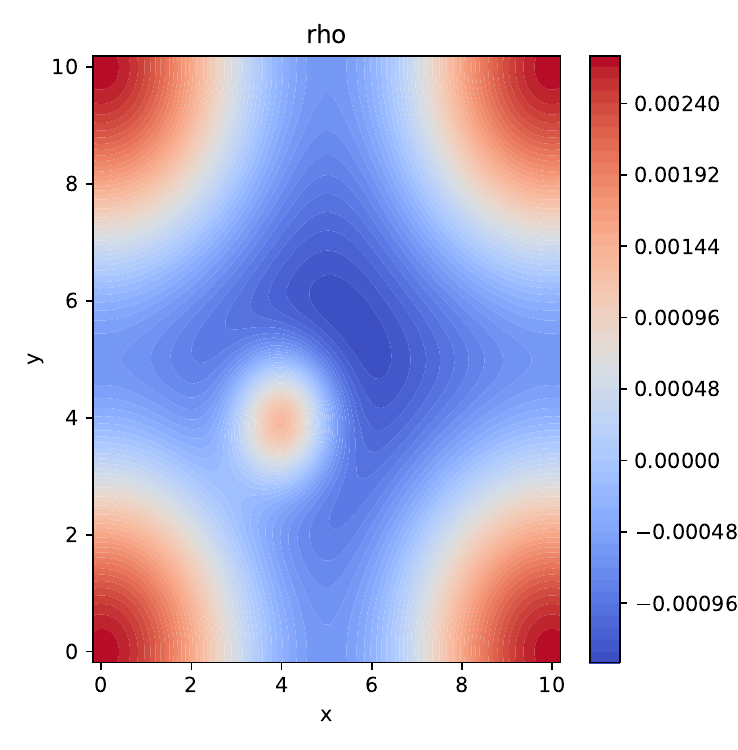}}\qquad
  \subfigure[FV-2]{\includegraphics[width=0.3\textwidth,height=0.25\textwidth,trim={0cm 0cm 51cm 0cm},clip]{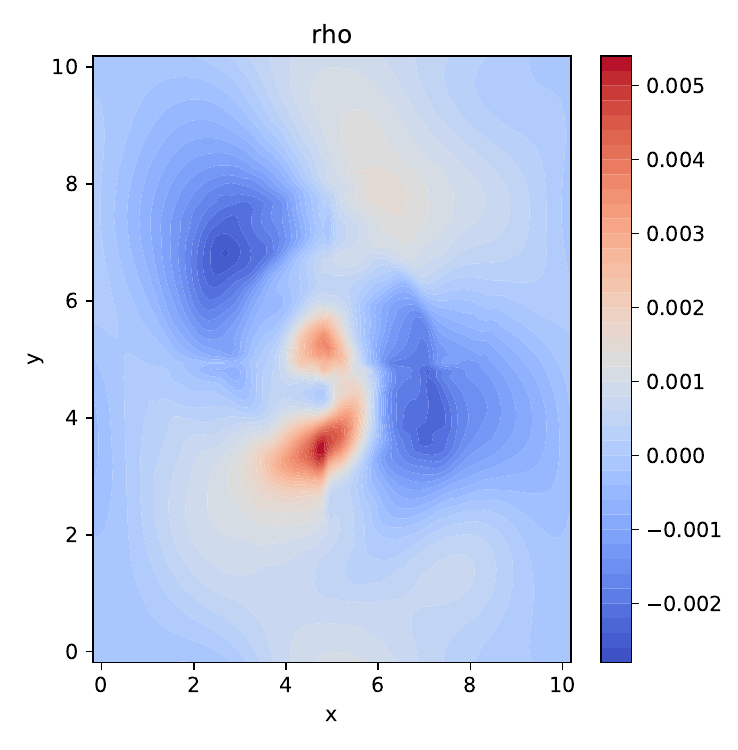}}\qquad
  \subfigure[GF]  {\includegraphics[width=0.3\textwidth,height=0.25\textwidth,trim={0cm 0cm 51cm 0cm},clip]{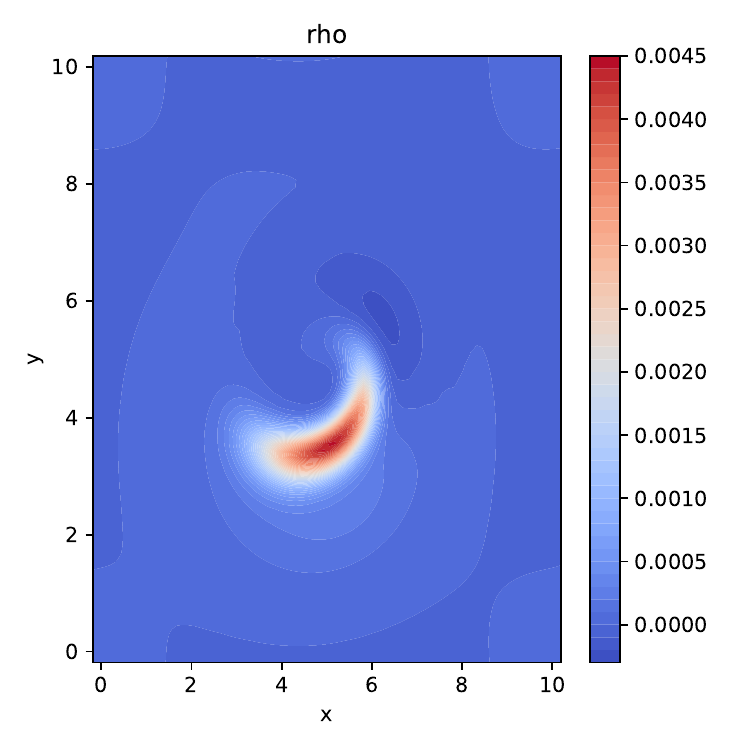}}
  \caption{Euler equations: perturbation of the isentropic vortex. Isocontours of the $\rho-\rho_{\text{eq}}$ norm obtained with FV-1, FV-2 and GF at final time $t_f=2$ with a $80\times 80$ mesh.}
  \label{fig:perturbationSHU}
\end{figure}

\subsubsection{Sod's circular problem}
Here, we test the robustness of the global flux method on the Euler equations for the Sod circular problem.
\new{This case is fully non-linear, and with Mach number of order one, allowing to show that the new method has ``all Mach'' capabilities.}
The problem is a two-dimensional extension of the classical shock tube problem.
The simulation is performed on a domain $[-1,1]\times[-1,1]$ and the initial condition is 
given by 
\begin{align*}
  Q(\mathbf{x},0) = 
  \begin{cases} 
    {Q}_i \quad\text{if}\quad r \leq R, \\
    {Q}_e \quad\text{if}\quad r > R, 
  \end{cases}
\end{align*}
with $r = \sqrt{x^2+y^2}$. The circle of radius $R=0.5$ is centered in the origin and separates
the inner state ${Q}_i$ from the outer state ${Q}_e$, where ${Q}=(\rho,u,v,p)$. The initial conditions
are given by ${Q}_i=(1,0,0,1)$ and ${Q}_e=(0.125,0,0,0.1)$. For a reference solution of this problem, we refer to \cite{boscheri2013arbitrary}.

\begin{figure}
	\centering
	\includegraphics[width=0.055\textwidth]{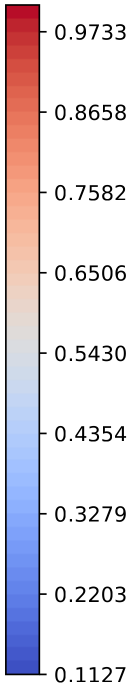}
	\subfigure[FV-1 density]{\includegraphics[width=0.29\textwidth]{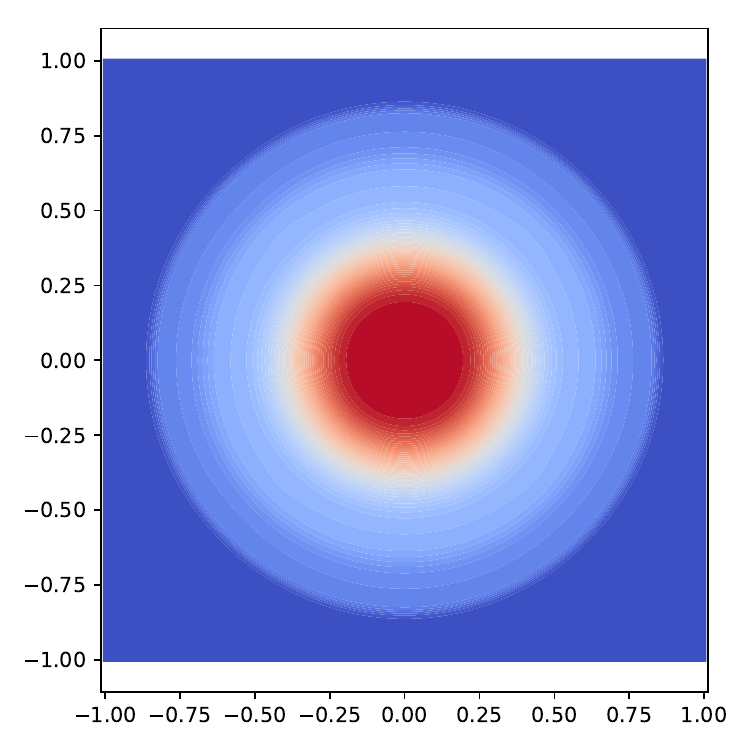}}
	\subfigure[FV-2 density]{\includegraphics[width=0.29\textwidth]{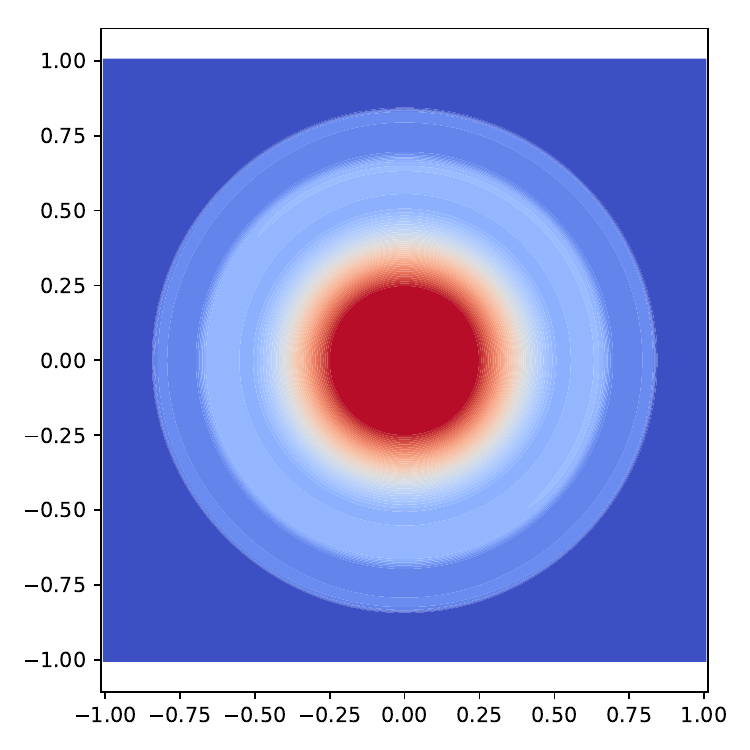}}
	\subfigure[GF density]{\includegraphics[width=0.29\textwidth]{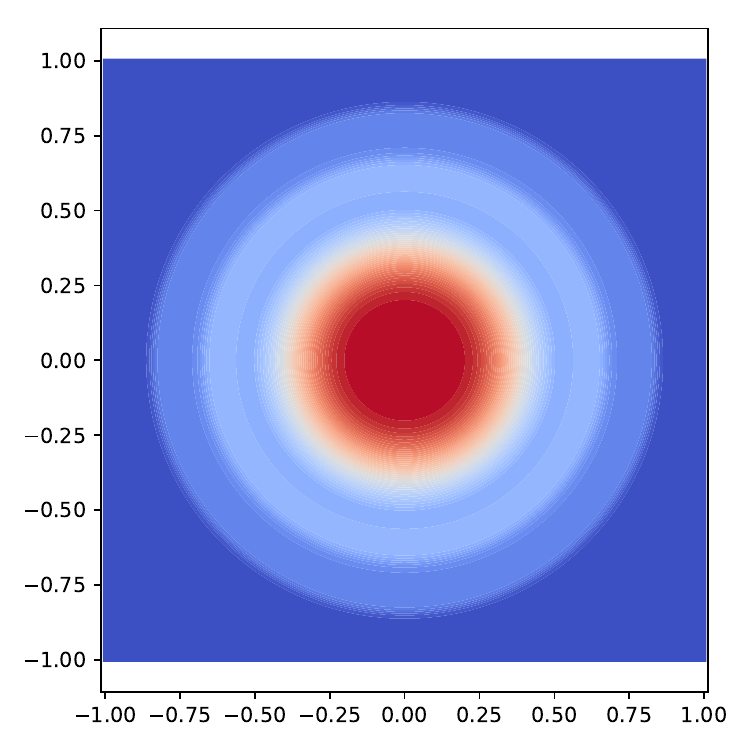}}
	\\
	\subfigure[Slice of $\rho$ at $x=0$]{
		\begin{tikzpicture}
			\begin{axis}[
				grid=major,
				xlabel={$y$},
				ylabel={$\rho$},
				xmin = -0.05,
				xlabel shift = 1 pt,
				ylabel shift = 1 pt,
				legend pos= north east,
				legend style={nodes={scale=0.6, transform shape}},
				tick label style={font=\scriptsize},
				width=.3\textwidth
				]
				\addplot[thick,magenta]   table [x=y, y=rho]{final_sol_GF_Sod_400x400_slice.dat};
				\addlegendentry{GF}
				\addplot[thick,blue]   table [x=y, y=rho]{final_sol_Sod_400x400_order1_slice.dat};
				\addlegendentry{FV-1}
				\addplot[thick,gray]   table [x=y, y=rho]{final_sol_Sod_400x400_order2_slice.dat};
				\addlegendentry{FV-2}
			\end{axis}
		\end{tikzpicture}
	}
	\subfigure[Slice of $v$ at $x=0$]{
		\begin{tikzpicture}
			\begin{axis}[
				grid=major,
				xlabel={$y$},
				ylabel={$v$},
				xmin = -0.05,
				xlabel shift = 1 pt,
				ylabel shift = 1 pt,
				legend pos= north west,
				legend style={nodes={scale=0.6, transform shape}},
				tick label style={font=\scriptsize},
				width=.3\textwidth
				]
				\addplot[thick,magenta]   table [x=y, y expr=\thisrow{rhov}/\thisrow{rho}]{final_sol_GF_Sod_400x400_slice.dat};
				\addlegendentry{GF}
				\addplot[thick,blue]   table [x=y,  y expr=\thisrow{rhov}/\thisrow{rho}]{final_sol_Sod_400x400_order1_slice.dat};
				\addlegendentry{FV-1}
				\addplot[thick,gray]   table [x=y,  y expr=\thisrow{rhov}/\thisrow{rho}]{final_sol_Sod_400x400_order2_slice.dat};
				\addlegendentry{FV-2}
			\end{axis}
		\end{tikzpicture}
	}
	\subfigure[Slice of $p$ at $x=0$]{
	\begin{tikzpicture}
		\begin{axis}[
			grid=major,
			xlabel={$y$},
			ylabel={$p$},
			xmin = -0.05,
			xlabel shift = 1 pt,
			ylabel shift = 1 pt,
			legend pos= north east,
			legend style={nodes={scale=0.6, transform shape}},
			tick label style={font=\scriptsize},
			width=.3\textwidth
			]
			\addplot[thick,magenta]   table [x=y, y expr=10./4.*(\thisrow{rhoE}-\thisrow{rhou}*\thisrow{rhou}/\thisrow{rho}/2-\thisrow{rhov}*\thisrow{rhov}/\thisrow{rho}/2)]{final_sol_GF_Sod_400x400_slice.dat};
			\addlegendentry{GF}
			\addplot[thick,blue]   table [x=y,  y expr=10./4.*(\thisrow{rhoE}-\thisrow{rhou}*\thisrow{rhou}/\thisrow{rho}/2-\thisrow{rhov}*\thisrow{rhov}/\thisrow{rho}/2)]{final_sol_Sod_400x400_order1_slice.dat};
			\addlegendentry{FV-1}
			\addplot[thick,gray]   table [x=y,  y expr=10./4.*(\thisrow{rhoE}-\thisrow{rhou}*\thisrow{rhou}/\thisrow{rho}/2-\thisrow{rhov}*\thisrow{rhov}/\thisrow{rho}/2)]{final_sol_Sod_400x400_order2_slice.dat};
			\addlegendentry{FV-2}
		\end{axis}
	\end{tikzpicture}
}
	\caption{Euler equations: Sod's circular problem. Numerical results obtained on a $400\times 400$ mesh with FV-1, FV-2 and GF methods run until a final time $t_f=0.2$.}
	\label{fig:Euler_Sod}
\end{figure}%
To have a smoother initial condition, the two states are connected by a smooth transition region
given by an erfc function defined as
$$ \zeta(r) = \frac12 \text{erfc}\left(\frac{r-R}{\delta}\right),$$
where $\delta=0.01$. Therefore, we can define the smoothed initial condition as
$$ {Q}(r,0) = \zeta(r) {Q}_i + (1-\zeta(r)) {Q}_e.$$
The simulation is run until final time $t_f=0.2$ before the shock waves reach the boundaries.

In figure \ref{fig:Euler_Sod}, we present the numerical results obtained on a $400\times 400$ mesh with 
the three methods FV-1, FV-2 and GF. For all situations, we show the density contour plot at the 
final time, along with a slice of the density and vertical velocity on the $x=0$ axis.
It can be noticed that, among all three simulations, the GF performs much better than the standard first order scheme and it is clearly comparable to a second order one. It is able to sharply capture the three
waves, which are smoothed out by the classical FV-1. GF also avoids oscillations at the beginning of the rarefaction, while the FV-2 shows small oscillations. On the foot of the rarefaction, the GF show sharper results, while it is a little more diffusive on the contact discontinuity with respect to the FV-2. On the shock, the GF does not oscillate, while the FV-2 shows minimal oscillations and is a little more sharply representing the discontinuity.

This test case not only allows us to show the robustness of the method to deal with unsteady shock propagation. It also provides interesting insights into its low dissipation even though the method has not been designed to have any particular properties on unsteady solutions.

\subsubsection{Kelvin-Helmholtz instability}
\new{We consider  a smooth Kelvin-Helmholtz instability for the Euler equations, }
introduced in \cite{leidi2024performance} to \new{ further confirm the ability of our approach to correctly cope with the low Mach limit.}
 numerical scheme to cope with low Mach number flow and to assess qualitatively the numerical diffusion of the method. 
 There is a large body of work available in the literature concerning the shortcomings of classical Finite Volume methods in the subsonic regime (see e.g.\ \cite{barsukow2021truly}). 
 The effect of stabilizing diffusion becomes bigger as the Mach number decreases, making it necessary to use highly resolved grids in order to capture the features of the flow. Numerical methods that are not low Mach number compliant typically also stabilize Kelvin-Helmholtz setups in an artificial way.

The simulations are performed in the domain $[0,2]\times[-1/2,1/2]$ until a final time $t_f=80$.
The initial condition is given by the following primitive variables: 
\begin{align*}
\rho &= \gamma + \mathcal H(y) \, r, \qquad
u    = M \, \mathcal H(y), \qquad
v    = \delta \, M \, \sin(2\pi x), \qquad
p    = 1,
\end{align*}
where the Mach number parameter is $M=10^{-2}$ and we use $r=10^{-3}$ and $\delta=0.1$.
The function $\mathcal H(y)$ is defined as, 
\begin{equation*}
  \mathcal H(y) =
  \begin{cases}
    -\sin\left(\frac{\pi}{\omega} \left(y+\frac14\right)\right), &\quad\text{if} \quad  -\frac14-\frac\omega2 \leq y < -\frac14+\frac\omega2, \\
    -1,&\quad\text{if} \quad  -\frac14+\frac\omega2 \leq y < \frac14-\frac\omega2, \\
    \sin\left(\frac{\pi}{\omega} \left(y-\frac14\right)\right), &\quad\text{if} \quad  \frac14-\frac\omega2 \leq y < \frac14+\frac\omega2, \\
    1 &\quad\text{else}, 
  \end{cases}
\end{equation*}
where $\omega=1/16$. Observe that the shear flow is smooth such that for short times, there exists a solution to which numerical methods converge upon mesh refinement (\cite{leidi2024performance}).
\begin{figure}
	\centering
	\subfigure[FV-1 ($64\times32$)]   {\includegraphics[width=0.29\textwidth,trim={1cm 0cm 1cm 8mm},clip]{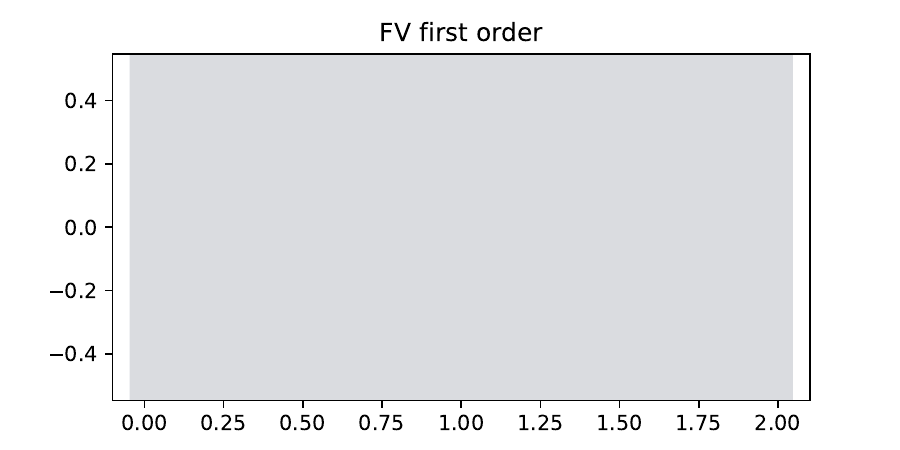}}
	\subfigure[FV-2 ($64\times32$)]   {\includegraphics[width=0.29\textwidth,trim={1cm 0cm 1cm 8mm},clip]{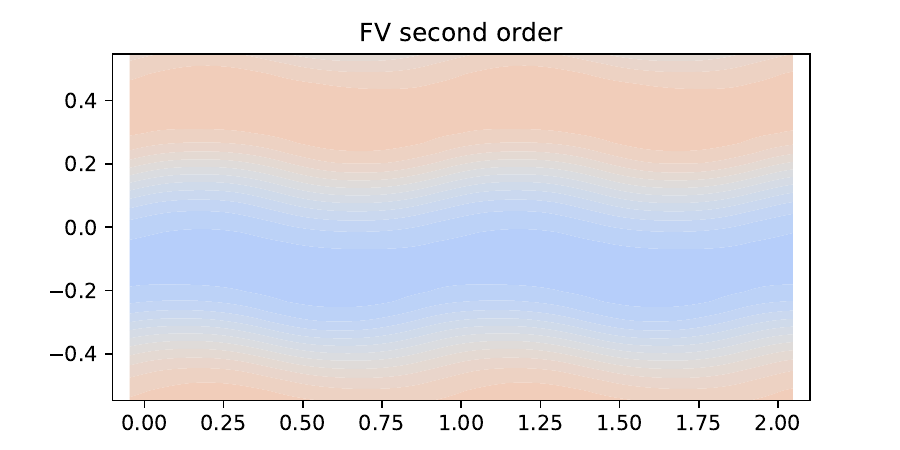}}
	\subfigure[GF ($64\times32$)]   {\includegraphics[width=0.29\textwidth,trim={1cm 0cm 1cm 8mm},clip]{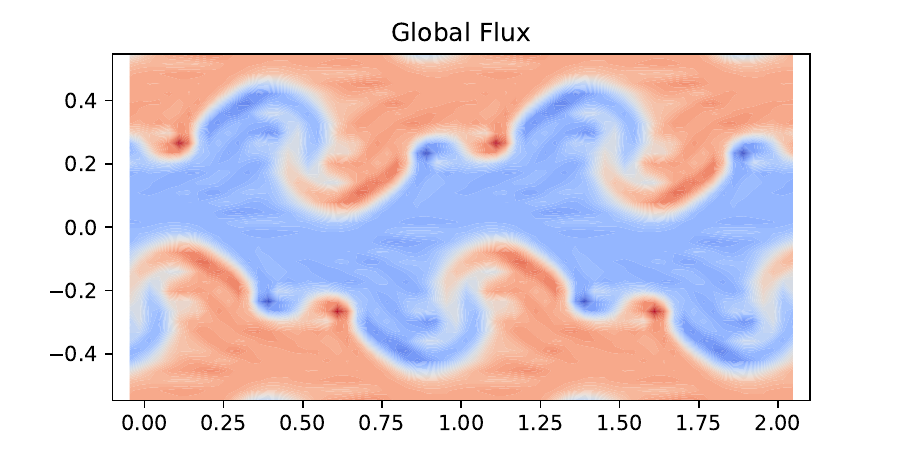}}
	\includegraphics[height=0.15\textwidth]{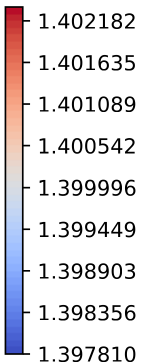}
	\\
		\subfigure[FV-1 ($128\times64$)]   {\includegraphics[width=0.29\textwidth,trim={1cm 0cm 1cm 8mm},clip]{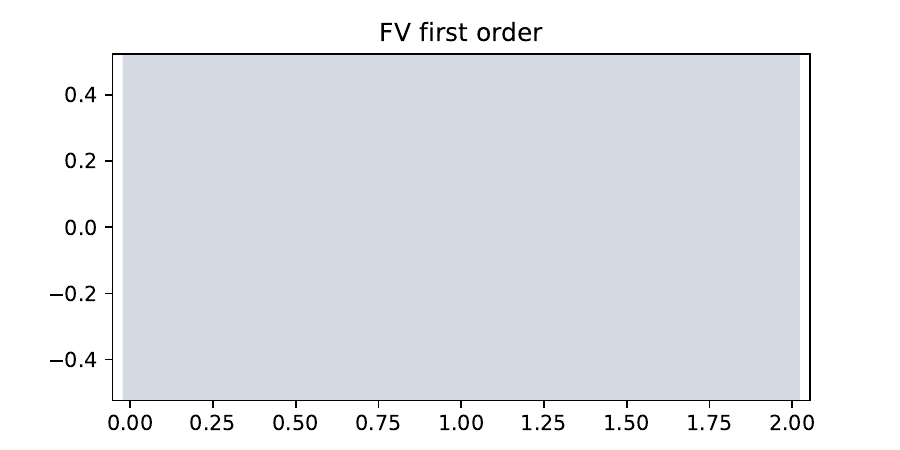}}
	\subfigure[FV-2 ($128\times64$)]   {\includegraphics[width=0.29\textwidth,trim={1cm 0cm 1cm 8mm},clip]{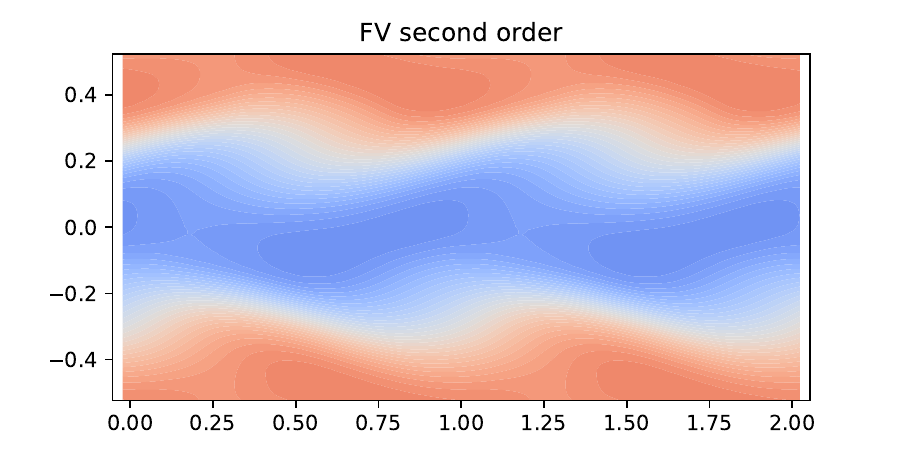}}
	\subfigure[GF ($128\times64$)]   {\includegraphics[width=0.29\textwidth,trim={1cm 0cm 1cm 8mm},clip]{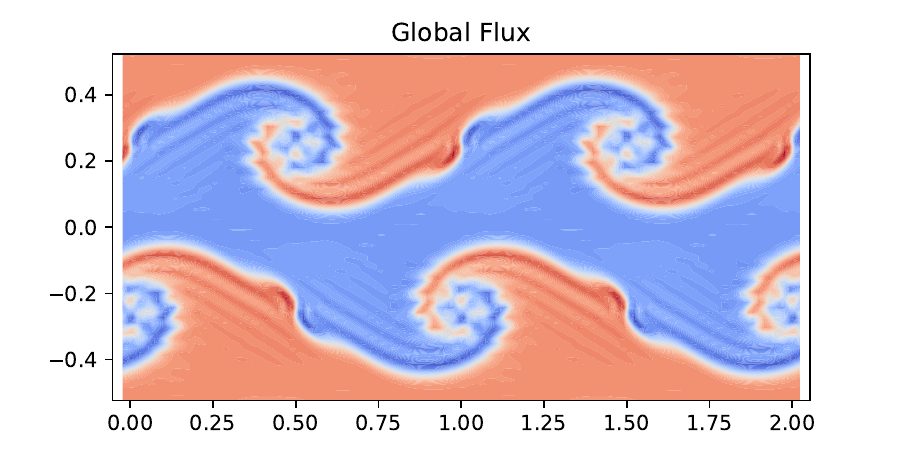}}
	\includegraphics[height=0.15\textwidth]{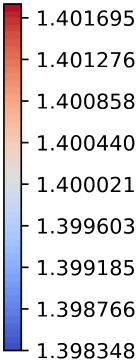}
	\\
		\subfigure[FV-1 ($256\times128$)]   {\includegraphics[width=0.29\textwidth,trim={1cm 0cm 1cm 8mm},clip]{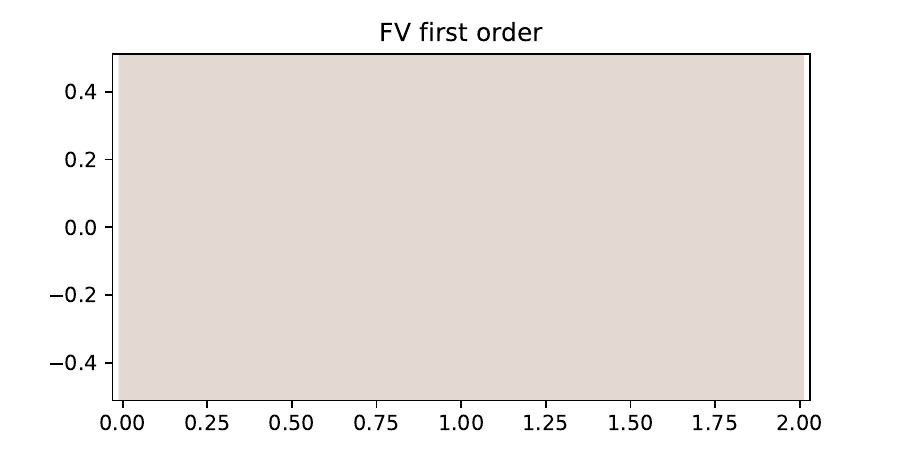}}
	\subfigure[FV-2 ($256\times128$)]   {\includegraphics[width=0.29\textwidth,trim={1cm 0cm 1cm 8mm},clip]{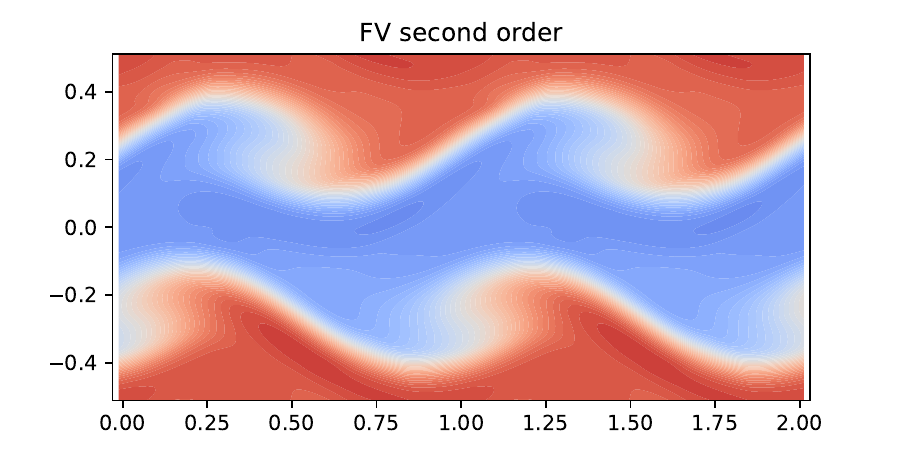}}
	\subfigure[GF ($256\times128$)]   {\includegraphics[width=0.29\textwidth,trim={1cm 0cm 1cm 8mm},clip]{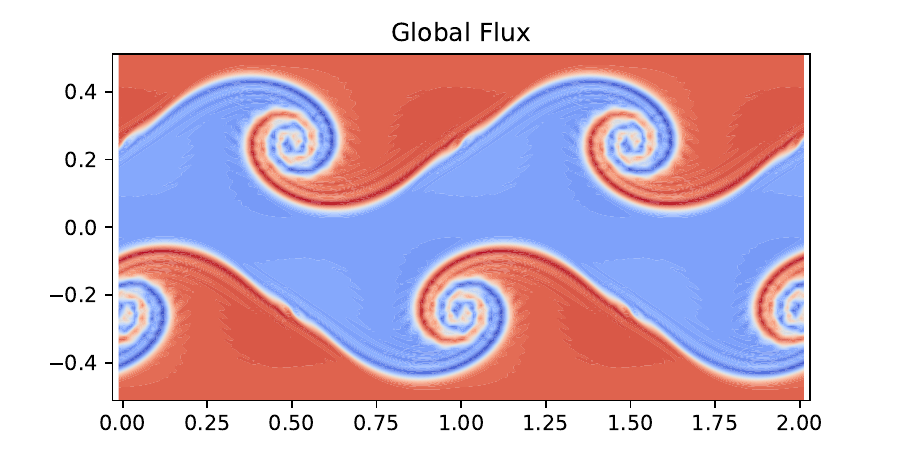}}
	\includegraphics[height=0.15\textwidth]{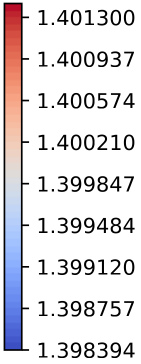}
	\\
		\subfigure[FV-1 ($512\times 256$)]   {\includegraphics[width=0.29\textwidth,trim={1cm 0cm 1cm 8mm},clip]{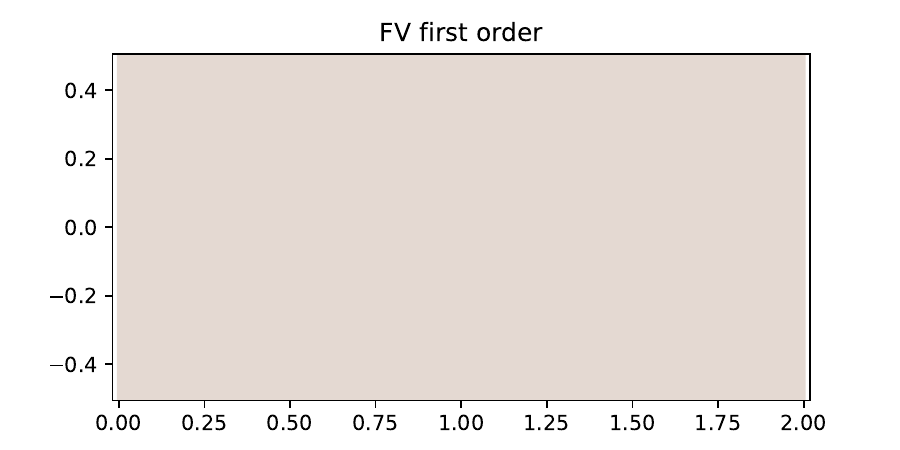}}
	\subfigure[FV-2 ($512\times 256$)]   {\includegraphics[width=0.29\textwidth,trim={1cm 0cm 1cm 8mm},clip]{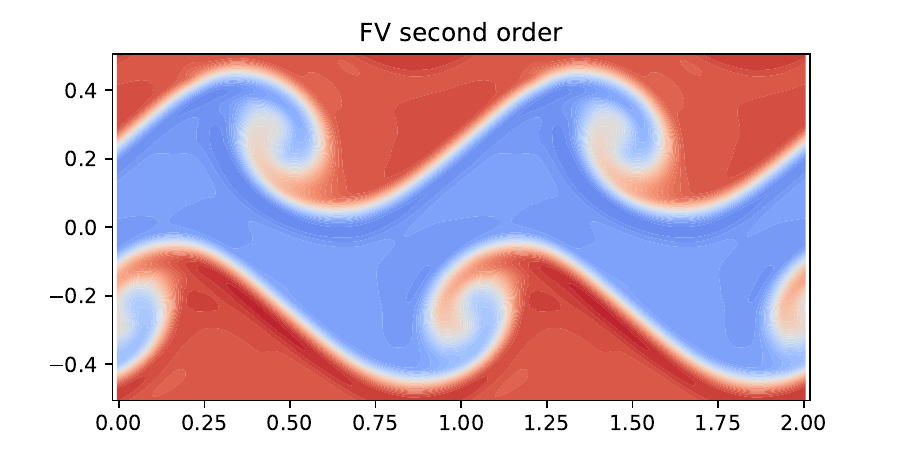}}
	\subfigure[GF ($512\times 256$)]   {\includegraphics[width=0.29\textwidth,trim={1cm 0cm 1cm 8mm},clip]{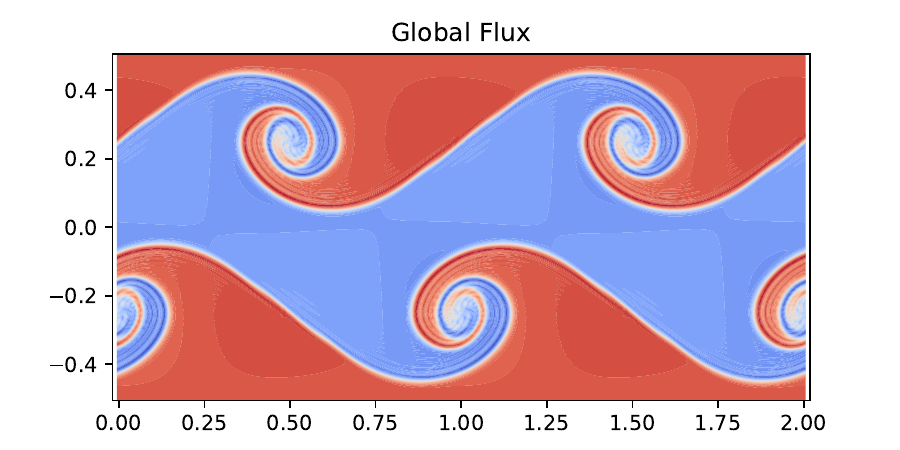}}
	\includegraphics[height=0.15\textwidth]{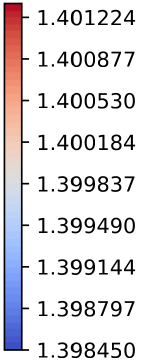}
	\caption{Euler equations: Kelvin-Helmholtz instability. Density isocontours are presented for a set of nested meshes to compare FV-1 (top), FV-2 (middle) and GF (bottom).}
	\label{fig:Euler_KH}
\end{figure}

In figure \ref{fig:Euler_KH}, we present the numerical results obtained with FV-1, FV-2 and GF for the
Kelvin-Helmholtz instability arising from the aforementioned initial conditions. The simulations are performed on a set of 4 nested
grids from a $64\times 32$, the coarsest, to $512\times 256$, the finest.

The FV-1 scheme is not able to capture any of the features arising from the instability.
No vortices form, since FV-1 is not low Mach number compliant.

Much improved results are obtained using the FV-2 with a linear reconstruction of the conservative variables. Here, the higher order of accuracy helps to overcome excessive diffusion at this Mach number and for this simulation time.
However, the structures still appear diffused and would need even more resolution for the vortex details to be captured.

Very differently from these methods, the GF method is able to capture all details of the flow very accurately. 
Already on the coarsest mesh, the fluid structures start to appear and develop. Here, some spurious vortices are visible, which are a known artefact of virtually any numerical method (see e.g. \cite{brown95}).
When increasing the resolution, the fluid features converge to 
the expected solution found in other references \cite{leidi2024performance}. Comparison to the results obtained with low Mach compliant methods studied in \cite{leidi2024performance} shows that the GF method is at least as good.

It has been suggested in \cite{barsukow2019stationarity} that numerical methods for the Euler equations whose linearization ($=$ method for linear acoustics) is stationarity preserving, are low Mach number compliant. A nonlinear stationarity preserving method naturally has this property, and some experimental examples of this behavior can also be found in \cite{barsukow2018low}. Thus, even though we set out to improve the performance of the numerical method at stationary state, here we observe that this property is beneficial even for solutions far away from it.

%
\subsection{Shallow water system}

\subsubsection{Potential flow}

The first test case implemented for the shallow water equations is an equilibrium (see \cite{ricchiuto2009stabilized})
characterized by a known exact solution, for which it is possible to perform a convergence analysis. 
The initial condition is a potential flow defined on the square $[-1,1]\times[-1,1]$ with Dirichlet boundary conditions, which is given by
\begin{align*}
  h(x,y) &= (x-x_0)(y-y_0) + C,\\
  u(x,y) &= (x-x_0),\\
  v(x,y) &=-(y-y_0) 
\end{align*}
where $C=3/2$ and $(x_0,y_0)=(0,0)$.
The 2D equilibrium is achieved thanks to a special bathymetry given by
\begin{align*}
b(x,y) = \frac1g \left(30-\frac{x^2+y^2}{2}\right) -xy -C.
\end{align*}

The solution of this potential flow is shown in figure \ref{fig:SW_potentialflow}, and the convergence rates
computed at final time $t_f=1$ are presented in table \ref{tab:SWpotentialconv_T1}, 
demonstrating the improvement brought about by the global flux formulation. 
Again, since the setup is stationary, superconvergence is observed.
Moreover, the new method is even able to outperform FV-2.

\begin{figure}
  \center
  \includegraphics[width=0.9\textwidth,height=0.25\textwidth,trim={0cm 0cm 11.5cm 0cm},clip]{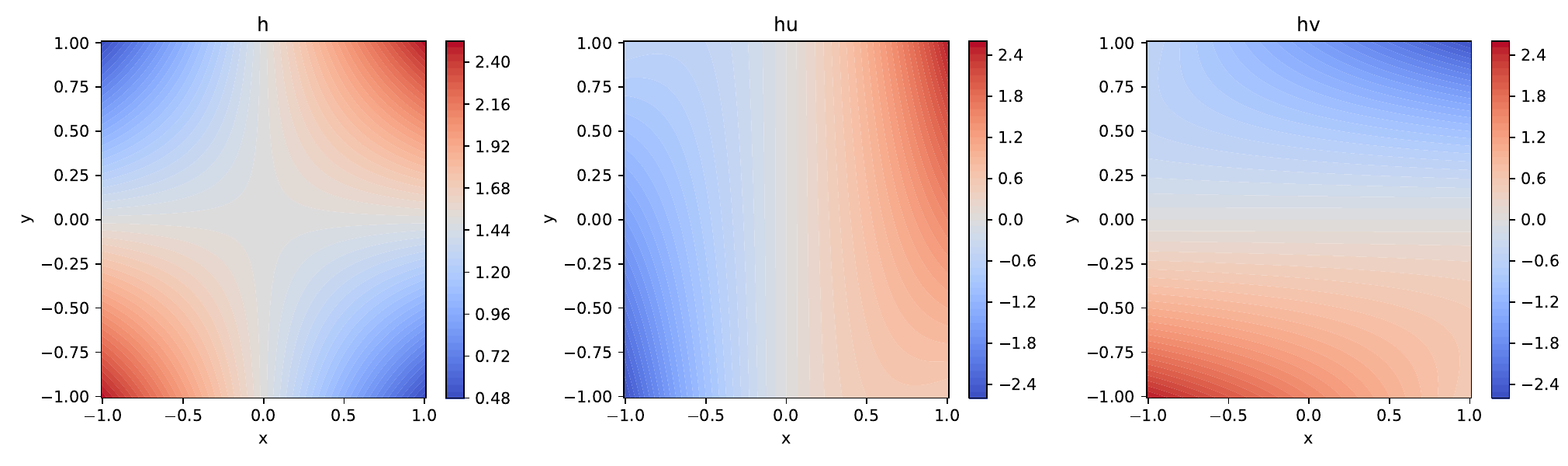}
  \caption{Shallow water system: potential flow. Reference solution of the conservative variables.}
  \label{fig:SW_potentialflow}
\end{figure}

\begin{table}
  \caption{Shallow water system: potential flow ($t_f=1$). $L_2$ error and order of accuracy $\tilde{n}$ for FV-1, FV-2 and GF. }\label{tab:SWpotentialconv_T1}
  \footnotesize
  \centering
  \begin{tabular}{ccccccc}
          \hline\hline
          &\multicolumn{2}{c}{$h$} &\multicolumn{2}{c}{$hu$}   &\multicolumn{2}{c}{$hv$}\\[0.5mm]
          \cline{2-7}
          $N_x,N_y$ & $L_2$        & $\tilde{n}$ & $L_2$        & $\tilde{n}$ & $L_2$      & $\tilde{n}$ \\[0.5mm]\hline
          &\multicolumn{6}{c}{FV-1}\\[0.5mm]
          20  &  1.54E-02  &  --  & 1.57E-01   &  --  & 1.71E-01 &  --    \\
          40  &  8.25E-03  & 0.89 & 1.08E-01   & 0.53 & 1.11E-01 & 0.62   \\
          80  &  4.30E-03  & 0.94 & 6.60E-02   & 0.71 & 6.43E-02 & 0.78   \\
          160 &  2.18E-03  & 0.97 & 3.68E-02   & 0.84 & 3.47E-02 & 0.88   \\
          320 &  1.10E-03  & 0.99 & 1.95E-02   & 0.91 & 1.80E-02 & 0.94   \\
          &\multicolumn{6}{c}{FV-2}\\[0.5mm]
          20  &  2.49E-04  &  --  & 1.06E-03   &  --  & 1.47E-03 &  --   \\
          40  &  5.26E-05  & 2.24 & 2.61E-04   & 2.02 & 3.25E-04 &  2.17 \\
          80  &  1.09E-05  & 2.27 & 7.11E-05   & 1.87 & 8.17E-05 &  1.99 \\
          160 &  2.24E-06  & 2.28 & 1.86E-05   & 1.93 & 2.06E-05 &  1.98 \\
          320 &  4.81E-07  & 2.21 & 4.69E-06   & 1.98 & 5.17E-06 &  1.99 \\
          &\multicolumn{6}{c}{GF}\\[0.5mm]
          20  &  1.15E-04 &  --   &  4.29E-04   &  --   &  1.11E-03 & --  \\
          40  &  2.69E-05 & 2.09  &  1.01E-04   &  2.08 &  2.39E-04 & 2.21 \\
          80  &  6.49E-06 & 2.05  &  2.46E-05   &  2.03 &  5.50E-05 & 2.11 \\
          160 &  1.59E-06 & 2.02  &  6.08E-06   &  2.01 &  1.32E-05 & 2.05 \\
          320 &  3.95E-07 & 2.01  &  1.51E-06   &  2.00 &  3.24E-06 & 2.02 \\
          \hline\hline\\[1pt]
  \end{tabular}
  \end{table}

\subsubsection{Lake at rest} 

In this section, we test the well-balanced property, proven in section ~\ref{sec:sourceSW}, of the global flux method 
for lake at rest solutions of the shallow water system.
The problem is set in a rectangular domain $[0,1]\times[0,1]$ with periodic boundary conditions.
The initial and exact solution is given by
$$ h(x,y) = 1 - b(x,y), \qquad u(x,y)=v(x,y)\equiv 0,$$
where the bathymetry is defined as  
$$ b(x,y) = 0.1\sin(2\pi x)\cos(2\pi y). $$
In table \ref{tab:SWlakeatrest}, a convergence study is presented at final time $t_f=0.1$.
As expected, thanks to the well-balanced property of the global flux method, the GF is able to achieve machine precision errors.
The standard FV-1 and FV-2 methods show only the classical first and second order convergence slopes.
In figure \ref{fig:lar}, we present the comparison between the well-balanced GF method and the non-well-balanced FV-1 and FV-2 methods.
\begin{table}
  \caption{Shallow water system: lake at rest ($t_f=0.1$). $L_2$ error and order of accuracy $\tilde{n}$ for FV-1 and FV-2 schemes with the novel GF.}\label{tab:SWlakeatrest}
  \footnotesize
  \centering
  \begin{tabular}{ccccccc}
          \hline\hline
          &\multicolumn{2}{c}{$h$} &\multicolumn{2}{c}{$hu$}   &\multicolumn{2}{c}{$hv$}\\[0.5mm]
          \cline{2-7}
          $N_x,N_y$ & $L_2$        & $\tilde{n}$ & $L_2$        & $\tilde{n}$ & $L_2$      & $\tilde{n}$ \\[0.5mm]\hline
          &\multicolumn{6}{c}{FV-1}\\[0.5mm]
          20  & 7.13E-03  &  --  & 4.67E-02   &  --  & 3.61E-02  &  --   \\
          40  & 2.79E-03  & 1.35 & 2.42E-02   & 0.95 & 2.04E-02  & 0.82  \\
          80  & 1.20E-03  & 1.22 & 1.21E-02   & 1.00 & 1.10E-02  & 0.89  \\
          160 & 5.46E-04  & 1.12 & 5.99E-03   & 1.00 & 5.71E-03  & 0.94  \\
          320 & 2.60E-04  & 1.06 & 2.98E-03   & 1.00 & 2.91E-03  & 0.97  \\
          &\multicolumn{6}{c}{FV-2}\\[0.5mm]
          20  & 2.92E-03 &  --  &  2.17E-03  &  --  & 2.26E-03  &   --  \\
          40  & 6.76E-04 & 2.10 &  3.65E-04  & 2.57 & 3.94E-04  &  2.51 \\
          80  & 1.60E-04 & 2.07 &  7.58E-05  & 2.26 & 8.10E-05  &  2.28 \\
          160 & 3.87E-05 & 2.04 &  1.70E-05  & 2.15 & 1.78E-05  &  2.18 \\
          320 & 9.50E-06 & 2.02 &  4.02E-06  & 2.08 & 4.13E-06  &  2.11 \\
          &\multicolumn{6}{c}{GF}\\[0.5mm]
          20  &  4.50E-16 & -- & 6.24E-15  & --  & 6.36E-15  &  --  \\
          40  &  9.71E-16 & -- & 1.36E-14  & --  & 1.31E-14  &  --  \\
          80  &  1.58E-15 & -- & 3.03E-14  & --  & 3.29E-14  &  --  \\
          160 &  3.38E-15 & -- & 8.62E-14  & --  & 8.67E-14  &  --  \\
          320 &  6.47E-15 & -- & 2.28E-13  & --  & 2.29E-13  &  --  \\
          \hline\hline\\[1pt]
  \end{tabular}
  \end{table}

\begin{figure}
  \center
  \subfigure[FV-1]{\includegraphics[width=0.9\textwidth,height=0.25\textwidth,trim={0cm 0cm 13cm 0cm},clip]{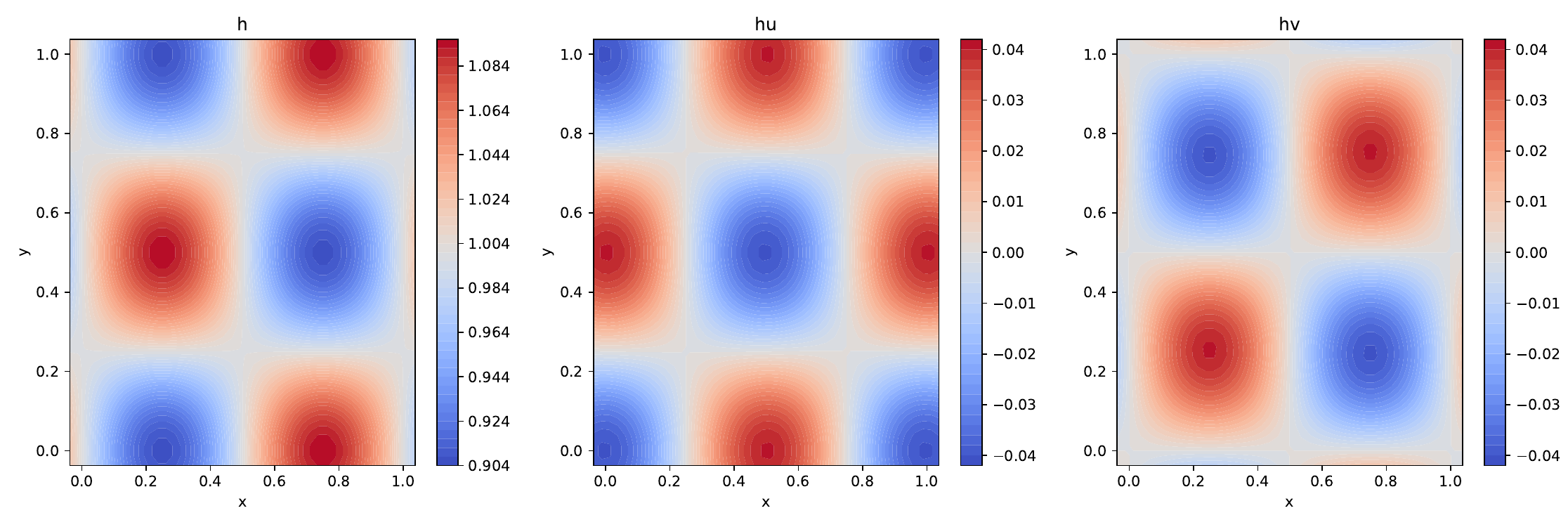}}\qquad
  \subfigure[FV-2]{\includegraphics[width=0.9\textwidth,height=0.25\textwidth,trim={0cm 0cm 13cm 0cm},clip]{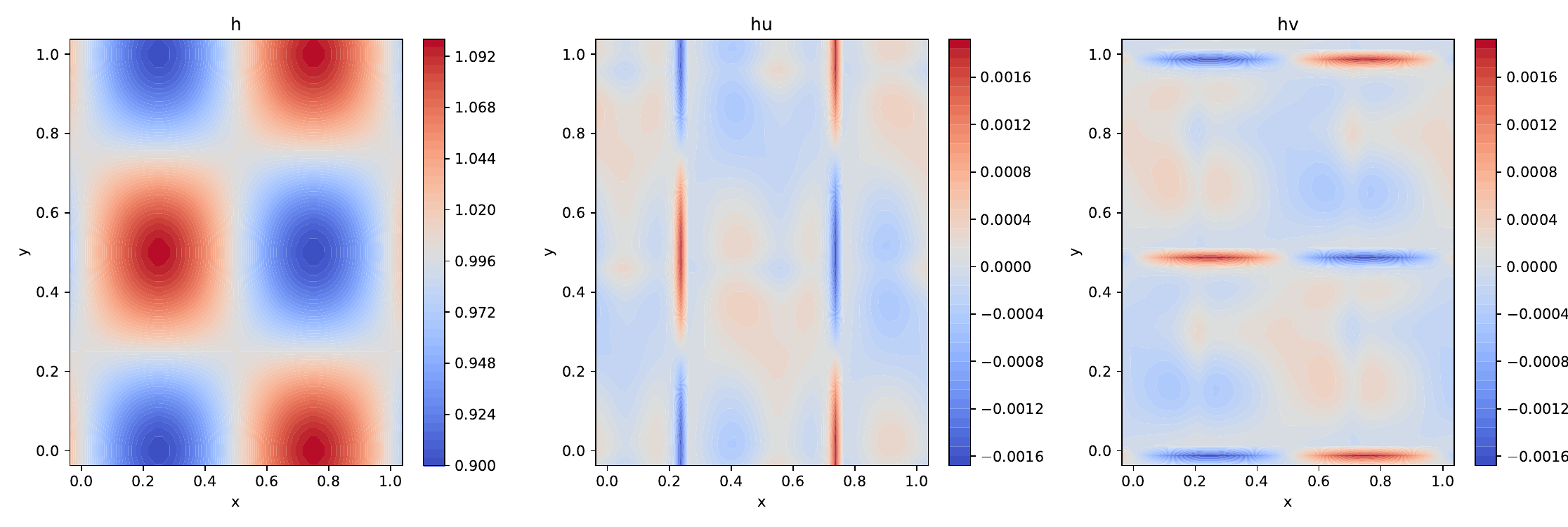}}\qquad
  \subfigure[GF]{\includegraphics[width=0.9\textwidth,height=0.25\textwidth,trim={0cm 0cm 13cm 0cm},clip]{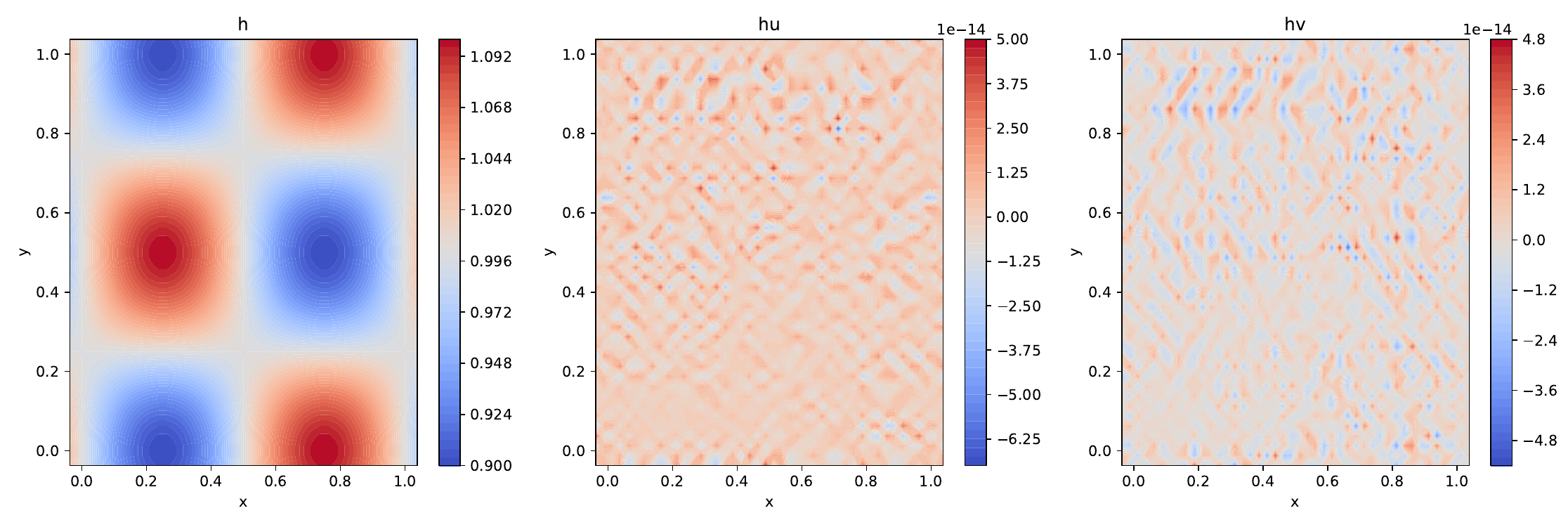}}
  \caption{Shallow water system: lake at rest. Numerical results for the lake at rest solution on the coarse mesh $40\times 40$ at final time $t_f=0.1$ obtained with FV-1, FV-2 and GF.}
  \label{fig:lar}  
\end{figure}

\subsubsection{2D supercritical \RII{moving water} equilibria} 

We consider two fully multi-dimensional \RII{moving water}
steady states of the shallow water system, characterized by constant momentum in 
supercritical regimes. However, contrary to the one-dimensional version of such equilibria \cite{ciallella2023arbitrary,ciallella2025high},
no exact solution is known for the simulations presented in this section. 
For this reason, the numerical results obtained through the three schemes will be compared qualitatively.
The problems are simulated on a rectangular domain $[0,25]\times[0,8]$, and are made fully multi-dimensional
by employing a 2D bathymetry that is a function of both $x$ and $y$, given by
\begin{align*}
b(x,y) = \begin{cases}\frac15 \left(1 - \left(\frac{r(x,y)}{2}\right)^2\right),\quad\text{where}\quad r(x,y)<2 \\
0,\quad\text{elsewhere}
\end{cases} 
\end{align*}
with $r(x,y) = \sqrt{(x-x_0)^2 + (y-y_0)^2}$ and $(x_0,y_0) = (10,4)$.
The initial conditions of the first problem are given by
\begin{alignat*}{3}
&h(x,y,0) = 2 - b(x,y),\qquad &&q_x(x,y,0) = 24, \qquad &&q_y(x,y,0) = 0. 
\end{alignat*}
Inlet boundary conditions (equal to the initial conditions) are imposed on the left boundary of the domain, 
and outlet (homogeneous Neumann) on the right. Top and bottom of the domain are periodic boundaries.
In figure \ref{fig:SW_super}, we present the numerical solutions for the conservative variables  
when the numerical steady state is reached (time residual close to machine precision).
All simulations are performed on a mesh of $450\times 450$ elements.
GF is able to capture and resolve sharply
the many shocks appearing behind the bathymetry bump. 
Although the mesh resolution for this case is quite fine, FV-1 still presents a highly diffused result.
An improvement is experienced when using the linear reconstruction for FV-2, although all the waves 
still appear as smooth transitions. However, the results of the classical methods still remain significantly inferior to those of GF, which captures
all waves sharply in much fewer cells. 
\begin{figure}
  \centering
  \begin{minipage}{0.42\textwidth}
  	\centering $h$, FV-1\\
  	\includegraphics[width=\textwidth, height=0.31\textwidth,trim={0 0 0 0},clip]{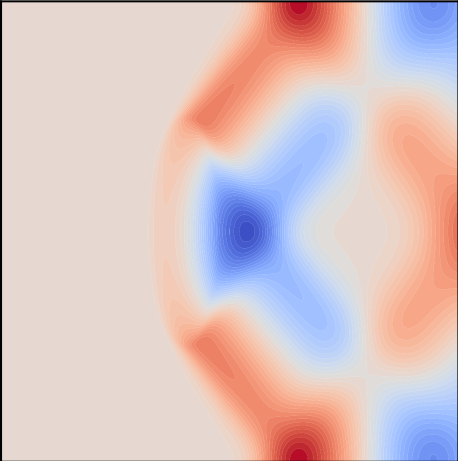}
  	\centering $h$, FV-2\\
	\includegraphics[width=\textwidth, height=0.31\textwidth,trim={0 0 0 0},clip]{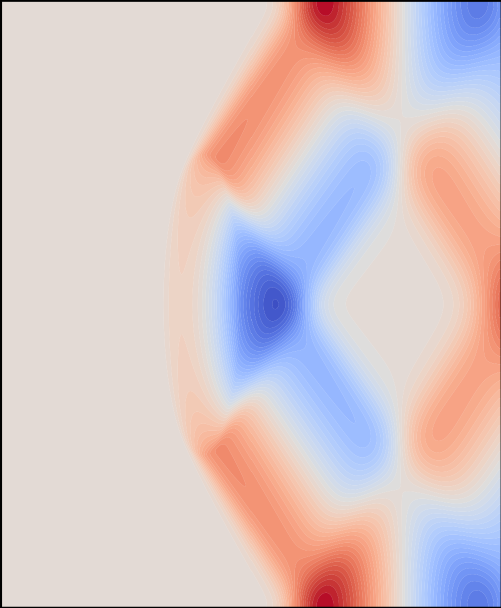}
  	\centering $h$, GF\\
  	\includegraphics[width=\textwidth, height=0.31\textwidth,trim={0 0 0 0},clip]{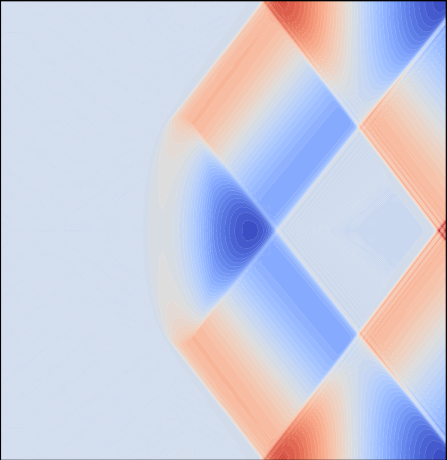}
  \end{minipage}\,
  \begin{minipage}{0.05\textwidth}
  	\includegraphics[width=\textwidth,trim={0 0 0 0},clip]{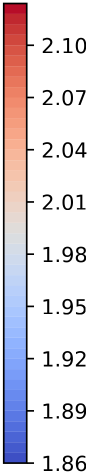}
  \end{minipage}\hfill
  \begin{minipage}{0.42\textwidth}
  	\centering $hv$, FV-1\\
  	\includegraphics[width=\textwidth, height=0.31\textwidth,trim={0 0 0 0},clip]{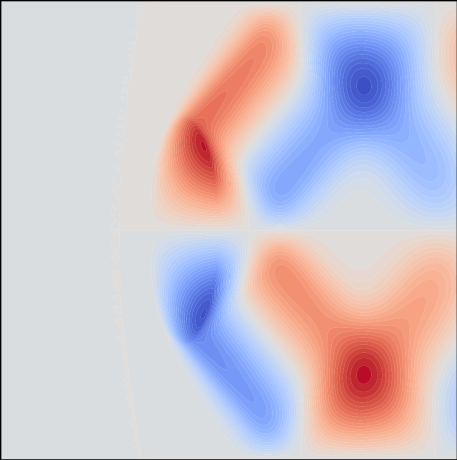}
  	\centering $hv$, FV-2\\
  	\includegraphics[width=\textwidth, height=0.31\textwidth,trim={0 0 0 0},clip]{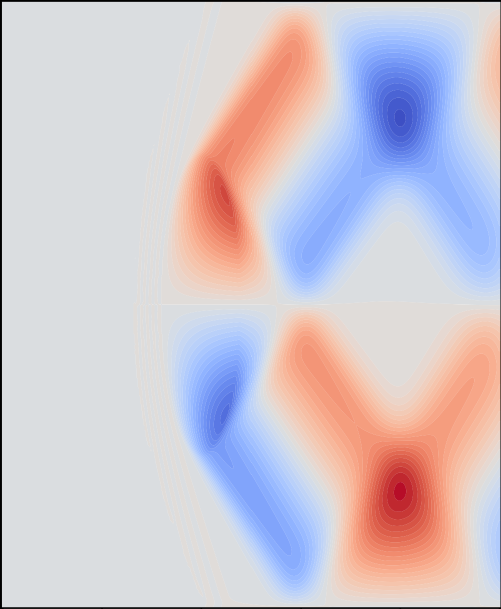}
  	\centering $hv$,   GF\\
  	\includegraphics[width=\textwidth, height=0.31\textwidth,trim={0 0 0 0},clip]{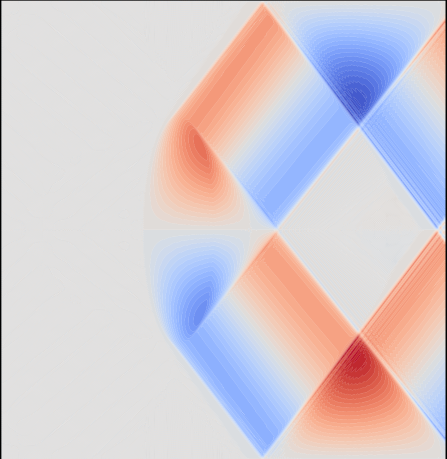}
  \end{minipage}\,
  \begin{minipage}{0.05\textwidth}
  	\includegraphics[width=\textwidth,trim={0 0 0 0},clip]{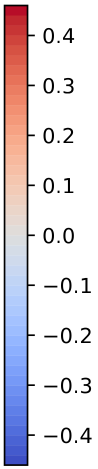}
  \end{minipage}
  \caption{Shallow water system: 2D supercritical equilibria. Numerical results obtained with FV-1, FV-2 and GF to steady state for $N_x=N_y=450$.}
  \label{fig:SW_super}
\end{figure}
%

\begin{figure}
	\centering
	\begin{minipage}{0.42\textwidth}
		\centering $h-h_{eq}$, FV-1 $N_x=150,N_y=50$\\
		\includegraphics[width=\textwidth, height=0.31\textwidth,trim={13.6mm 14.9mm 415mm 8mm},clip]{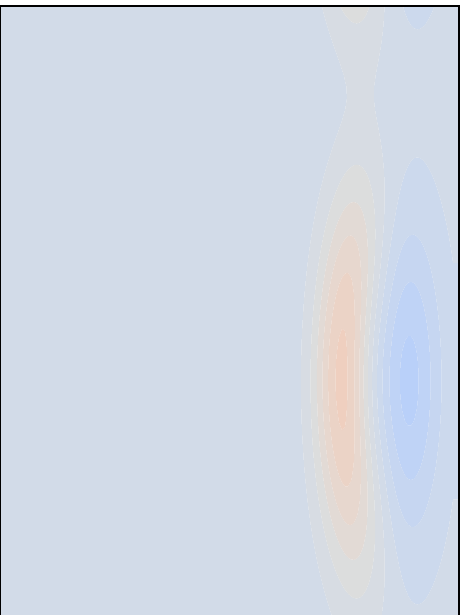}
		\centering $h-h_{eq}$, FV-2 $N_x=150,N_y=50$\\
		\includegraphics[width=\textwidth, height=0.31\textwidth,trim={13.6mm 14.9mm 415mm 8mm},clip]{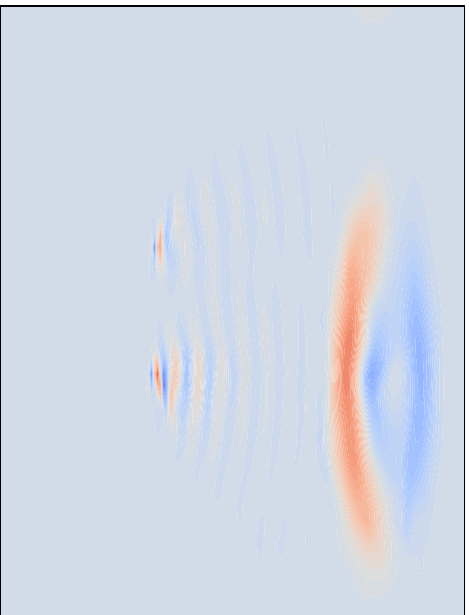}
	\end{minipage}\,
	\begin{minipage}{0.42\textwidth}
		\centering $h-h_{eq}$, GF $N_x=150,N_y=50$\\
		\includegraphics[width=\textwidth, height=0.31\textwidth,trim={13.6mm 14.9mm 415mm 8mm},clip]{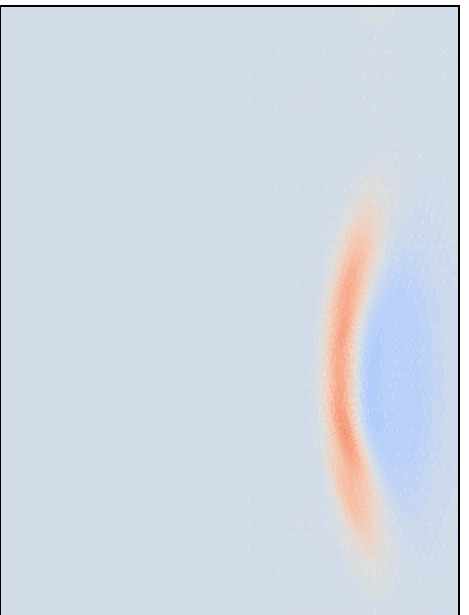}
		\centering $h-h_{eq}$, FV-1 $N_x=450,N_y=450$\\
		\includegraphics[width=\textwidth, height=0.31\textwidth,trim={13.6mm 14.9mm 415mm 8mm},clip]{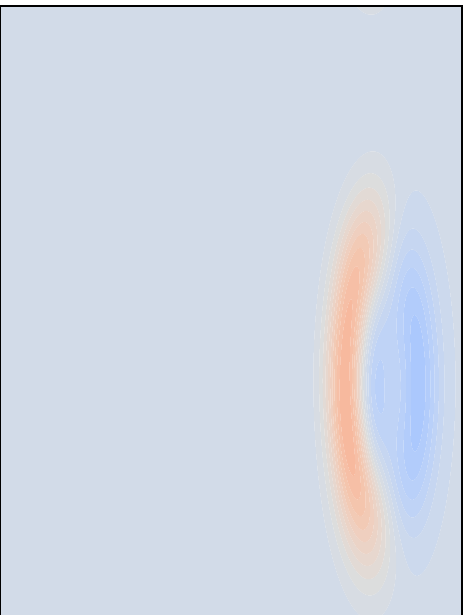}
	\end{minipage}\,
	\begin{minipage}{0.1\textwidth}
		\includegraphics[width=\textwidth,trim={92.5mm 14.2mm 378mm 0mm},clip]{SW_super2D_simul_PERT_FV_Nx450_Ny450_T0_4_order_1.pdf}
	\end{minipage}
	\caption{\RII{Shallow water system: perturbation of 2D supercritical equilibria. Difference between the solution and the equilibrium for different methods on the mesh $N_x=150,\,N_y=50$ and reference solution with FV on mesh $N_x=450,\,N_y=450$}}\label{fig:sw_eq_pert}
\end{figure}
\RII{In figure~\ref{fig:sw_eq_pert}, we show a perturbation of the numerical equilibrium obtained above. We add a small drop of water shaped as 
$$
\delta h(x,y) = 10^{-4}e^{-\frac{(x-16)^2 + (y-3)^2}{0.8^2}}
$$
on the $h$ variable and we continue the simulation until $t_f=0.4$. We plot the difference between each numerical equilibrium and solutions for all the methods on a coarse grid $N_x=150,\,N_y=50$, while we run a finer test for the FV-1 test case on the grid $N_x=450,\,N_y=450$. 
Looking at the simulation of FV-1 on the coarse mesh, we barely see the perturbation as the method is too dissipative. For the FV-2 simulations, the perturbation is visible, but small oscillations are present due to the fact that the method does not reach proper time convergence (the time residual stays around $10^{-6}$).
On the other hand, the GF nicely shows the perturbation moving towards the right without any spurious oscillations.}

As a second problem, and in order to test the robustness of the method, we also consider a crooked supercritical equilibrium 
with the same bathymetry but different initial conditions: 
\begin{alignat*}{3}
&h(x,y,0) = 2 - b(x,y),\qquad &&q_x(x,y,0) = 24, \qquad &&q_y(x,y,0) = 4\pi .
\end{alignat*}
In this case, left and bottom boundaries are inlet boundaries, while right and top are outlets.
It should be noticed, that this test case is even more challenging than the one shown before since no part of the fluid is aligned with the background Cartesian mesh.
In figure \ref{fig:SW_super_storto}, the results obtained for the conservative variables are presented,
where the same conclusion about the quality of the result of the GF method can be drawn as for the previous test. 
All the physical features of the equilibria are well captured, while they
are significantly more diffused by the FV-1 and FV-2 methods. 
\begin{figure}

  \centering
  \begin{minipage}{0.42\textwidth}
  	\centering $h$, FV-1\\
  	\includegraphics[width=\textwidth, height=0.31\textwidth]{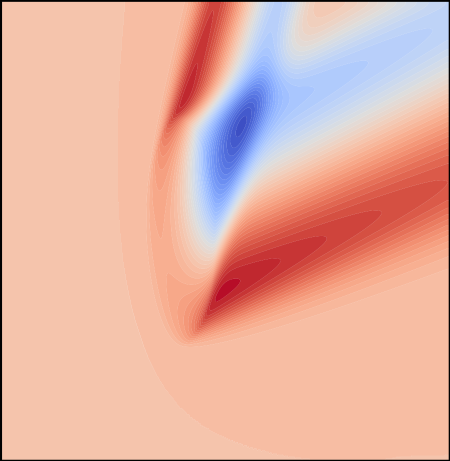}
  	\centering $h$, FV-2\\
  	\includegraphics[width=\textwidth, height=0.31\textwidth]{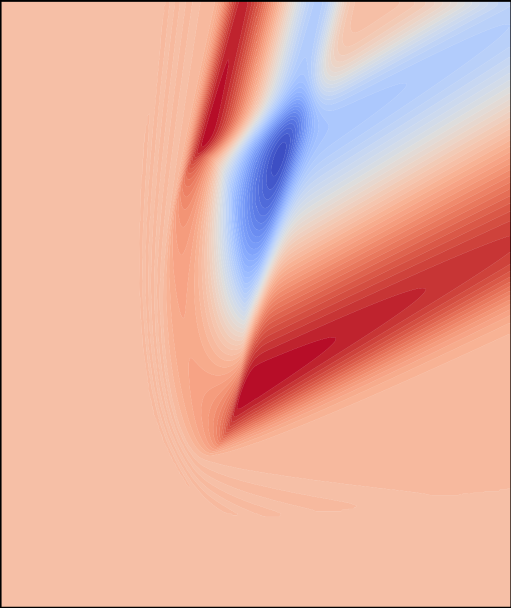}
  	\centering $h$, GF\\
  	\includegraphics[width=\textwidth, height=0.31\textwidth]{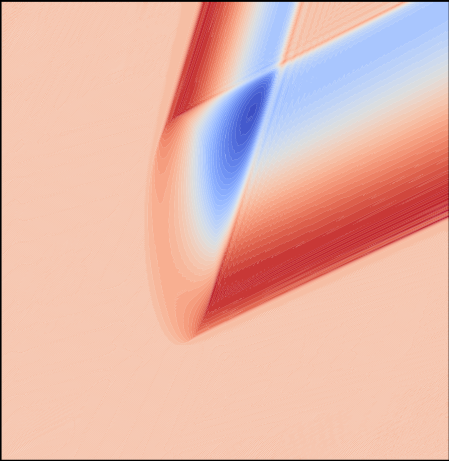}
  \end{minipage}\,
  \begin{minipage}{0.05\textwidth}
  	\includegraphics[width=\textwidth]{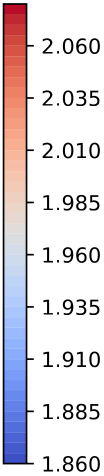}
  \end{minipage}\hfill
  \begin{minipage}{0.42\textwidth}
  	\centering $hv$, FV-1\\
  	\includegraphics[width=\textwidth, height=0.31\textwidth]{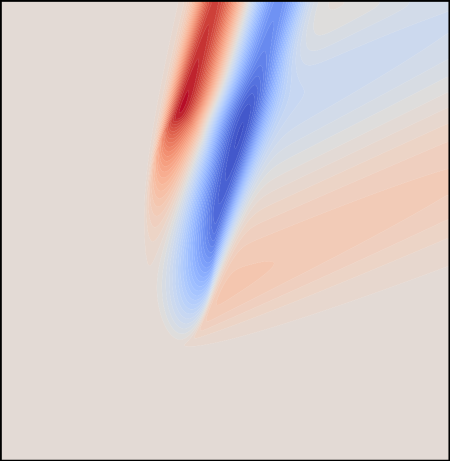}
  	\centering $hv$, FV-2\\
  	\includegraphics[width=\textwidth, height=0.31\textwidth]{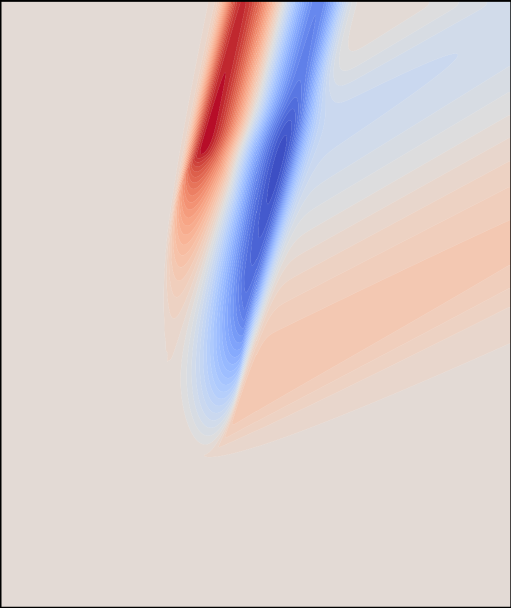}
  	\centering $hv$,   GF\\
  	\includegraphics[width=\textwidth, height=0.31\textwidth]{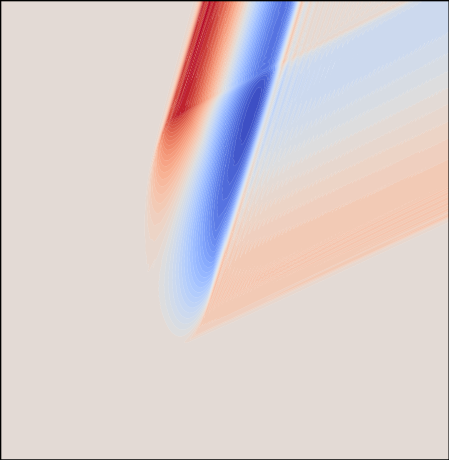}
  \end{minipage}\,
  \begin{minipage}{0.05\textwidth}
  	\includegraphics[width=\textwidth]{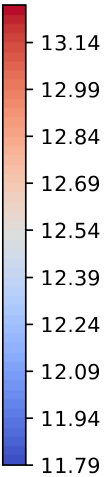}
  \end{minipage}
  \caption{Shallow water system: 2D crooked supercritical equilibria. Numerical results  obtained with FV-1, FV-2 and GF to steady state for $N_x=N_y=450$. }
  \label{fig:SW_super_storto}
\end{figure}

\section{Conclusions and perspectives}\label{sec:conclusions}

In this work, we have presented a new way to derive finite volume methods for nonlinear multi-dimensional hyperbolic systems, which is based on the global flux approach \eqref{eq:conservationlawglobalflux}, introduced in \cite{barsukow2025structure}.
It is a general way to obtain stationarity preserving schemes for nonlinear problems.
Besides its generality,
the method is also able to achieve super-convergence on steady problems, \RI{with error reductions of one or two orders of magnitude compared to
a standard second order finite volume approach.
Despite a focus on stationary states during its design, we observe  
remarkably high resolution   for unsteady multi-dimensional problems, outperforming standard first and even second-order finite volume methods,
for a large span  of  Mach/Froude  numbers.}

This work opens the way to several future developments. 
In particular, the extension of the finite volume formulation to high order methods by using
high-degree polynomial reconstruction techniques like WENO \cite{jiang1999high} is a natural next step,
following the work on the 1D global flux WENO approach introduced in \cite{ciallella2023arbitrary}.
Moreover, the first order finite volume method can also be seen as the starting point to develop
a new family of multi-dimensional high order discontinuous Galerkin methods based on the 
global flux formulation.
More investigations will also be dedicated to the extension of the method to deal with 
mathematical models characterized by curl-free solutions, like the Maxwell equations.  
Extending, for instance, the observation that stationarity preserving methods are also low Mach number compliant,
theoretical work will include further analysis of the method in unsteady situations.

\bibliographystyle{plain}
\bibliography{literature}

\end{document}